\documentclass{amsart}
\usepackage{amssymb,latexsym,amsmath,amscd,graphicx,graphics,epic,eepic,bm,color}
\usepackage{enumerate}
\usepackage[all,knot,poly]{xy}

\newcommand{\nat}{\mathbb{N}}
\newcommand{\zed}{\mathbb{Z}}

\newcommand{\C}{\mathbb{C}}

\newcommand{\Hom}{\mathrm{Hom}}

\newcommand{\Sym}{\mathrm{Sym}}
\newcommand{\qb}[2]{\genfrac{[}{]}{0pt}{}{#1}{#2}}
\newcommand{\ve}{\varepsilon}

\newcommand{\id}{\mathrm{id}}

\newcommand{\gdim}{\mathrm{gdim}}
\newcommand{\mf}{\mathrm{mf}}
\newcommand{\MF}{\mathrm{MF}}
\newcommand{\hmf}{\mathrm{hmf}}
\newcommand{\HMF}{\mathrm{HMF}}
\newcommand{\ch}{\mathsf{Ch}^{\mathsf{b}}}
\newcommand{\hch}{\mathsf{hCh}^{\mathsf{b}}}
\newcommand{\Tr}{\mathrm{Tr}}

\newcommand{\tc}{\mathrm{tc}}

\theoremstyle{plain}
\newtheorem{theorem}{Theorem}[section]
\newtheorem{lemma}[theorem]{Lemma}
\newtheorem{proposition}[theorem]{Proposition}
\newtheorem{corollary}[theorem]{Corollary}

\theoremstyle{definition}
\newtheorem{definition}[theorem]{Definition}

\theoremstyle{remark}
\newtheorem{remark}[theorem]{Remark}

\numberwithin{equation}{section}

\begin{document}

\title{Colored $\mathfrak{sl}(N)$ link homology via matrix factorizations}

\author{Hao Wu}

\address{Department of Mathematics, The George Washington University, Monroe Hall, Room 240, 2115 G Street, NW, Washington DC 20052}

\email{haowu@gwu.edu}



\begin{abstract}
The Reshetikhin-Turaev $\mathfrak{sl}(N)$ polynomial of links colored by wedge powers of the defining representation has been categorified via several different approaches. Here, we give a concise introduction to the categorification using matrix factorizations, which is a direct generalization of the Khovanov-Rozansky homology. Full details of the construction are given in \cite{Wu-color}. We also briefly review deformations and applications of this categorification given in \cite{Wu-color-equi,Wu-color-ras,Wu-color-MFW}.
\end{abstract}

\maketitle

\tableofcontents

\section{Introduction}\label{sec-intro}

\subsection{Background} In the early 1980s, Jones \cite{Jones} defined the Jones polynomial, which was generalized to the HOMFLY-PT polynomial in \cite{HOMFLY,PT}. Later, Reshetikhin and Turaev \cite{Resh-Tur1} constructed a large family of polynomial invariants for framed links whose components are colored by finite dimensional representations of a complex semisimple Lie algebra, of which the HOMFLY-PT polynomial is the special example corresponding to the defining representation of $\mathfrak{sl}(N;\C)$. 

In general, the Reshetikhin-Turaev invariants for links are abstract and hard to evaluate. But, when the Lie algebra is $\mathfrak{sl}(N;\C)$ and every component of the link is colored by a wedge power of the defining representation, Murakami, Ohtsuki and Yamada \cite{MOY} constructed a state sum for the corresponding $\mathfrak{sl}(N)$ quantum invariant. Their construction also comes with a set of graphical recursive relations, which we call the MOY calculus.

If every component of the link is colored by the defining representation, then the construction in \cite{MOY} recovers the uncolored $\mathfrak{sl}(N)$ HOMFLY-PT polynomial. Based on this, Khovanov and Rozansky \cite{KR1} categorified the uncolored $\mathfrak{sl}(N)$ HOMFLY-PT polynomial, which generalizes the Khovanov homology \cite{K1}. Their construction is based on matrix factorizations associated to $1,2$-colored MOY graphs.

\subsection{Some conventions}\label{subsec-conventions} Throughout this paper, $N$ is a fixed positive integer. 

All links and tangles in this paper are oriented and colored. That is, every component of the link or tangle is assigned an orientation and an element of $\{0,1,\dots,N\}$, which we call the color\footnote{In this paper, instead of saying that an object is colored by the $k$-fold wedge power of the defining representation of $\mathfrak{sl}(N;\C)$, we simply say that it is colored by $k$.} of this component. A link that is completely colored by $1$ is called uncolored.

Following the convention in \cite{KR1}, the degree of a polynomial in this paper is twice its usual degree.

\subsection{The colored $\mathfrak{sl}(N)$ link homology}\label{subsec-main-results} Khovanov and Rozansky's construction in \cite {KR1} was generalized in \cite{Wu-color} to categorify the Reshetikhin-Turaev $\mathfrak{sl}(N)$ polynomial of links colored by any wedge powers of the defining representation. The goal of the present paper is to give a concise introduction to this generalization. In particular, we will sketch proofs of Theorems \ref{main} and \ref{euler-char-main} below.

\begin{theorem}\cite[Theorem 1.1]{Wu-color}\label{main}
Let $D$ be a diagram of a tangle whose components are colored by elements of $\{0,1,\dots,N\}$, and $C(D)$ the chain complex defined in Definition \ref{complex-knotted-MOY-def}. Then the following are true:

\begin{enumerate}[(i)]
	\item $C(D)$ is a bounded chain complex over the homotopy category of graded matrix factorizations.
	\item $C(D)$ is $\zed_2\oplus\zed\oplus\zed$-graded, where the $\zed_2$-grading is the $\zed_2$-grading of the underlying matrix factorizations, the first $\zed$-grading is the quantum grading of the underlying matrix factorizations, and the second $\zed$-grading is the homological grading.
	\item The homotopy type of $C(D)$, with the $\zed_2\oplus\zed\oplus\zed$-grading, is invariant under Reidemeister moves. 
	\item If every component of $D$ is colored by $1$, then $C(D)$ is isomorphic to the chain complex defined by Khovanov and Rozansky in \cite{KR1}.
\end{enumerate}
\end{theorem}

Since the homotopy category $\hmf_{R,w}$ of graded matrix factorizations is not abelian, we can not directly define the homology of $C(D)$. But, as in \cite{KR1}, we can still construct a homology $H(D)$ from $C(D)$. Recall that each matrix factorization comes with a differential map $d_{mf}$. If $D$ is a link diagram, then the base ring $R$ is $\C$, and the potential $w=0$. So all the matrix factorizations in $C(D)$ are actually cyclic chain complexes. Taking the homology with respect to $d_{mf}$, we change $C(D)$ into a chain complex $(H(C(D),d_{mf}), d^\ast)$ of finite dimensional graded vector spaces, where $d^\ast$ is the differential map induced by the differential map $d$ of $C(D)$. We define 
\[
H(D)= H(H(C(D),d_{mf}), d^\ast). 
\]
If $D$ is a diagram of a tangle with end points, then $R$ is a graded polynomial ring with homogeneous indeterminates of positive gradings, and $w$ is in the maximal homogeneous ideal $\mathfrak{I}$ of $R$ generated by all the indeterminates. So $(C(D)/\mathfrak{I}\cdot C(D), d_{mf})$ is a cyclic chain complex. Its homology $(H(C(D)/\mathfrak{I}\cdot C(D), d_{mf}),d^\ast)$ is a chain complex of finite dimensional graded vector spaces, where $d^\ast$ is the differential map induced by the differential map $d$ of $C(D)$. We define
\[
H(D)= H(H(C(D)/\mathfrak{I}\cdot C(D),d_{mf}), d^\ast).
\] 
In either case, $H(D)$ inherits the $\zed_2\oplus\zed\oplus\zed$-grading of $C(D)$. We call $H(D)$ the colored $\mathfrak{sl}(N)$ homology of $D$. The corollary below follows easily from Theorem \ref{main}.

\begin{corollary}\label{homology-inv-main}
Let $D$ be a diagram of a tangle whose components are colored by elements of $\{0,1,\dots,N\}$. Then $H(D)$ is a finite dimensional $\zed_2\oplus\zed\oplus\zed$-graded vector space over $\C$. Reidemeister moves of $D$ induce isomorphisms of $H(D)$ preserving the $\zed_2\oplus\zed\oplus\zed$-grading.
\end{corollary}

For a tangle $T$, denote by $H^{\ve,i,j}(T)$ the subspace of $H(T)$ of homogeneous elements of $\zed_2$-degree $\ve$, quantum degree $i$ and homological degree $j$. The Poincar\'e polynomial $\mathrm{P}_T (\tau, q, t)$ of $H(T)$ is defined to be 
\[
\mathrm{P}_T (\tau, q, t) = \sum_{\ve,i,j} \tau^\ve q^i t^j \dim H^{\ve,i,j}(T) ~\in \C[\tau,q,t]/(\tau^2-1).
\]

Based on the construction by Murakami, Ohtsuki and Yamada \cite{MOY}, we give in Definition \ref{MOY-poly-def} a re-normalization $\mathrm{RT}_L(q)$ of the Reshetikhin-Turaev $\mathfrak{sl}(N)$ polynomial for links colored by non-negative integers. For a link $L$ colored by non-negative integers, the graded Euler characteristic of $H(L)$ is equal to $\mathrm{RT}_L(q)$. More precisely, we have the following theorem.

\begin{theorem}\cite[Theorem 1.3]{Wu-color}\label{euler-char-main}
Let $L$ be a link colored by non-negative integers. Then 
\[
\mathrm{P}_L (1, q, -1) = \mathrm{RT}_L(q).
\]

Moreover, define the total color $\tc(L)$ of $L$ to be the sum of the colors of the components of $L$. Then $H^{\ve,i,j}(L) =0$ if $\ve-\tc(L)= 1 \in \zed_2$. In particular,
\[
\mathrm{P}_L (\tau, q, t) = \tau^{\tc(L)} \sum_{i,j}  q^i t^j \dim H^{\tc(L),i,j}(L).
\]
\end{theorem}

\subsection{Deformations and applications} The construction of the colored $\mathfrak{sl}(N)$ link homology $H$ is based on matrix factorizations associated to MOY graphs with potentials induced by $X^{N+1}$. One can modify this construction by considering matrix factorizations with potentials induced by 
\[
f(X)=X^{N+1} + \sum_{k=1}^N (-1)^{k}\frac{N+1}{N+1-k}B_{k} X^{N+1-k},
\]
where $B_{k}$ is a homogeneous indeterminate of degree $2k$, and get an equivariant $\mathfrak{sl}(N)$ link homology $H_f$. $H_f$ is a finitely generated $\zed_2\oplus\zed\oplus\zed$-graded $\C[B_1,\dots,B_N]$-module. The construction of $H_f$ and the proof of its invariance are given in \cite{Wu-color-equi}, which generalizes the work of Krasner \cite{Krasner} in the uncolored case.

For any $b_1,\dots,b_N \in \C$, one can perform the above construction using matrix factorizations associated to MOY graphs with potentials induced by
\[
P(X)=X^{N+1} + \sum_{k=1}^N (-1)^{k}\frac{N+1}{N+1-k}b_{k} X^{N+1-k},
\]
which gives a deformed $\mathfrak{sl}(N)$ link homology $H_P$. For any link $L$, $H_P(L)$ is a finitely dimensional $\zed_2\oplus\zed$-graded and $\zed$-filtered vector space over $\C$. The quotient map 
\[
\pi:\C[B_1,\dots,B_N] \rightarrow \C~(\cong \C[B_1,\dots,B_N]/(B_1-b_1,\dots,B_N-b_N))
\] 
given by $\pi(B_k)=b_k$ induces a functor $\varpi$ of between categories of matrix factorizations. Using this functor, one can easily show that the invariance of the equivariant $\mathfrak{sl}(N)$ link homology $H_f$ implies the invariance of the deformed $\mathfrak{sl}(N)$ link homology $H_P$. As in the uncolored case, the filtration of $H_P$ induces a spectral sequence converging to $H_P$ with $E_1$-page isomorphic to the undeformed $\mathfrak{sl}(N)$ link homology $H$. Proofs of these results can be found in \cite{Wu-color-equi}.

When $P(X)$ is generic, that is, when $P'(X)=(N+1)(X^N+\sum_{k=1}^N (-1)^{k}b_{k} X^{N-k})$ has $N$ distinct root in $\C$, $H_P(L)$ has a basis that generalizes the basis given by Lee \cite{Lee2} and Gornik \cite{Gornik}. See \cite{Wu-color-ras} for the construction. \cite{Wu-color-ras} also contains the definition of the colored $\mathfrak{sl}(N)$ Rasmussen invariants and the bounds for slice genus and self linking number given by these invariants.

The undeformed $\mathfrak{sl}(N)$ link homology $H$ also gives new bounds for the self linking number and the braid index. (See \cite{Wu-color-MFW}.) These bounds generalize the well known Morton-Franks-Williams inequality \cite{FW,Mo}.

\subsection{Other approaches to the colored $\mathfrak{sl}(N)$ link homology}  The Reshetikhin-Turaev $\mathfrak{sl}(N)$ polynomial of links colored by wedge powers of the defining representation has been categorified via several different approaches. Next we quickly review some recent results in this direction.

Using matrix factorizations, Yonezawa \cite{Yonezawa3} defined essentially the Poincar\'e polynomial $\mathrm{P}_T$ of the colored $\mathfrak{sl}(N)$ link homology $H$.

Stroppel \cite{Stroppel} gave a Lie-theoretic construction of the Khovanov homology, which is proved in \cite{Brundan-Stroppel} to be isomorphic to Khovanov's original construction. (See also \cite[Section 5]{Stroppel2}.) Mazorchuk and Stroppel \cite{Mazorchuk-Stroppel} described a Koszul dual construction for $\mathfrak{sl}(N)$ link homology.

Mackaay, Stosic and Vaz \cite{Mackaay-Stosic-Vaz2} constructed a $\zed^{\oplus 3}$-graded HOMFLY-PT homology for $1,2$-colored links, which generalizes Khovanov and Rozansky's construction in \cite{KR2}. Webster and Williamson \cite{Webster-Williamson} further generalized this homology to links colored by any non-negative integers using the equivariant cohomology of general linear groups and related spaces.

Cautis and Kamnitzer \cite{Cautis-Kamnitzer-1,Cautis-Kamnitzer-2} constructed a link homology using the derived category of coherent sheaves on certain flag-like varieties. Their homology is conjectured to be isomorphic to the $\mathfrak{sl}(N)$ Khovanov-Rozansky homology in \cite{KR1}. Using $\mathfrak{sl}(2)$ actions on certain categories of D-modules and coherent sheaves, they \cite{Cautis-talk} also categorified the $\mathfrak{sl}(N)$ polynomial for links in $S^3$ colored by wedge powers of the defining representation.

Using categorifications of the tensor products of integrable representations of Kac-Moody algebras and quantum groups, Webster \cite{Webster1,Webster2} categorified, for any simple complex Lie algebra $\mathfrak{g}$, the quantum $\mathfrak{g}$ invariant for links colored by any finite dimensional representations of $\mathfrak{g}$. All the known categorifications of the colored Reshetikhin-Turaev $\mathfrak{sl}(N)$ polynomial are expected to agree with Webster's categorification for $\mathfrak{g}=\mathfrak{sl}(N;\C)$.

\section{The MOY Construction of the Reshetikhin-Turaev $\mathfrak{sl}(N)$ Polynomial}\label{sec-MOY-polynomial}

In \cite{MOY}, Murakami, Ohtsuki and Yamada gave an alternative construction of the Reshetikhin-Turaev $\mathfrak{sl}(N)$ polynomial for links colored by non-negative integers. The construction of our colored $\mathfrak{sl}(N)$ link homology $H$ is modeled on their construction. So we review their construction in this section. The notations and normalizations we use here are slightly different from those used in \cite{MOY}. 

\subsection{MOY graphs} 
\begin{definition}\label{MOY-graph-def}
An abstract MOY graph is an oriented graph with each edge colored by a non-negative integer such that, for every vertex $v$ with valence at least $2$, the sum of the colors of the edges entering $v$ is equal to the sum of the colors of the edges leaving $v$. We call this common sum the width of $v$.

A vertex of valence $1$ in an abstract MOY graph is called an end point. A vertex of valence greater than $1$ is called an internal vertex. An abstract MOY graph $\Gamma$ is said to be closed if it has no end points. We say that an abstract MOY graph is trivalent is all of its internal vertices have valence $3$.

An embedded MOY graph, or simply a MOY graph, is an embedding of an abstract MOY graph into $\mathbb{R}^2$ such that, through each vertex $v$, there is a straight line $L_v$ so that all the edges entering $v$ enter through one side of $L_v$ and all edges leaving $v$ leave through the other side of $L_v$.
\end{definition}

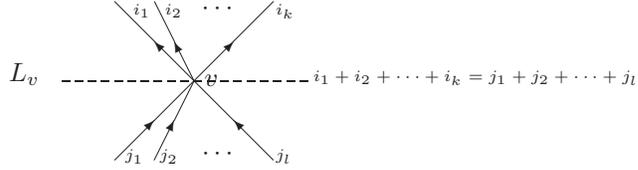
\begin{figure}[ht]

\setlength{\unitlength}{1pt}

\begin{picture}(360,80)(-180,-40)


\put(0,0){\vector(-1,1){15}}

\put(-15,15){\line(-1,1){15}}

\put(-23,25){\tiny{$i_1$}}

\put(0,0){\vector(-1,2){7.5}}

\put(-7.5,15){\line(-1,2){7.5}}

\put(-11,25){\tiny{$i_2$}}

\put(3,25){$\cdots$}

\put(0,0){\vector(1,1){15}}

\put(15,15){\line(1,1){15}}

\put(31,25){\tiny{$i_k$}}


\put(4,-2){$v$}

\multiput(-50,0)(5,0){19}{\line(1,0){3}}

\put(-70,0){$L_v$}

\put(45,0){\tiny{$i_1+i_2+\cdots +i_k = j_1+j_2+\cdots +j_l$}}


\put(-30,-30){\vector(1,1){15}}

\put(-15,-15){\line(1,1){15}}

\put(-26,-30){\tiny{$j_1$}}

\put(-15,-30){\vector(1,2){7.5}}

\put(-7.5,-15){\line(1,2){7.5}}

\put(-13,-30){\tiny{$j_2$}}

\put(3,-30){$\cdots$}

\put(30,-30){\vector(-1,1){15}}

\put(15,-15){\line(-1,1){15}}

\put(31,-30){\tiny{$j_l$}}

\end{picture}

\caption{An internal vertex of a MOY graph}\label{general-MOY-vertex-poly-figure}

\end{figure}

\subsection{The MOY graph polynomial} To each closed trivalent MOY graph, Murakami, Ohtsuki and Yamada \cite{MOY} associated a polynomial, which we call the MOY graph polynomial. They express the colored Reshetikhin-Turaev $\mathfrak{sl}(N)$ polynomial as a combination of MOY graph polynomials. We review the MOY graph polynomial in this subsection.

Define $\mathcal{N}=\{-N+1, -N+3,\cdots, N-3, N-1\}$ and $\mathcal{P(N)}$ to be the set of subsets of $\mathcal{N}$. For a finite set $A$, denote by $\#A$ the cardinality of $A$. Define a function $\pi:\mathcal{P(N)} \times \mathcal{P(N)} \rightarrow \zed_{\geq 0}$ by
\[
\pi (A_1, A_2) = \# \{(a_1,a_2) \in A_1 \times A_2 ~|~ a_1>a_2\} \text{ for } A_1,~A_2 \in \mathcal{P(N)}.
\]

\begin{figure}[ht]
$
\xymatrix{
\input{tri-vertex-s} & \text{or} & \input{tri-vertex-m}
} 
$
\caption{}\label{tri-vertex} 

\end{figure}

Let $\Gamma$ be a closed trivalent MOY graph, and $E(\Gamma)$ the set of edges of $\Gamma$. Denote by $\mathrm{c}:E(\Gamma) \rightarrow \nat$ the color function of $\Gamma$. That is, for every edge $e$ of $\Gamma$, $\mathrm{c}(e) \in \nat$ is the color of $e$. A state of $\Gamma$ is a function $\sigma: E(\Gamma) \rightarrow \mathcal{P(N)}$ such that
\begin{enumerate}[(i)]
	\item For every edge $e$ of $\Gamma$, $\#\sigma(e) = \mathrm{c}(e)$.
	\item For every vertex $v$ of $\Gamma$, as depicted in Figure \ref{tri-vertex}, we have $\sigma(e)=\sigma(e_1) \cup \sigma(e_2)$. (In particular, this implies that $\sigma(e_1) \cap \sigma(e_2)=\emptyset$.)
\end{enumerate}

For a state $\sigma$ of $\Gamma$ and a vertex $v$ of $\Gamma$ (as depicted in Figure \ref{tri-vertex}), the weight of $v$ with respect to $\sigma$ is defined to be 
\[
\mathrm{wt}(v;\sigma) = q^{\frac{\mathrm{c}(e_1)\mathrm{c}(e_2)}{2} - \pi(\sigma(e_1),\sigma(e_2))}.
\]

Given a state $\sigma$ of $\Gamma$, replace each edge $e$ of $\Gamma$ by $\mathrm{c}(e)$ parallel edges, assign to each of these new edges a different element of $\sigma(e)$ and, at every vertex, connect each pair of new edges assigned the same element of $\mathcal{N}$. This changes $\Gamma$ into a collection $\mathcal{C}$ of embedded oriented circles, each of which is assigned an element of $\mathcal{N}$. By abusing notation, we denote by $\sigma(C)$ the element of $\mathcal{N}$ assigned to $C\in \mathcal{C}$. Note that: 
\begin{itemize}
	\item There may be intersections between different circles in $\mathcal{C}$. But, each circle in $\mathcal{C}$ is embedded, that is, without self-intersections or self-tangency.
	\item There may be more than one way to do this. But if we view $\mathcal{C}$ as a virtue link and the intersection points between different elements of $\mathcal{C}$ as virtual crossings, then the above construction is unique up to purely virtual regular Reidemeister moves.
\end{itemize}
For each $C\in \mathcal{C}$, define the rotation number $\mathrm{rot}(C)$ the usual way. That is,
\[
\mathrm{rot}(C) = 
\begin{cases}
1 & \text{if } C \text{ is counterclockwise,} \\
-1 & \text{if } C \text{ is clockwise.}
\end{cases}
\]
The rotation number $\mathrm{rot}(\sigma)$ of $\sigma$ is then defined to be
\[
\mathrm{rot}(\sigma) = \sum_{C\in \mathcal{C}} \sigma(C) \mathrm{rot}(C).
\]

The $\mathfrak{sl}(N)$ MOY polynomial of $\Gamma$ is defined to be
\begin{equation}\label{MOY-bracket-def}
\left\langle \Gamma \right\rangle_N := \sum_{\sigma} (\prod_v \mathrm{wt}(v;\sigma)) q^{\mathrm{rot}(\sigma)},
\end{equation}
where $\sigma$ runs through all states of $\Gamma$ and $v$ runs through all vertices of $\Gamma$.

Murakami, Ohtsuki and Yamada \cite{MOY} established a set of graphical recursive relations for the $\mathfrak{sl}(N)$ MOY polynomial, which we call the MOY calculus. The MOY calculus plays an important role in guiding us through the construction of the colored $\mathfrak{sl}(N)$ homology. 

Before stating the MOY calculus, we need to introduce our normalization of quantum integers.

\begin{definition}\label{def-quantum-integers}
Quantum integers are elements of $\zed[q,q^{-1}]$. In this paper, we use the normalization
\begin{eqnarray*}
[j] & := & \frac{q^j-q^{-j}}{q-q^{-1}}, \\ 
{[j]}! & := & [1] \cdot [2] \cdots [j], \\
\qb{j}{k} & := & \frac{[j]!}{[k]!\cdot [j-k]!}.
\end{eqnarray*}
\end{definition}

The following is the MOY calculus.

\begin{theorem}\cite{MOY}\label{MOY-poly-skein}
The $\mathfrak{sl}(N)$ MOY graph polynomial $\left\langle \ast\right\rangle_N$ for close trivalent MOY graphs satisfies:
\begin{enumerate}
  \item $\left\langle \bigcirc_m \right\rangle_N = \qb{N}{m}$, where $\bigcirc_m$ is a circle colored by $m$.
  \item $\left\langle \setlength{\unitlength}{1pt}
\begin{picture}(50,50)(-80,20)

\put(-60,10){\vector(0,1){10}}

\put(-60,20){\vector(-1,1){20}}

\put(-60,20){\vector(1,1){10}}

\put(-50,30){\vector(-1,1){10}}

\put(-50,30){\vector(1,1){10}}

\put(-75,3){\tiny{$i+j+k$}}

\put(-55,21){\tiny{$j+k$}}

\put(-80,42){\tiny{$i$}}

\put(-60,42){\tiny{$j$}}

\put(-40,42){\tiny{$k$}}

\end{picture} \right\rangle_N = \left\langle \setlength{\unitlength}{1pt}
\begin{picture}(50,50)(40,20)

\put(60,10){\vector(0,1){10}}

\put(60,20){\vector(1,1){20}}

\put(60,20){\vector(-1,1){10}}

\put(50,30){\vector(1,1){10}}

\put(50,30){\vector(-1,1){10}}

\put(45,3){\tiny{$i+j+k$}}

\put(38,21){\tiny{$i+j$}}

\put(80,42){\tiny{$k$}}

\put(60,42){\tiny{$j$}}

\put(40,42){\tiny{$i$}}

\end{picture} \right\rangle_N$.
	\item $\left\langle \input{v-vector-m+n-bubble-slide}\right\rangle_N = \qb{m+n}{n} \cdot\left\langle \setlength{\unitlength}{.75pt}
\begin{picture}(55,80)(-20,40)
\put(0,0){\vector(0,1){80}}
\put(5,75){\tiny{$_{m+n}$}}
\end{picture}\right\rangle_N $.
	\item $\left\langle \setlength{\unitlength}{.75pt}
\begin{picture}(60,80)(-30,40)
\put(0,0){\vector(0,1){30}}
\put(0,30){\vector(0,1){20}}
\put(0,50){\vector(0,1){30}}

\put(-1,40){\line(1,0){2}}

\qbezier(0,30)(25,20)(25,30)
\qbezier(0,50)(25,60)(25,50)
\put(25,50){\vector(0,-1){20}}

\put(5,75){\tiny{$_{m}$}}
\put(5,5){\tiny{$_{m}$}}
\put(-30,38){\tiny{$_{m+n}$}}
\put(14,60){\tiny{$_{n}$}}
\end{picture}\right\rangle_N = \qb{N-m}{n} \cdot \left\langle \setlength{\unitlength}{.75pt}
\begin{picture}(40,80)(-20,40)
\put(0,0){\vector(0,1){80}}
\put(5,75){\tiny{$_{m}$}}
\end{picture}\right\rangle_N$.
	\item $\left\langle \input{decomp-III-1-slide}\right\rangle_N = \left\langle \setlength{\unitlength}{.75pt}
\begin{picture}(60,60)(-30,30)

\put(-20,0){\vector(0,1){60}}

\put(20,60){\vector(0,-1){60}}

\put(-25,30){\tiny{$_1$}}

\put(22,30){\tiny{$_m$}}
\end{picture}\right\rangle_N + [N-m-1] \cdot \left\langle \setlength{\unitlength}{.75pt}
\begin{picture}(60,60)(100,30)

\put(110,0){\vector(1,1){20}}

\put(130,20){\vector(1,-1){20}}

\put(130,40){\vector(0,-1){20}}

\put(130,40){\vector(-1,1){20}}

\put(150,60){\vector(-1,-1){20}}

\put(105,0){\tiny{$_1$}}

\put(105,55){\tiny{$_1$}}

\put(152,0){\tiny{$_m$}}

\put(152,55){\tiny{$_m$}}

\put(132,30){\tiny{$_{m-1}$}}

\end{picture}\right\rangle_N$.
	\item $\left\langle \input{decomp-IV-1-slide}\right\rangle_N = \qb{m-1}{n} \cdot \left\langle \setlength{\unitlength}{.75pt}
\begin{picture}(85,90)(-30,45)

\put(-20,0){\vector(0,1){45}}

\put(-20,45){\vector(0,1){45}}

\put(20,0){\vector(0,1){45}}

\put(20,45){\vector(0,1){45}}

\put(20,45){\vector(-1,0){40}}

\put(-27,20){\tiny{$_1$}}

\put(23,20){\tiny{$_{m+l-1}$}}

\put(-27,65){\tiny{$_l$}}

\put(23,65){\tiny{$_m$}}

\put(-5,38){\tiny{$_{l-1}$}}

\end{picture}\right\rangle_N  + \qb{m-1}{n-1} \cdot \left\langle \setlength{\unitlength}{.75pt}
\begin{picture}(60,90)(110,45)

\put(110,0){\vector(2,3){20}}

\put(150,0){\vector(-2,3){20}}

\put(130,30){\vector(0,1){30}}

\put(130,60){\vector(-2,3){20}}

\put(130,60){\vector(2,3){20}}

\put(117,20){\tiny{$_1$}}

\put(140,20){\tiny{$_{m+l-1}$}}

\put(117,65){\tiny{$_l$}}

\put(140,65){\tiny{$_m$}}

\put(133,42){\tiny{$_{m+l}$}}
\end{picture} \right\rangle_N$.\vspace{.5cm}	
	\item $\left\langle \input{decomp-V-1-slide}\right\rangle_N = \sum_{j=\max\{m-n,0\}}^m \qb{l}{k-j} \cdot \left\langle \input{decomp-V-2-slide} \right\rangle_N$. \vspace{.5cm}
	\item The above equations remain true if we reverse the orientation of the MOY graph or the orientation of $\mathbb{R}^2$.
	\item The above properties uniquely determine the $\mathfrak{sl}(N)$ MOY graph polynomial $\left\langle \ast\right\rangle_N$.
\end{enumerate}
\end{theorem}

Parts (1)-(8) of Theorem \ref{MOY-poly-skein} are proved in \cite{MOY}. Part (9) is essentially proved in \cite[Proof of Theorem 14.7]{Wu-color}\footnote{Part (9) of Theorem \ref{MOY-poly-skein} was known to experts before I wrote \cite{Wu-color}. But I did not find a written proof of it back then.}.

\subsection{The colored Reshetikhin-Turaev $\mathfrak{sl}(N)$ polynomial}

\begin{definition}\cite{MOY}\label{MOY-poly-def}
For a link diagram $D$ colored by non-negative integers, define $\left\langle D \right\rangle_N$ by applying the following at every crossing of $D$.
\[
\left\langle \setlength{\unitlength}{1pt}
\begin{picture}(40,40)(-20,0)

\put(-20,-20){\vector(1,1){40}}

\put(20,-20){\line(-1,1){15}}

\put(-5,5){\vector(-1,1){15}}

\put(-11,15){\tiny{$_m$}}

\put(9,15){\tiny{$_n$}}

\end{picture} \right\rangle_N = \sum_{k=\max\{0,m-n\}}^{m} (-1)^{m-k} q^{k-m}\left\langle \input{square-m-n-k-left-poly}\right\rangle_N,
\]
\[
\left\langle \setlength{\unitlength}{1pt}
\begin{picture}(40,40)(-20,0)

\put(20,-20){\vector(-1,1){40}}

\put(-20,-20){\line(1,1){15}}

\put(5,5){\vector(1,1){15}}

\put(-11,15){\tiny{$_m$}}

\put(9,15){\tiny{$_n$}}

\end{picture} \right\rangle_N = \sum_{k=\max\{0,m-n\}}^{m} (-1)^{k-m} q^{m-k}\left\langle \input{square-m-n-k-left-poly}\right\rangle_N.
\]
\vspace{.5cm}

Also, for each crossing $c$ of $D$, define the shifting factor $\mathsf{s}(c)$ of $c$ by
\[
\mathsf{s}\left(\setlength{\unitlength}{1pt}
\begin{picture}(40,40)(-20,0)

\put(-20,-20){\vector(1,1){40}}

\put(20,-20){\line(-1,1){15}}

\put(-5,5){\vector(-1,1){15}}

\put(-11,15){\tiny{$_m$}}

\put(9,15){\tiny{$_n$}}

\end{picture}\right) = 
\begin{cases}
(-1)^{-m} q^{m(N+1-m)} & \text{if } m=n,\\
1 & \text{if } m \neq n,
\end{cases}
\]
\[
\mathsf{s}\left(\setlength{\unitlength}{1pt}
\begin{picture}(40,40)(-20,0)

\put(20,-20){\vector(-1,1){40}}

\put(-20,-20){\line(1,1){15}}

\put(5,5){\vector(1,1){15}}

\put(-11,15){\tiny{$_m$}}

\put(9,15){\tiny{$_n$}}

\end{picture}\right) = 
\begin{cases}
(-1)^m q^{-m(N+1-m)} & \text{if } m=n,\\
1 & \text{if } m \neq n.
\end{cases}
\]

The re-normalized Reshetikhin-Turaev $\mathfrak{sl}(N)$-polynomial $\mathrm{RT}_D(q)$ of $D$ is defined to be
\[
\mathrm{RT}_D(q) = \left\langle D \right\rangle_N \cdot \prod_c \mathsf{s}(c),
\]
where $c$ runs through all crossings of $D$.
\end{definition}

\begin{theorem}\cite{MOY}\label{MOY-poly-inv}
$\left\langle D \right\rangle_N$ is invariant under regular Reidemeister moves. $\mathrm{RT}_D(q)$ is invariant under all Reidemeister moves and is a re-normalization of the Reshetikhin-Turaev $\mathfrak{sl}(N)$ polynomial for links colored by non-negative integers.
\end{theorem}

\section{Graded Matrix Factorizations}\label{sec-concept-mf}

In the following two sections, we briefly review the core algebraic concepts, matrix factorizations and symmetric polynomials, that are used in our construction of the colored $\mathfrak{sl}(N)$ link homology. For more details, see \cite[Sections 2, 3, 4]{Wu-color}.

Recall that $N$ is a fixed positive integer. In this section, $R=\C[X_1,\dots,X_k]$, where $X_1,\dots,X_k$ are homogeneous indeterminates with positive degrees.

\subsection{Graded matrix factorizations} 

\begin{definition}\label{def-mf}
Let $w$ be a homogeneous element of $R$ of degree $2N+2$. A graded matrix factorization $M$ over $R$ with potential $w$ is a collection of two free graded $R$-modules $M_0$, $M_1$ and two homogeneous $R$-module maps $d_0:M_0\rightarrow M_1$, $d_1:M_1\rightarrow M_0$ of degree $N+1$, called differential maps, such that
\[
d_1 \circ d_0=w\cdot\id_{M_0}, \hspace{1cm}  d_0 \circ d_1=w\cdot\id_{M_1}.
\]
We usually write $M$ as
\[
M_0 \xrightarrow{d_0} M_1 \xrightarrow{d_1} M_0.
\]
$M$ has two gradings: a $\zed_2$-grading that takes value $\ve$ on $M_\ve$ and a quantum grading, which is the $\zed$-grading of the underlying graded $R$-module. We denote the degree from the $\zed_2$-grading by ``$\deg_{\zed_2}$" and the degree from the quantum grading by ``$\deg$".

Following \cite{KR1}, we denote by $M\left\langle 1\right\rangle$ the matrix factorization
\[
M_1 \xrightarrow{d_1} M_0 \xrightarrow{d_0} M_1,
\]
and write $M\left\langle j\right\rangle = M \underbrace{\left\langle 1\right\rangle\cdots\left\langle 1\right\rangle}_{j \text{ times }}$.

For graded matrix factorizations $M=M_0 \xrightarrow{d_0} M_1 \xrightarrow{d_1} M_0$ and $\tilde{M}=\tilde{M}_0 \xrightarrow{\tilde{d}_0} \tilde{M}_1 \xrightarrow{\tilde{d}_1} \tilde{M}_0$ with potential $w$, $M \oplus\tilde{M}$ is the matrix factorization
\[
M \oplus\tilde{M}= M_0 \oplus\tilde{M}_0 \xrightarrow{D_0}  M_1 \oplus\tilde{M}_1 \xrightarrow{D_1} M_0 \oplus\tilde{M}_0,
\]
where
\[
D_\ve = \left(%
\begin{array}{cc}
  d_\ve & 0 \\
  0 & \tilde{d}_\ve \\
\end{array}%
\right)  \text{ for } \ve=0,1.
\]
The potential of $M \oplus\tilde{M}$ is $w$.

For graded matrix factorizations $M=M_0 \xrightarrow{d_0} M_1 \xrightarrow{d_1} M_0$ with potential $w$ and $M'=M'_0 \xrightarrow{d'_0} M'_1 \xrightarrow{d'_1} M'_0$ with potential $w'$, the tensor product $M\otimes M'$ is the graded matrix factorization with 
\begin{eqnarray*}
(M\otimes M')_0 & = & (M_0\otimes M'_0)\oplus (M_1\otimes M'_1), \\
(M\otimes M')_1 & = & (M_1\otimes M'_0)\oplus (M_1\otimes M'_0),
\end{eqnarray*}
and the differential given by signed Leibniz rule. That is, for $m\in M_\ve$ and $m'\in M'$,
\[
d(m\otimes m')=(dm)\otimes m' + (-1)^\ve m \otimes (d'm').
\]
The potential of $M\otimes M'$ is $w+w'$.
\end{definition}

\begin{definition}\label{def-grading-shift-q}
Let $M$ be a graded $R$-module. Denote by $M^i$ the $\C$-linear subspace of $M$ of homogeneous elements of degree $i$. For any $j\in \zed$, define $M\{q^j\}$ to be $M$ with the grading shifted up by $j$. That is, the grading of $M\{q^j\}$ is defined by $M\{q^j\}^{(i)}=M^{(i-j)}$. 

More generally, for 
\[
F(q)=\sum_{j=k}^l a_j q^j ~\in \zed_{\geq0}[q,q^{-1}],
\]
we define $M\{F(q)\}$ to be graded $R$-module
\[
M\{F(q)\} = \bigoplus_{j=k}^l (\underbrace{M\{q^j\}\oplus\cdots\oplus M\{q^j\}}_{a_j-\text{fold}}).
\]

If $M$ is a graded matrix factorization over $R$, then $M\{q^j\}$ means shifting the quantum grading of $M$ up by $j$.
\end{definition}

\begin{definition}\label{koszul-mf-def}
If $a_0,a_1\in R$ are homogeneous elements with $\deg a_0 +\deg a_1=2N+2$, then denote by $(a_0,a_1)_R$ the matrix factorization $R \xrightarrow{a_0} R\{q^{N+1-\deg{a_0}}\} \xrightarrow{a_1} R$, which has potential $a_0a_1$. More generally, if $a_{1,0},a_{1,1},\dots,a_{k,0},a_{k,1}\in R$ are homogeneous with $\deg a_{j,0} +\deg a_{j,1}=2N+2$, denote by 
\[
\left(%
\begin{array}{cc}
  a_{1,0}, & a_{1,1} \\
  a_{2,0}, & a_{2,1} \\
  \dots & \dots \\
  a_{k,0}, & a_{k,1}
\end{array}%
\right)_R
\]
the tenser product 
\[
(a_{1,0},a_{1,1})_R \otimes_R (a_{2,0},a_{2,1})_R \otimes_R \cdots \otimes_R (a_{k,0},a_{k,1})_R,
\] 
which is a graded matrix factorization with potential $\sum_{j=1}^k a_{j,0} a_{j,1}$, and is called the Koszul matrix factorization associated to the above matrix. We drop``$R$" from the notation when it is clear from the context. Note that the above Koszul matrix factorization is finitely generated over $R$.
\end{definition}

\subsection{Morphisms of graded matrix factorizations} Given two graded matrix factorizations $M$ with potential $w$ and $M'$ with potential $w'$ over $R$, consider the $R$-module $\Hom_R(M,M')$. It admits a $\zed_2$-grading that takes value 
\[
\left\{%
\begin{array}{l}
    0 \text{ on } \Hom^0_R(M,M')=\Hom_R(M_0,M'_0)\oplus\Hom_R(M_1,M'_1), \\ 
    1 \text{ on } \Hom^1_R(M,M')=\Hom_R(M_1,M'_0)\oplus\Hom_R(M_0,M'_1). 
\end{array}%
\right.
\]
It also admits a quantum pregrading. (See \cite[Subsections 2.1, 2.2]{Wu-color}.) induced by the quantum gradings of homogeneous elements. Moreover, $\Hom_R(M,M')$ has a differential map $d$ given by
\[
d(f)=d_{M'} \circ f -(-1)^\ve f \circ d_M \text{ for } f \in \Hom^\ve_R(M,M').
\]
Note that $d$ is homogeneous of degree $N+1$ and satisfies that 
\[
d^2=(w'-w) \cdot \id_{\Hom_R(M,M')}.
\]
Following \cite{KR1}, we write $M_\bullet = \Hom_R(M,R)$.

In general, $\Hom_R(M,M')$ is not a graded matrix factorization since $\Hom_R(M,M')$ might not be a free module and the quantum pregrading might not be a grading. 

\begin{definition}\label{def-morph-mf}
Let $M$ and $M'$ be two graded matrix factorizations over $R$ with potential $w$. Then $\Hom_R(M,M')$, with the above differential map $d$, is a cyclic chain complex. 

We say that an $R$-module map $f:M\rightarrow M'$ is a morphism of matrix factorizations if and only if $df=0$. that is, for $f\in\Hom^\ve_R(M,M')$, $f$ is a morphism of matrix factorizations if and only if $d_{M'}\circ f = (-1)^\ve f\circ d_M$. $f$ is called an isomorphism of matrix factorizations if it is a morphism of matrix factorizations and an isomorphism of the underlying $R$-modules. Two morphisms $f,g:M\rightarrow M'$ of $\zed_2$-degree $\ve$ are homotopic if $f-g$ is a boundary element in $\Hom_R(M,M')$, that is, if $\exists ~h\in \Hom^{\ve+1}_R(M,M')$ such that $f-g = d(h) = d_{M'} \circ h -(-1)^{\ve+1} h \circ d_M$.

\begin{enumerate}[(i)]
	\item We say that $M,M'$ are isomorphic, or $M\cong M'$, if and only if there is a homogeneous isomorphism $f:M\rightarrow M'$ that preserves both gradings.
	\item We say that $M,M'$ are homotopic, or $M\simeq M'$, if and only if there are homogeneous morphisms $f:M\rightarrow M'$ and $g:M'\rightarrow M$ preserving both gradings such that $g\circ f \simeq \id_M$ and $f\circ g \simeq \id_{M'}$. 
\end{enumerate}
\end{definition}

\begin{lemma}\cite[Lemma 2.9]{Wu-color}\label{morphism-sign}
Let $M,~M',~\mathcal{M}, ~\mathcal{M}'$ be graded matrix factorizations over $R$ such that $M,\mathcal{M}$ have the same potential and $M',\mathcal{M}'$ have the same potential. Assume that $f:M \rightarrow \mathcal{M}$ and $f':M'\rightarrow \mathcal{M}'$ are morphisms of matrix factorizations of $\zed_2$-degrees $j$ and $j'$. Define $F: M\otimes M' \rightarrow \mathcal{M}\otimes \mathcal{M}'$ by $F(m\otimes m') = (-1)^{i\cdot j'} f(m)\otimes f'(m')$ for $m \in M_i$ and $m'\in M'$. Then $F$ is a morphism of matrix factorizations of $\zed_2$-degree $j+j'$. In particular, if $f$ or $f'$ is homotopic to $0$, then so is $F$.

From now on, we will write $F = f \otimes f'$.
\end{lemma}

\subsection{Categories of homotopically finite graded matrix factorizations}

\begin{definition}\label{homotopically-finite-def}
Let $M$ be a graded matrix factorization over $R$ with potential $w$. We say that $M$ is homotopically finite if there exists a finitely generated graded matrix factorization $\mathcal{M}$ over $R$ with potential $w$ such that $M\simeq \mathcal{M}$. 
\end{definition}

\begin{definition}\label{hom-all-gradings-def}
Let $M$ and $M'$ be any two graded matrix factorizations over $R$ with potential $w$. Denote by $d$ the differential map of $\Hom_R(M,M')$.

$\Hom_{\MF}(M,M')$ is defined to be the submodule of $\Hom_R(M,M')$ consisting of morphisms of matrix factorizations from $M$ to $M'$. Or, equivalently, $\Hom_{\MF}(M,M') := \ker d$.

$\Hom_{\mf}(M,M')$ is defined to be the $\C$-linear subspace of $\Hom_{\MF}(M,M')$ consisting of homogeneous morphisms with $\zed_2$-degree $0$ and quantum degree $0$.

$\Hom_{\HMF}(M,M')$ is defined to be the $R$-module of homotopy classes of morphisms of matrix factorizations from $M$ to $M'$. Or, equivalently,
$\Hom_{\HMF}(M,M')$ is the homology of the chain complex $(\Hom_R(M,M'),d)$.

$\Hom_{\hmf}(M,M')$ is defined to be the $\C$-linear subspace of $\Hom_{\HMF}(M,M')$ consisting of homogeneous elements with $\zed_2$-degree $0$ and quantum degree $0$.
\end{definition}

Now we introduce four categories of homotopically finite graded matrix factorizations relevant to our construction. We will be mainly concerned with the homotopy categories $\HMF_{R,w}$ and $\hmf_{R,w}$.

\begin{definition}\label{categories-def}
Let $w\in R$ be an homogeneous element of degree $2N+2$.  We define $\MF_{R,w}$, $\HMF_{R,w}$, $\mf_{R,w}$ and $\hmf_{R,w}$ by the following table. 

\begin{center}
\small{
\begin{tabular}{|c|c|c|}
\hline
Category & Objects & Morphisms \\
\hline
$\MF_{R,w}$ & all homotopically finite graded matrix factorizations over   & $\Hom_{\MF}$ \\
 &  $R$ of potential $w$ with quantum gradings bounded below &  \\
\hline
$\mf_{R,w}$ & all homotopically finite graded matrix factorizations over    & $\Hom_{\mf}$ \\
 & $R$ of potential $w$ with quantum gradings bounded below &  \\
\hline
$\HMF_{R,w}$ & all homotopically finite graded matrix factorizations over    & $\Hom_{\HMF}$ \\
 & $R$ of potential $w$ with quantum gradings bounded below &  \\
\hline
$\hmf_{R,w}$ & all homotopically finite graded matrix factorizations over    & $\Hom_{\hmf}$ \\
 & $R$ of potential $w$ with quantum gradings bounded below &  \\
\hline
\end{tabular}
}
\end{center}
\end{definition}

\begin{remark} 
\begin{enumerate}[(i)]
  \item The above categories are additive.
	\item The definitions of these categories here are slightly different from those in \cite{KR1}.
  \item The grading of a finitely generated graded matrix factorization over $R$ is bounded below. So finitely generated graded matrix factorizations are objects of the above categories.
    \item Comparing Definition \ref{categories-def} to Definition \ref{def-morph-mf}, one can see that, for any object $M$ and $M'$ of the above categories, $M \cong M'$ means they are isomorphic as objects of $\mf_{R,w}$, and $M \simeq M'$ means they are isomorphic as objects of $\hmf_{R,w}$.
\end{enumerate}
\end{remark}

\begin{lemma}\cite[Lemma 2.32]{Wu-color}\label{hom-all-gradings-homo-finite}
If $M$ and $M'$ are objects of $\HMF_{R,w}$, then the quantum pregrading of $\Hom_{\HMF}(M,M')$ is a grading.
\end{lemma}

\subsection{Homology of graded matrix factorizations over $R$}\label{homology-homotopy} Set $\mathfrak{I}=(X_1,\dots,X_k)$, the homogeneous ideal of $R$ generated by all the indeterminates of $R$. Let $w\in \mathfrak{I}$ be a homogeneous element of degree $2N+2$, and $M$ a graded matrix factorization over $R$ with potential $w$. Note that $M/\mathfrak{I} M$ is a chain complex over $\C$, and it inherits the gradings of $M$.

\begin{definition}\label{homology-matrix-factorization-def}
$H_R(M)$ is defined to be the homology of $M/\mathfrak{I} M$. It inherits the gradings of $M$. If $R$ is clear from the context, we drop it from the notations. 

Denote by $H_R^{\ve,j}(M)$ the subspace of $H_R(M)$ consisting of homogeneous elements of $\zed_2$-degree $\ve$ and quantum degree $j$. Following \cite{KR1}, we define the graded dimension of $M$ to be
\[
\gdim_R(M) = \sum_{j,\ve} \tau^{\ve} q^j \dim_\C H_R^{\ve,j}(M) ~\in ~\zed[[q]][\tau]/(\tau^2-1).
\]
Again, if $R$ is clear from the context, we drop it from the notations.
\end{definition}

\begin{proposition}\cite[Proposition 8]{KR1}\label{prop-homotopy-equal-isomorphism-on-homology}
Let $M$ and $M'$ be graded matrix factorizations over $R$ with homogeneous potential $w\in\mathfrak{I}$ of degree $2N+2$. Assume the quantum gradings of $M$ and $M'$ are bounded below. Suppose that $f:M\rightarrow M'$ is a homogeneous morphism preserving both gradings. Then $f$ is a homotopy equivalence if and only if it induces an isomorphism of the homology $f_\ast: H_R(M) \rightarrow H_R(M')$.
\end{proposition}

\begin{proposition}\cite[Proposition 7]{KR1}\label{prop-homology-detects-homotopy}
Let $M$ be a graded matrix factorization over $R$ with homogeneous potential $w\in\mathfrak{I}$ of degree $2N+2$. Assume the quantum grading of $M$ is bounded below. Then 
\begin{enumerate}[(i)]
	\item $M\simeq 0$ if and only if $H_R(M)=0$ or, equivalently, $\gdim_R (M)=0$;
	\item $M$ is homotopically finite if and only if $H_R(M)$ is finite dimensional over $\C$ or, equivalently, $\gdim_R (M) \in \zed[q,\tau]/(\tau^2-1)$.
\end{enumerate}
\end{proposition}

\subsection{Categories of chain complexes} Now we introduce our notations for categories of chain complexes.

\begin{definition}\label{categories-of-complexes}
Let $\mathcal{C}$ be an additive category. We denote by $\ch(\mathcal{C})$ the category of bounded chain complexes over $\mathcal{C}$. More precisely,
\begin{itemize}
	\item An object of $\ch(\mathcal{C})$ is a chain complex 
	\begin{equation}\label{chain-complex-form}
	\xymatrix{
	\cdots  \ar[r]^{d_{i-1}} & A_i \ar[r]^{d_{i}} & A_{i+1} \ar[r]^{d_{i+1}} & A_{i+2} \ar[r]^{d_{i+2}} & \cdots 
	}
	\end{equation}
	where $A_i$'s are objects of $\mathcal{C}$, $d_i$'s are morphisms of $\mathcal{C}$ such that $d_{i+1} \circ d_i =0$ for $i\in \zed$, and there exists integers $k\leq K$ such that $A_i =0$ if $i>K$ or $i<k$.
	\item A morphism $f$ of $\ch(\mathcal{C})$ is a commutative diagram
	\[
	\xymatrix{
	\cdots  \ar[r]^{d_{i-1}} & A_i \ar[r]^{d_{i}} \ar[d]_{f_{i}} & A_{i+1} \ar[r]^{d_{i+1}} \ar[d]_{f_{i+1}} & A_{i+2} \ar[r]^{d_{i+2}} \ar[d]_{f_{i+2}} & \cdots \\
		\cdots  \ar[r]^{d'_{i-1}} & A'_i \ar[r]^{d'_{i}} & A'_{i+1} \ar[r]^{d'_{i+1}} & A'_{i+2} \ar[r]^{d'_{i+2}}  & \cdots
	},
	\]
	where each row is an object of $\ch(\mathcal{C})$ and vertical arrows are morphisms of $\mathcal{C}$.
\end{itemize}
Chain homotopy in $\ch(\mathcal{C})$ is defined the usual way.

We denote by $\hch(\mathcal{C})$ the homotopy category of bounded chain complexes over $\mathcal{C}$, or simply the homotopy category of $\mathcal{C}$. $\hch(\mathcal{C})$ is defined by 
\begin{itemize}
	\item An object of $\hch(\mathcal{C})$ is an object of $\ch(\mathcal{C})$.
	\item For any two objects $A$ and $B$ of $\hch(\mathcal{C})$ $\Hom_{\hch(\mathcal{C})}(A,B)$ is $\Hom_{\ch(\mathcal{C})}(A,B)$ modulo the subgroup of null homotopic morphisms.
\end{itemize}

An isomorphism in $\ch(\mathcal{C})$ is denoted by ``$\cong$". An isomorphism in $\hch(\mathcal{C})$ is commonly known as a homotopy equivalence and denoted by ``$\simeq$". 

Let $A$ be the object of $\ch(\mathcal{C})$ (and $\hch(\mathcal{C})$) given in \eqref{chain-complex-form}. Then $A$ admits an obvious bounded homological grading $\deg_h$ with $\deg_h A_i =i$.  Morphisms of $\ch(\mathcal{C})$ and $\hch(\mathcal{C})$ preserve this grading. Denote by $A\| k \|$ the object of $\ch(\mathcal{C})$ obtained by shifting the homological grading of $A$ up by $k$. That is, $A\| k \|$ is the same chain complex as $A$ except that $\deg_h A_i =i+k$ in $A\| k \|$.
\end{definition}

\subsection{The Krull-Schmidt property}

\begin{definition}\cite{Drozd}\label{Krull-Schmidt-category-def}
An additive category $\mathcal{C}$ is called a $\C$-category if all morphism sets $\Hom_\mathcal{C}(A,B)$ are $\C$-linear spaces and the composition of morphisms is $\C$-bilinear.

A $\C$-category $\mathcal{C}$ is called Krull-Schmidt if 
\begin{itemize}
	\item every object of $\mathcal{C}$ is isomorphic to a finite direct sum $A_1\oplus\cdots\oplus A_n$ of indecomposable objects of $\mathcal{C}$;
	\item and, if $A_1\oplus\cdots\oplus A_n \cong A'_1\oplus\cdots\oplus A'_l$, where $A_1,\dots\,A_n, A'_1,\dots,A'_l$ are indecomposable objects of $\mathcal{C}$, then $n=l$ and there is a permutation $\sigma$ of $\{1,\dots,n\}$ such that $A_i \cong A'_{\sigma(i)}$ for $i=1,\dots,n$.
\end{itemize}
\end{definition}

One can prove that $\mathcal{C}$ and $\hch(\mathcal{C})$ are Krull-Schmidt by showing $\mathcal{C}$ is fully additive and locally finite dimensional. (See for example \cite{Drozd}.) This was done for $\hmf_{R,w}$ in \cite{KR1}.

\begin{lemma}\cite[Propositions 6 and 24]{KR1}\label{Krull-Schmidt-hmf}
Let $w$ be a homogeneous element of $R$ of degree $2N+2$. Then $\hmf_{R,w}$ and $\hch(\hmf_{R,w})$ are Krull-Schmidt.
\end{lemma}

The Krull-Schmidt property of $\hmf_{R,w}$ and $\hch(\hmf_{R,w})$ makes the construction of the colored $\mathfrak{sl}(N)$ link homology much easier. Let us briefly recall some useful facts about Krull-Schmidt categories.

\begin{lemma}\cite[Lemma 3.17]{Wu-color}\label{KS-oplus-cancel}
Let $\mathcal{C}$ be a Krull-Schmidt category, and $A,B,C$ objects of $\mathcal{C}$. If $A \oplus C \cong B \oplus C$, then $A \cong B$.
\end{lemma}

\begin{definition}\label{strongly-non-periodic-def}
Let $\mathcal{C}$ be an additive category, and $F:\mathcal{C} \rightarrow \mathcal{C}$ an autofunctor with inverse functor $F^{-1}$. We say that $F$ is strongly non-periodic if, for any object $A$ of $\mathcal{C}$ and $k \in \zed$, $F^k(A) \cong A$ implies that either $A\cong 0$ or $k=0$.

Denote by $\zed_{\geq 0}[F,F^{-1}]$ the ring of formal Laurent polynomials of $F$ whose coefficients are non-negative integers. Each $G=\sum_{i=k}^l b_i F^i \in \zed_{\geq 0}[F,F^{-1}]$ admits a natural interpretation as an endofunctor on $\mathcal{C}$, that is, for any object $A$ of $\mathcal{C}$, 
\[
G(A) = \bigoplus_{i=k}^l (\underbrace{F^i(A)\oplus\cdots\oplus F^i(A)}_{b_i \text{ fold}}).
\] 
\end{definition}

I learned the following Lemma from Yonezawa.

\begin{lemma}\cite[Lemma 3.19]{Wu-color}\label{yonezawa-lemma}
Let $\mathcal{C}$ be a Krull-Schmidt category, and $F:\mathcal{C} \rightarrow \mathcal{C}$ a strongly non-periodic autofunctor. Suppose that $A$, $B$ are objects of $\mathcal{C}$, and there exists a $G \in \zed_{\geq 0}[F,F^{-1}]$ such that $G\neq 0$ and $G(A) \cong G(B)$. Then $A \cong B$.
\end{lemma}

\section{Symmetric Polynomials}\label{sec-sym-poly}

\subsection{Alphabets} An alphabet is a finite set of homogeneous indeterminates of degree $2$. (Note that, the degree of a polynomial in this paper is twice its usual degree.) We denote alphabets by $\mathbb{A},~\mathbb{B}, \dots,~\mathbb{X},~\mathbb{Y}$. Of course, we avoid using letters $\mathbb{C},~\mathbb{N},~\mathbb{Q},~\mathbb{R},~\mathbb{Z}$ to represent alphabets. For an alphabet $\mathbb{X}=\{x_1,\dots,x_m\}$, denote by $\Sym(\mathbb{X})$ the graded ring of symmetric polynomials over $\C$ in $\mathbb{X}=\{x_1,\dots,x_m\}$. (Again, note that our grading of $\Sym(\mathbb{X})$ is twice the usual grading.)

Given a collection $\{\mathbb{X}_1,\dots,\mathbb{X}_l\}$ of pairwise disjoint alphabets, we denote by $\Sym(\mathbb{X}_1|\cdots|\mathbb{X}_l)$ the ring of polynomials in $\mathbb{X}_1\cup\cdots\cup\mathbb{X}_l$ over $\C$ that are symmetric in each $\mathbb{X}_i$, which is a graded free $\Sym(\mathbb{X}_1\cup\cdots\cup\mathbb{X}_l)$-module. Clearly,
\[
\Sym(\mathbb{X}_1|\cdots|\mathbb{X}_l) \cong \Sym(\mathbb{X}_1) \otimes_{\C} \cdots \otimes_{\C} \Sym(\mathbb{X}_l).
\]

\subsection{Basic symmetric polynomials}\label{subsec-basic-sym-poly} For an alphabet $\mathbb{X}=\{x_1,\dots,x_m\}$, we use following notations for the elementary, complete and power sum symmetric polynomials in $\mathbb{X}$:
\begin{eqnarray*}
X_k & = & \begin{cases}
\sum_{1\leq i_1<i_2<\cdots<i_k\leq m} x_{i_1}x_{i_2}\cdots x_{i_k} & \text{if } 1\leq k\leq m, \\
1 & \text{if } k=0, \\
0 & \text{if } k<0 \text{ or } k>m.
\end{cases} \\
h_k(\mathbb{X}) & = & 
\left\{%
\begin{array}{ll}
    \sum_{1\leq i_1\leq i_2 \leq \cdots \leq i_k\leq m} x_{i_1}x_{i_2}\cdots x_{i_k} & \text{if } k>0, \\
    1 & \text{if } k=0, \\
    0 & \text{if } k<0. 
\end{array}%
\right. \\
p_k(\mathbb{X}) & = & 
\left\{%
\begin{array}{ll}
    \sum_{i=1}^{m} x_i^k & \text{if } k\geq0, \\
    0 & \text{if } k<0, 
\end{array}%
\right.
\end{eqnarray*}
$X_k$, $p_k(\mathbb{X})$ and $h_k(\mathbb{X})$ are homogeneous elements of $\Sym(\mathbb{X})$ of degree $2k$. It is well known that $\Sym(\mathbb{X})=\C[X_1,\dots,X_m]$. So $p_k(\mathbb{X})$ and $h_k(\mathbb{X})$ can be uniquely expressed as polynomials in $X_1,\cdots,X_m$.

\begin{lemma}\cite[Lemma 4.1]{Wu-color}\label{power-derive}
Write 
\begin{eqnarray*}
p_k(\mathbb{X}) & = & p_{m,k}(X_1,\dots,X_m) \\
h_k(\mathbb{X}) & = & h_{m,k}(X_1,\dots,X_m).
\end{eqnarray*}
Then
\[
\frac{\partial}{\partial X_j} p_{m,l}(X_1,\dots,X_m) = (-1)^{j+1} l h_{m,l-j}(X_1,\dots,X_m).
\]
\end{lemma}

\subsection{Partitions and Schur polynomials} A partition $\lambda=(\lambda_1\geq\dots\geq\lambda_m)$ is a finite non-increasing sequence of non-negative integers. We denote by $\Lambda_{m,n}$ the set of partitions 
\[
\Lambda_{m,n}=\{(\lambda_1\geq\dots\geq\lambda_m)~|~\lambda_1\leq n\}.
\]

For a partition $\lambda=(\lambda_1\geq\dots\geq\lambda_m)$ with $\lambda_1=n$, define $\lambda'_i=\#\{j~|~\lambda_j\geq i\}$. Then $\lambda':= (\lambda'_1\geq\dots\geq\lambda'_n)$ is also a partition and is called the conjugate of $\lambda$. Clearly, $(\lambda')'=\lambda$.

For an alphabet $\mathbb{X}=\{x_1,\dots,x_m\}$ and a partition $\lambda=(\lambda_1\geq\dots\geq\lambda_m)$, denote by $S_\lambda(\mathbb{X})$ the Schur polynomial in $\mathbb{X}$ associated to the partition $\lambda$. That is, 
\begin{equation}\label{schur-complete}
S_\lambda(\mathbb{X}) = \det (h_{\lambda_i -i +j}(\mathbb{X})) = \left|%
\begin{array}{llll}
h_{\lambda_1}(\mathbb{X}) & h_{\lambda_1+1}(\mathbb{X}) & \dots & h_{\lambda_1+m-1}(\mathbb{X}) \\
h_{\lambda_2-1}(\mathbb{X}) & h_{\lambda_2}(\mathbb{X}) & \dots & h_{\lambda_2+m-2}(\mathbb{X}) \\
\dots & \dots & \dots & \dots \\
h_{\lambda_m-m+1}(\mathbb{X}) & h_{\lambda_m-m+2}(\mathbb{X}) & \dots & h_{\lambda_m}(\mathbb{X}) 
\end{array}%
\right|.
\end{equation}
The elementary and complete symmetric polynomials are special examples of the Schur polynomial.
\begin{eqnarray*}
h_j(\mathbb{X}) & = & S_{(j)}(\mathbb{X}), \\
X_j & = & S_{(\underbrace{1\geq1\geq\cdots\geq1}_{j \text{ parts}})}(\mathbb{X}).
\end{eqnarray*}

There is also a notion of $S_\lambda(-\mathbb{X})$. First, for any $j\in \zed$, 
\[
h_j(-\mathbb{X}):=(-1)^j X_j. 
\]
Then, for any partition $\lambda=(\lambda_1\geq\cdots\geq\lambda_n)$ with $\lambda_1\leq m$, 
\begin{equation}\label{negative-schur-complete}
S_\lambda(-\mathbb{X}) = \det (h_{\lambda_i -i +j}(-\mathbb{X})) = \left|%
\begin{array}{llll}
h_{\lambda_1}(-\mathbb{X}) & h_{\lambda_1+1}(-\mathbb{X}) & \dots & h_{\lambda_1+n-1}(-\mathbb{X}) \\
h_{\lambda_2-1}(-\mathbb{X}) & h_{\lambda_2}(-\mathbb{X}) & \dots & h_{\lambda_2+n-2}(-\mathbb{X}) \\
\dots & \dots & \dots & \dots \\
h_{\lambda_n-n+1}(-\mathbb{X}) & h_{\lambda_n-n+2}(-\mathbb{X}) & \dots & h_{\lambda_n}(-\mathbb{X}) 
\end{array}%
\right|.
\end{equation}

If we write the Schur polynomials in $\mathbb{X}$ as $S_{\lambda}(\mathbb{X})=S_{\lambda}(x_1,\dots,x_m)$, then the Schur polynomials in $-\mathbb{X}$ are given by 
\[
S_\lambda(-\mathbb{X}) = S_{\lambda'}(-x_1,\dots,-x_m),
\]
where $\lambda'$ is the conjugate of $\lambda$.

\begin{theorem}\cite[Proposition \emph{Gr}5]{Lascoux-notes}\label{part-symm-str}
Let $\mathbb{X}=\{x_1,\dots,x_m\}$ and $\mathbb{Y}=\{y_1,\dots,y_n\}$ be two disjoint alphabets. Then $\Sym(\mathbb{X}|\mathbb{Y})$ is a graded free $\Sym(\mathbb{X}\cup\mathbb{Y})$-module. 
 
$\{S_\lambda(\mathbb{X})~|~ \lambda \in \Lambda_{m,n}\}$ and $\{S_\lambda(-\mathbb{Y})~|~ \lambda \in \Lambda_{m,n}\}$ are two homogeneous bases for the $\Sym(\mathbb{X}\cup\mathbb{Y})$-module $\Sym(\mathbb{X}|\mathbb{Y})$.

Moreover, there is a unique $\Sym(\mathbb{X}\cup\mathbb{Y})$-module homomorphism 
\[
\zeta:\Sym(\mathbb{X}|\mathbb{Y}) \rightarrow \Sym(\mathbb{X}\cup\mathbb{Y}),
\] 
called the Sylvester operator, such that, for $\lambda,\mu \in \Lambda_{m,n}$,
\[
\zeta(S_\lambda(\mathbb{X}) \cdot S_\mu(-\mathbb{Y})) = \left\{%
\begin{array}{ll}
    1 & \text{if } \lambda_j + \mu_{m+1-j} =n ~\forall j=1,\dots,m, \\
    0 & \text{otherwise.}  \\
\end{array}%
\right.
\]
\end{theorem}

It is well known that 
\begin{equation}\label{compute-quantum-binary}
\qb{m+n}{n} =  q^{-mn}\sum_{\lambda\in\Lambda_{m,n}} q^{2|\lambda|}.
\end{equation}
So Theorem \ref{part-symm-str} implies the following corollary.

\begin{corollary}\label{part-symm-grade}
Let $\mathbb{X}=\{x_1,\dots,x_m\}$ and $\mathbb{Y}=\{y_1,\dots,y_n\}$ be two disjoint alphabets. Then, as graded $\Sym(\mathbb{X}\cup\mathbb{Y})$-modules, 
\[
\Sym(\mathbb{X}|\mathbb{Y}) \cong \Sym(\mathbb{X}\cup\mathbb{Y})\{\qb{m+n}{n}\cdot q^{mn}\}.
\]
\end{corollary}

Let $G_{m,N}$ be the complex $(m,N)$-Grassmannian, that is, the manifold of all complex $m$-dimensional subspaces of $\C^N$. Theorem \ref{part-symm-str} is essentially a computation of the $GL(N;\C)$-equivariant cohomology of $G_{m,N}$, where the Sylvester operator is essentially the Poincar\'{e} duality. (See for example \cite[Lecture 6]{Fulton-notes}.) The following is a similar computation of the usual cohomology of $G_{m,N}$.

\begin{theorem}\cite[Lecture 6]{Fulton-notes}\label{grassmannian}
Let $\mathbb{X}$ be an alphabet of $m$ independent indeterminates. Then 
\[
H^\ast(G_{m,N};\C) \cong \Sym(\mathbb{X})/(h_{N+1-m}(\mathbb{X}),h_{N+2-m}(\mathbb{X}),\dots,h_{N}(\mathbb{X}))
\] 
as graded $\C$-algebras. 

As a graded $\C$-linear space, $H^\ast(G_{m,N};\C)$ has a homogeneous basis 
\[
\{S_\lambda(\mathbb{X})~|~ \lambda \in \Lambda_{m,N-m}\}.
\]

Under the above basis, the Poincar\'{e} duality of $H^\ast(G_{m,N};\C)$ is given by a $\C$-linear trace map 
\[
\Tr:\Sym(\mathbb{X})/(h_{N+1-m}(\mathbb{X}),h_{N+2-m}(\mathbb{X}),\dots,h_{N}(\mathbb{X})) \rightarrow \C
\] 
satisfying
\[
\Tr(S_\lambda(\mathbb{X}) \cdot S_\mu(\mathbb{X})) = \left\{%
\begin{array}{ll}
    1 & \text{if } \lambda_j + \mu_{m+1-j} =N-m ~\forall j=1,\dots,m, \\
    0 & \text{otherwise.}  \\
\end{array}%
\right.
\]
\end{theorem}

\begin{corollary}\label{grassmannian-grade}
As graded $\C$-linear spaces,
\[
H^\ast(G_{m,N};\C) \cong \C \{\qb{N}{m}\cdot q^{m(N-m)}\},
\]
where $\C$ on the right hand side has grading $0$.
\end{corollary}

\section{Matrix Factorizations Associated to MOY Graphs}\label{sec-mf-MOY}

In this section, we define the matrix factorizations associated to MOY graphs and discuss their basic properties.

\subsection{Markings} To define the matrix factorization associated to a MOY graph, we need to put a marking on this MOY graph.

\begin{definition}\label{MOY-marking-def}
A marking of a MOY graph $\Gamma$ consists the following:
\begin{enumerate}
	\item A finite collection of marked points on $\Gamma$ such that
	\begin{itemize}
	\item every edge of $\Gamma$ contains at least one marked point;
	\item all the end points (vertices of valence $1$) are marked;
	\item none of the internal vertices (vertices of valence at least $2$) is marked.
  \end{itemize}
  \item An assignment of pairwise disjoint alphabets to the marked points such that the alphabet associated to a marked point on an edge of color $m$ has $m$ independent indeterminates. (Recall that an alphabet is a finite collection of homogeneous indeterminates of degree $2$.)
\end{enumerate}
\end{definition}

\begin{figure}[ht]

\setlength{\unitlength}{1pt}

\begin{picture}(360,80)(-180,-40)


\put(0,0){\vector(-1,1){15}}

\put(-15,15){\line(-1,1){15}}

\put(-23,25){\tiny{$i_1$}}

\put(-33,32){\small{$\mathbb{X}_1$}}

\put(0,0){\vector(-1,2){7.5}}

\put(-7.5,15){\line(-1,2){7.5}}

\put(-11,25){\tiny{$i_2$}}

\put(-18,32){\small{$\mathbb{X}_2$}}

\put(3,25){$\cdots$}

\put(0,0){\vector(1,1){15}}

\put(15,15){\line(1,1){15}}

\put(31,25){\tiny{$i_k$}}

\put(27,32){\small{$\mathbb{X}_k$}}


\put(4,-2){$v$}

\multiput(-50,0)(5,0){19}{\line(1,0){3}}

\put(-70,0){$L_v$}

\put(45,0){\tiny{$i_1+i_2+\cdots +i_k = j_1+j_2+\cdots +j_l$}}


\put(-30,-30){\vector(1,1){15}}

\put(-15,-15){\line(1,1){15}}

\put(-26,-30){\tiny{$j_1$}}

\put(-33,-40){\small{$\mathbb{Y}_1$}}

\put(-15,-30){\vector(1,2){7.5}}

\put(-7.5,-15){\line(1,2){7.5}}

\put(-13,-30){\tiny{$j_2$}}

\put(-18,-40){\small{$\mathbb{Y}_2$}}

\put(3,-30){$\cdots$}

\put(30,-30){\vector(-1,1){15}}

\put(15,-15){\line(-1,1){15}}

\put(31,-30){\tiny{$j_l$}}

\put(27,-40){\small{$\mathbb{Y}_l$}}

\end{picture}

\caption{A marking of a neighborhood of an internal vertex}\label{general-MOY-vertex}

\end{figure}
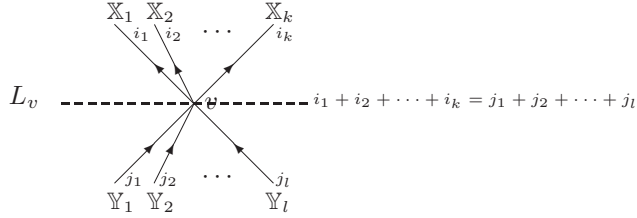

\subsection{The matrix factorization associated to a MOY graph} Recall that $N$ is a fixed positive integer. For a MOY graph $\Gamma$ with a marking, cut it at the marked points. This gives a collection of marked MOY graphs, each of which is a star-shaped neighborhood of a vertex in $\Gamma$ and is marked only at its end points. (If an edge of $\Gamma$ has two or more marked points, then some of these pieces may be oriented arcs from one marked point to another. In this case, we consider such an arc as a neighborhood of an additional vertex of valence $2$ in the middle of that arc.)

Let $v$ be a vertex of $\Gamma$ with coloring and marking around it given as in Figure \ref{general-MOY-vertex}. Set $m=i_1+i_2+\cdots +i_k = j_1+j_2+\cdots +j_l$ (the width of $v$.) Define 
\[
R=\Sym(\mathbb{X}_1|\dots|\mathbb{X}_k|\mathbb{Y}_1|\dots|\mathbb{Y}_l).
\] 
Write $\mathbb{X}=\mathbb{X}_1\cup\cdots\cup \mathbb{X}_k$ and $\mathbb{Y}=\mathbb{Y}_1\cup\cdots\cup \mathbb{Y}_l$. Denote by $X_j$ the $j$-th elementary symmetric polynomial in $\mathbb{X}$ and by $Y_j$ the $j$-th elementary symmetric polynomial in $\mathbb{Y}$. For $j=1,\dots,m$, define
\[
U_j = \frac{p_{m,N+1}(Y_1,\dots,Y_{j-1},X_j,\dots,X_m) - p_{m,N+1}(Y_1,\dots,Y_j,X_{j+1},\dots,X_m)}{X_j-Y_j},
\]
where $p_{m,N+1}(X_1,\dots,X_m)=p_{N+1}(\mathbb{X})$ is the $(N+1)$-th power sum symmetric polynomial in $\mathbb{X}$. (See Subsection \ref{subsec-basic-sym-poly}.)

The matrix factorization associated to the vertex $v$ is
\[
C(v)=\left(%
\begin{array}{cc}
  U_1 & X_1-Y_1 \\
  U_2 & X_2-Y_2 \\
  \dots & \dots \\
  U_m & X_m-Y_m
\end{array}%
\right)_R
\{q^{-\sum_{1\leq s<t \leq k} i_si_t}\},
\]
whose potential is $\sum_{j=1}^m (X_j-Y_j)U_j = p_{N+1}(\mathbb{X})-p_{N+1}(\mathbb{Y})$, where $p_{N+1}(\mathbb{X})$ and $p_{N+1}(\mathbb{Y})$ are the $(N+1)$-th power sum symmetric polynomials in $\mathbb{X}$ and $\mathbb{Y}$. (Again, see Subsection \ref{subsec-basic-sym-poly} for the definition.)

\begin{remark}\label{MOY-freedom}
Since 
\[
\Sym(\mathbb{X}|\mathbb{Y})=\C[X_1,\dots,X_m,Y_1,\dots,Y_m]=\C[X_1-Y_1,\dots,X_m-Y_m,Y_1,\dots,Y_m],
\] 
it is clear that $\{X_1-Y_1,\dots,X_m-Y_m\}$ is $\Sym(\mathbb{X}|\mathbb{Y})$-regular. (See \cite[Definition 2.17]{Wu-color} for the definition.) By Theorem \ref{part-symm-str}, $R$ is a free $\Sym(\mathbb{X}|\mathbb{Y})$-module. So $\{X_1-Y_1,\dots,X_m-Y_m\}$ is also $R$-regular. Thus, by \cite[Theorem 2.1]{KR3}, the isomorphism type of $C(v)$ does not depend on the particular choice of $U_1,\dots,U_m$ as long as they are homogeneous with the right degrees and the potential of $C(v)$ remains $\sum_{j=1}^m (X_j-Y_j)U_j = p_{N+1}(\mathbb{X})-p_{N+1}(\mathbb{Y})$. 

This observation is useful since it allows us to
\begin{enumerate}[(i)]
	\item choose more convenient $U_j$'s in some computations,
	\item not keep track of what $U_j$'s are in most computations. (In this case, we denote $U_j$'s simply by $\ast$'s.)
\end{enumerate} 
\end{remark}

\begin{definition}\label{MOY-mf-def}
\[
C(\Gamma) := \bigotimes_{v} C(v),
\]
where $v$ runs through all the internal vertices of $\Gamma$ (including those additional $2$-valent vertices.) Here, the tensor product is done over the common end points. More precisely, for two sub-MOY graphs $\Gamma_1$ and $\Gamma_2$ of $\Gamma$ intersecting only at (some of) their end points, let $\mathbb{W}_1,\dots,\mathbb{W}_n$ be the alphabets associated to these common end points. Then, in the above tensor product, $C(\Gamma_1)\otimes C(\Gamma_2)$ is the tensor product $C(\Gamma_1)\otimes_{\Sym(\mathbb{W}_1|\dots|\mathbb{W}_n)} C(\Gamma_2)$.

$C(\Gamma)$ has a $\zed_2$-grading and a quantum grading.

If $\Gamma$ is closed, i.e. has no end points, then $C(\Gamma)$ is an object of $\hmf_{\C,0}$. (See Lemma \ref{MOY-object-of-hmf} below.)

Assume $\Gamma$ has end points. Let $\mathbb{E}_1,\dots,\mathbb{E}_n$ be the alphabets assigned to all end points of $\Gamma$, among which $\mathbb{E}_1,\dots,\mathbb{E}_k$ are assigned to exits and $\mathbb{E}_{k+1},\dots,\mathbb{E}_n$ are assigned to entrances. Then the potential of $C(\Gamma)$ is  
\[
w= \sum_{i=1}^k p_{N+1}(\mathbb{E}_i) - \sum_{j=k+1}^n p_{N+1}(\mathbb{E}_j).
\]
Let $R=\Sym(\mathbb{E}_1|\cdots|\mathbb{E}_n)$. Although the alphabets assigned to all marked points on $\Gamma$ are used in its construction, $C(\Gamma)$ is viewed as an object of $\hmf_{R,w}$. (See Lemma \ref{MOY-object-of-hmf} below.) Note that, in this case, $w$ is a non-degenerate element of $R$.

We allow the MOY graph to be empty. In this case, we define 
\[
C(\emptyset)=\C\rightarrow 0 \rightarrow \C,
\]
where the $\zed_2$-grading and the quantum grading of $\C$ are both $0$.
\end{definition}

\begin{definition}\label{homology-MOY-def}
Let $\Gamma$ be a MOY graph with a marking. 
\begin{enumerate}[(i)]
	\item If $\Gamma$ is closed, i.e. has no end points, then $C(\Gamma)$ is a chain complex. Denote by $H(\Gamma)$ the homology of $C(\Gamma)$. Note that $H(\Gamma)$ inherits both gradings of $C(\Gamma)$.
	\item If $\Gamma$ has end points, let $\mathbb{E}_1,\dots,\mathbb{E}_n$ be the alphabets assigned to all end points of $\Gamma$, and $R=\Sym(\mathbb{E}_1|\cdots|\mathbb{E}_n)$. Denote by $E_{i,j}$ the $j$-th elementary symmetric polynomial in $\mathbb{E}_i$ and by $\mathfrak{I}$ the maximal homogeneous ideal of $R$ generated by $\{E_{i,j}\}$. Then $H(\Gamma)$ is defined to be $H_R(C(\Gamma))$, that is the homology of the chain complex $C(\Gamma)/\mathfrak{I}\cdot C(\Gamma)$. $H(\Gamma)$ inherits both gradings of $C(\Gamma)$.
\end{enumerate}
Note that (i) is a special case of (ii).
\end{definition}

\subsection{Basic properties}

\begin{lemma}\cite[Lemma 5.6]{Wu-color}\label{marking-independence}
If $\Gamma$ is a MOY graph, then the homotopy type of $C(\Gamma)$ does not depend on the choice of the marking. Therefore, the isomorphism type of $H(\Gamma)$ does not depend on the choice of the marking.
\end{lemma}
\begin{proof}
We only need to show that adding or removing an extra marked point corresponds to a homotopy equivalence of matrix factorizations preserving both gradings. This follows easily from \cite[Proposition 2.19]{Wu-color}.
\end{proof}

\begin{lemma}\cite[Lemma 5.8]{Wu-color}\label{width-cap}
If $\Gamma$ is a MOY graph with a vertex of width greater than $N$, then $C(\Gamma)\simeq 0$.
\end{lemma}
\begin{proof}
Suppose the vertex $v$ of $\Gamma$ has width $m>N$. Using Newton's Identity, it is easy to check that, in the construction of $C(v)$, $U_{N+1}= (-1)^{N} (N+1)$ is a non-zero scalar. Apply \cite[Corollary 2.25]{Wu-color} to the entry $U_{N+1}$ in $C(v)$. One can see that $H(v)= 0$. By Proposition \ref{prop-homology-detects-homotopy}, this implies $C(v)\simeq 0$ and therefore $C(\Gamma)\simeq 0$.
\end{proof}

The proofs of the next two lemmas are slightly more technical. We refer the reader to \cite{Wu-color} for these proofs.

\begin{lemma}\cite[Lemma 5.10]{Wu-color}\label{schur-null-homotopic}
Let $\Gamma$ be a MOY graph, and $\mathbb{X}=\{x_1,\dots,x_m\}$ an alphabet associated to a marked point on an edge of $\Gamma$ of color $m$.  
Suppose that $\mu=(\mu_1\geq\cdots\geq\mu_m)$ is a partition such that $\mu_1>(N-m)$. Then multiplication by $S_{\mu}(\mathbb{X})$ is a null-homotopic endomorphism of $C(\Gamma)$.
\end{lemma}

\begin{lemma}\cite[Lemma 5.11]{Wu-color}\label{MOY-object-of-hmf} 
Let $\Gamma$ be a MOY graph, and $\mathbb{E}_1,\dots,\mathbb{E}_n$ the alphabets assigned to all end points of $\Gamma$, among which $\mathbb{E}_1,\dots,\mathbb{E}_k$ are assigned to exits and $\mathbb{E}_{k+1},\dots,\mathbb{E}_n$ are assigned to entrances. (Here we allow $n=0$, i.e. $\Gamma$ to be closed.) Write $R=\Sym(\mathbb{E}_1|\cdots|\mathbb{E}_n)$ and $w= \sum_{i=1}^k p_{N+1}(\mathbb{E}_i) - \sum_{j=k+1}^n p_{N+1}(\mathbb{E}_j)$. Then $C(\Gamma)$ is an object of $\hmf_{R,w}$.
\end{lemma}

\begin{figure}[ht]

\begin{picture}(360,90)(-180,-50)


\put(-130,0){\vector(-3,2){22.5}}

\put(-152.5,15){\line(-3,2){22.5}}

\put(-178,25){\tiny{$i_1$}}

\put(-178,32){\small{$\mathbb{X}_1$}}

\put(-152,15){$\cdots$}

\put(-150,25){\tiny{$i_s$}}

\put(-148,32){\small{$\mathbb{X}_s$}}

\put(-115,25){\tiny{$i_{s+1}$}}

\put(-117,32){\small{$\mathbb{X}_{s+1}$}}

\put(-130,0){\vector(0,1){7.5}}

\put(-130,7.5){\line(0,1){7.5}}

\put(-130,15){\vector(-1,1){7.5}}

\put(-137.5,22.5){\line(-1,1){7.5}}

\put(-130,15){\vector(1,1){7.5}}

\put(-122.5,22.5){\line(1,1){7.5}}

\put(-131,10){\line(1,0){2}}

\put(-129,6){\small{$\mathbb{A}$}}

\put(-119,15){$\cdots$}

\put(-130,0){\vector(3,2){22.5}}

\put(-107.5,15){\line(3,2){22.5}}

\put(-84,25){\tiny{$i_k$}}

\put(-87,32){\small{$\mathbb{X}_k$}}


\put(-123,-2){$v$}


\put(-152.5,-15){\line(3,2){22.5}}

\put(-175,-30){\vector(3,2){22.5}}

\put(-178,-27){\tiny{$j_1$}}

\put(-178,-38){\small{$\mathbb{Y}_1$}}

\put(-152,-20){$\cdots$}

\put(-137.5,-15){\line(1,2){7.5}}

\put(-145,-30){\vector(1,2){7.5}}

\put(-150,-27){\tiny{$j_t$}}

\put(-148,-38){\small{$\mathbb{Y}_t$}}

\put(-122.5,-15){\line(-1,2){7.5}}

\put(-115,-30){\vector(-1,2){7.5}}

\put(-115,-27){\tiny{$j_{t+1}$}}

\put(-117,-38){\small{$\mathbb{Y}_{t+1}$}}

\put(-119,-20){$\cdots$}

\put(-107.5,-15){\line(-3,2){22.5}}

\put(-85,-30){\vector(-3,2){22.5}}

\put(-84,-27){\tiny{$j_l$}}

\put(-87,-38){\small{$\mathbb{Y}_l$}}

\put(-130,-50){$\Gamma_1$}


\put(0,0){\vector(-3,2){22.5}}

\put(-22.5,15){\line(-3,2){22.5}}

\put(-48,25){\tiny{$i_1$}}

\put(-48,32){\small{$\mathbb{X}_1$}}

\put(-22,15){$\cdots$}

\put(0,0){\vector(-1,2){7.5}}

\put(-7.5,15){\line(-1,2){7.5}}

\put(-20,25){\tiny{$i_s$}}

\put(-18,32){\small{$\mathbb{X}_s$}}

\put(0,0){\vector(1,2){7.5}}

\put(7.5,15){\line(1,2){7.5}}

\put(15,25){\tiny{$i_{s+1}$}}

\put(13,32){\small{$\mathbb{X}_{s+1}$}}

\put(11,15){$\cdots$}

\put(0,0){\vector(3,2){22.5}}

\put(22.5,15){\line(3,2){22.5}}

\put(46,25){\tiny{$i_k$}}

\put(43,32){\small{$\mathbb{X}_k$}}


\put(7,-2){$v$}


\put(-22.5,-15){\line(3,2){22.5}}

\put(-45,-30){\vector(3,2){22.5}}

\put(-48,-27){\tiny{$j_1$}}

\put(-48,-38){\small{$\mathbb{Y}_1$}}

\put(-22,-20){$\cdots$}

\put(-7.5,-15){\line(1,2){7.5}}

\put(-15,-30){\vector(1,2){7.5}}

\put(-20,-27){\tiny{$j_t$}}

\put(-18,-38){\small{$\mathbb{Y}_t$}}

\put(7.5,-15){\line(-1,2){7.5}}

\put(15,-30){\vector(-1,2){7.5}}

\put(15,-27){\tiny{$j_{t+1}$}}

\put(13,-38){\small{$\mathbb{Y}_{t+1}$}}

\put(11,-20){$\cdots$}

\put(22.5,-15){\line(-3,2){22.5}}

\put(45,-30){\vector(-3,2){22.5}}

\put(46,-27){\tiny{$j_l$}}

\put(43,-38){\small{$\mathbb{Y}_l$}}

\put(0,-50){$\Gamma$}


\put(130,0){\vector(-3,2){22.5}}

\put(107.5,15){\line(-3,2){22.5}}

\put(82,25){\tiny{$i_1$}}

\put(82,32){\small{$\mathbb{X}_1$}}

\put(108,15){$\cdots$}

\put(130,0){\vector(-1,2){7.5}}

\put(122.5,15){\line(-1,2){7.5}}

\put(110,25){\tiny{$i_s$}}

\put(112,32){\small{$\mathbb{X}_s$}}

\put(130,0){\vector(1,2){7.5}}

\put(137.5,15){\line(1,2){7.5}}

\put(145,25){\tiny{$i_{s+1}$}}

\put(143,32){\small{$\mathbb{X}_{s+1}$}}

\put(141,15){$\cdots$}

\put(130,0){\vector(3,2){22.5}}

\put(152.5,15){\line(3,2){22.5}}

\put(176,25){\tiny{$i_k$}}

\put(173,32){\small{$\mathbb{X}_k$}}


\put(137,-2){$v$}


\put(107.5,-15){\line(3,2){22.5}}

\put(85,-30){\vector(3,2){22.5}}

\put(82,-27){\tiny{$j_1$}}

\put(82,-38){\small{$\mathbb{Y}_1$}}

\put(108,-20){$\cdots$}

\put(115,-30){\vector(1,1){7.5}}

\put(122.5,-22.5){\line(1,1){7.5}}

\put(110,-27){\tiny{$j_t$}}

\put(112,-38){\small{$\mathbb{Y}_t$}}

\put(145,-30){\vector(-1,1){7.5}}

\put(137.5,-22.5){\line(-1,1){7.5}}

\put(145,-27){\tiny{$j_{t+1}$}}

\put(143,-38){\small{$\mathbb{Y}_{t+1}$}}

\put(130,-7.5){\line(0,1){7.5}}

\put(130,-15){\vector(0,1){7.5}}

\put(129,-5){\line(1,0){2}}

\put(132,-12){\small{$\mathbb{B}$}}

\put(141,-20){$\cdots$}

\put(152.5,-15){\line(-3,2){22.5}}

\put(175,-30){\vector(-3,2){22.5}}

\put(176,-27){\tiny{$j_l$}}

\put(173,-38){\small{$\mathbb{Y}_l$}}

\put(130,-50){$\Gamma_2$}
\end{picture}

\caption{}\label{edge-contraction-figure}

\end{figure}

\begin{lemma}\cite[Lemma 5.12]{Wu-color}\label{edge-contraction}
Let $\Gamma$, $\Gamma_1$ and $\Gamma_2$ be MOY graphs shown in Figure \ref{edge-contraction-figure}. Then $C(\Gamma_1) \simeq C(\Gamma_2) \simeq C(\Gamma)$.
\end{lemma}
\begin{proof}
This follows easily from \cite[Proposition 2.19]{Wu-color}.
\end{proof}

\section{Homological MOY Calculus, Part One}\label{sec-MOY-decomps-part1}

The MOY calculus consists of seven graphical recursive relations (Parts (1)-(7) of Theorem \ref{MOY-poly-skein}.) In this section, we prove homological versions of Parts (1)-(4) of Theorem \ref{MOY-poly-skein}. The homological versions of the remaining three relations are harder to establish. We will sketch proofs of these in Section \ref{sec-MOY-decomps-part2} after we introduce the morphisms induced by certain local changes of MOY graphs in Section \ref{sec-morph}.

Note that the homological versions of the MOY relations remain true if we reverse the orientation of the MOY graph or the orientation of $\mathbb{R}^2$.

\subsection{MOY Relation (1)}

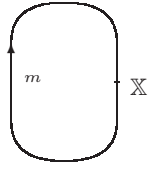
\begin{figure}[ht]

\setlength{\unitlength}{1pt}

\begin{picture}(360,60)(-180,0)


\qbezier(0,60)(-20,60)(-20,45)

\qbezier(0,60)(20,60)(20,45)

\put(-20,15){\vector(0,1){30}}

\put(20,15){\line(0,1){30}}

\qbezier(0,0)(-20,0)(-20,15)

\qbezier(0,0)(20,0)(20,15)

\put(19,30){\line(1,0){2}}

\put(-15,30){\tiny{$m$}}

\put(25,25){\small{$\mathbb{X}$}}
 
\end{picture}

\caption{$m$-colored circle with a single marked point}\label{circle-module-figure}

\end{figure}

\begin{proposition}\cite[Proposition 6.3]{Wu-color}\label{circle-module}
Let $\bigcirc_m$ be the circle colored by $m$ with a single marked point in Figure \ref{circle-module-figure}. Then, as graded matrix factorizations over $\C$,
\[
C(\bigcirc_m) \simeq (C(\emptyset)\{q^{-m(N-m)}\} \left\langle m \right\rangle) \otimes_\C \Sym(\mathbb{X})/(h_N(\mathbb{X}),h_{N-1}(\mathbb{X}),\dots,h_{N+1-m}(\mathbb{X})) ,
\] 
where $C(\emptyset)$ is the matrix factorization $\C\rightarrow 0 \rightarrow \C$, $\mathbb{X}$ is an alphabet of $m$ indeterminates marking $\bigcirc_m$ and $\Sym(\mathbb{X})/(h_N(\mathbb{X}),h_{N-1}(\mathbb{X}),\dots,h_{N+1-m}(\mathbb{X}))$ has $\zed_2$-grading $0$. 

Consequently, as $\zed$-graded modules over $\Sym(\mathbb{X})$, 
\[
H(\bigcirc_m) \cong H^\ast(G_{m,N})\{q^{-m(N-m)}\},
\] 
where $G_{m,N}$ is the complex $(m,N)$-Grassmannian.
\end{proposition}

The following is a straightforward corollary of Proposition \ref{circle-module} and Corollary \ref{grassmannian-grade}.

\begin{corollary}\cite[Corollary 6.1]{Wu-color}\label{circle-dimension}
$C(\bigcirc_m) \simeq C(\emptyset)\{\qb{N}{m}\}\left\langle m \right\rangle$ and $H(\bigcirc_m) \cong H(\emptyset)\{\qb{N}{m}\}\left\langle m \right\rangle$.
\end{corollary}

\begin{proof}[Proof of Proposition \ref{circle-module}]
By definition,
\[
C(\bigcirc_m) =
\left(%
\begin{array}{cc}
  U_1 & 0 \\
  \dots & \dots \\
  U_m & 0 
\end{array}%
\right)_{\Sym(\mathbb{X})}
\]
where $U_j=\frac{\partial}{\partial X_j}p_{m,N+1}(X_1,\dots,X_m)$ and $X_j$ is the $j$-th elementary symmetric polynomial in $\mathbb{X}$. From Lemma \ref{power-derive}, we know that
\[
U_j=(-1)^{j+1} (N+1) h_{N+1-j}(\mathbb{X}).
\]
By \cite[Proposition 6.2]{Wu-color}, $(U_1,\dots,U_{m})$ is $\Sym(\mathbb{X})$-regular. Thus, we can apply \cite[Corollary 2.25]{Wu-color} successively to the rows of $C(\bigcirc_m)$ from top to bottom to get a homogeneous morphism  
\[
C(\bigcirc_m) \rightarrow (C(\emptyset)\{q^{-m(N-m)}\} \left\langle m \right\rangle) \otimes_\C \Sym(\mathbb{X})/(h_N(\mathbb{X}),h_{N-1}(\mathbb{X}),\dots,h_{N+1-m}(\mathbb{X}))
\]
that preserves both gradings and induces an isomorphism
\[
H(\bigcirc_m) \xrightarrow{\cong} (H(\emptyset)\{q^{-m(N-m)}\} \left\langle m \right\rangle) \otimes_\C \Sym(\mathbb{X})/(h_N(\mathbb{X}),h_{N-1}(\mathbb{X}),\dots,h_{N+1-m}(\mathbb{X})).
\] 
By Proposition \ref{prop-homotopy-equal-isomorphism-on-homology}, this implies 
\[
C(\bigcirc_m) \simeq (C(\emptyset)\{q^{-m(N-m)}\} \left\langle m \right\rangle) \otimes_\C \Sym(\mathbb{X})/(h_N(\mathbb{X}),h_{N-1}(\mathbb{X}),\dots,h_{N+1-m}(\mathbb{X})).
\] 

It then follows from Theorem \ref{grassmannian} that 
\[
H(\bigcirc_m) \cong H^\ast(G_{m,N})\{q^{-m(N-m)}\}.
\] 
\end{proof}

\subsection{MOY Relation (2)} Lemma \ref{edge-contraction} implies that the matrix factorization associated to any MOY graph is homotopic to that associated to a trivalent MOY graph. Thus, we do not lose any information by considering only the trivalent MOY graphs. But, in some cases, it is more convenient to use vertices of higher valence. Moreover, applying Lemma \ref{edge-contraction} to trivalent MOY graphs, we get the following simple corollary.

\begin{corollary}\cite[Corollary 5.13]{Wu-color}\label{contract-expand} \vspace{-2pc}
\[
C(\setlength{\unitlength}{1pt}
\begin{picture}(50,50)(-80,20)

\put(-60,10){\vector(0,1){10}}

\put(-60,20){\vector(-1,1){20}}

\put(-60,20){\vector(1,1){10}}

\put(-50,30){\vector(-1,1){10}}

\put(-50,30){\vector(1,1){10}}

\put(-75,3){\tiny{$i+j+k$}}

\put(-55,21){\tiny{$j+k$}}

\put(-80,42){\tiny{$i$}}

\put(-60,42){\tiny{$j$}}

\put(-40,42){\tiny{$k$}}

\end{picture}) \simeq C(\setlength{\unitlength}{1pt}
\begin{picture}(50,50)(40,20)

\put(60,10){\vector(0,1){10}}

\put(60,20){\vector(1,1){20}}

\put(60,20){\vector(-1,1){10}}

\put(50,30){\vector(1,1){10}}

\put(50,30){\vector(-1,1){10}}

\put(45,3){\tiny{$i+j+k$}}

\put(38,21){\tiny{$i+j$}}

\put(80,42){\tiny{$k$}}

\put(60,42){\tiny{$j$}}

\put(40,42){\tiny{$i$}}

\end{picture}). \vspace{2pc}
\]
\end{corollary}

\subsection{MOY Relation (3)}

\begin{proposition}\cite[Theorem 5.14]{Wu-color}\label{decomp-II}

$~$ \vspace{-2pc}
\[
C(\input{v-vector-m+n-bubble-slide}) \simeq C(\setlength{\unitlength}{.75pt}
\begin{picture}(55,80)(-20,40)
\put(0,0){\vector(0,1){80}}
\put(5,75){\tiny{$_{m+n}$}}
\end{picture})\{\qb{m+n}{m}\}. \vspace{3pc}
\]
\end{proposition}

\begin{proof}
Mark the MOY graphs as in Figure \ref{decomp-II-figure}. Denote by $X_j$ be $j$-th elementary symmetric polynomial in $\mathbb{X}$, and use similar notations for the other alphabets. Let $\mathbb{W}=\mathbb{A}\cup\mathbb{B}$. Then the $k$-th elementary symmetric polynomial in $\mathbb{W}$ is
\[
W_k=\sum_{i+j=k}A_iB_j.
\] 

\begin{figure}[ht]

\setlength{\unitlength}{1pt}

\begin{picture}(360,70)(-180,-10)


\put(-60,0){\vector(0,1){15}}

\qbezier(-60,15)(-70,15)(-70,25)

\put(-70,25){\vector(0,1){10}}

\qbezier(-70,35)(-70,45)(-60,45)

\put(-71,26){\line(1,0){2}}

\qbezier(-60,15)(-50,15)(-50,25)

\put(-50,25){\vector(0,1){10}}

\qbezier(-50,35)(-50,45)(-60,45)

\put(-51,26){\line(1,0){2}}

\put(-60,45){\vector(0,1){15}}

\put(-85,55){\tiny{$m+n$}}

\put(-58,54){\small{$\mathbb{Y}$}}

\put(-85,0){\tiny{$m+n$}}

\put(-58,0){\small{$\mathbb{X}$}}

\put(-77,30){\tiny{$m$}}

\put(-77,22){\small{$\mathbb{A}$}}

\put(-47,30){\tiny{$n$}}

\put(-48,22){\small{$\mathbb{B}$}}

\put(-63,-10){$\Gamma$}


\put(60,0){\vector(0,1){60}}

\put(62,54){\small{$\mathbb{Y}$}}

\put(65,30){\tiny{$m+n$}}

\put(62,0){\small{$\mathbb{X}$}}

\put(57,-10){$\Gamma_1$}

\end{picture}

\caption{}\label{decomp-II-figure}

\end{figure}

By Theorem \ref{part-symm-str} and Corollary \ref{part-symm-grade}, 
\[
\Sym(\mathbb{X}|\mathbb{Y}|\mathbb{A}|\mathbb{B}) = \Sym(\mathbb{X}|\mathbb{Y}|\mathbb{W})\{q^{mn}\qb{m+n}{m}\}.
\]
So
\begin{eqnarray*}
C(\Gamma) & \cong & 
\left(%
\begin{array}{ll}
  \ast & Y_1 - W_1 \\
  \dots & \dots \\
  \ast & Y_n - W_n \\
  \ast & W_1 - X_1 \\
  \dots & \dots \\
  \ast & W_n - X_m 
\end{array}%
\right)_{\Sym(\mathbb{X}|\mathbb{Y}|\mathbb{A}|\mathbb{B})}
\{q^{-mn}\}, \\
& \cong & 
\left(%
\begin{array}{ll}
  \ast & Y_1 - W_1 \\
  \dots & \dots \\
  \ast & Y_n - W_n \\
  \ast & W_1 - X_1 \\
  \dots & \dots \\
  \ast & W_n - X_m 
\end{array}%
\right)_{\Sym(\mathbb{X}|\mathbb{Y}|\mathbb{W})}
\{\qb{m+n}{m}\}, \\
& \simeq & 
\left(%
\begin{array}{ll}
  \ast & Y_1 - X_1 \\
  \dots & \dots \\
  \ast & Y_n - X_n
\end{array}%
\right)_{\Sym(\mathbb{X}|\mathbb{Y})}
\{\qb{m+n}{m}\}, \\
& \cong & C(\Gamma_1)\{\qb{m+n}{m}\}.
\end{eqnarray*}
where the homotopy is given by \cite[Proposition 2.19]{Wu-color}. 
\end{proof}

\subsection{MOY Relation (4)}

\begin{proposition}\cite[Theorem 5.16]{Wu-color}\label{decomp-I}
$~$ \vspace{-2pc}
\[
C(\setlength{\unitlength}{.75pt}
\begin{picture}(60,80)(-30,40)
\put(0,0){\vector(0,1){30}}
\put(0,30){\vector(0,1){20}}
\put(0,50){\vector(0,1){30}}

\put(-1,40){\line(1,0){2}}

\qbezier(0,30)(25,20)(25,30)
\qbezier(0,50)(25,60)(25,50)
\put(25,50){\vector(0,-1){20}}

\put(5,75){\tiny{$_{m}$}}
\put(5,5){\tiny{$_{m}$}}
\put(-30,38){\tiny{$_{m+n}$}}
\put(14,60){\tiny{$_{n}$}}
\end{picture}) \simeq C(\setlength{\unitlength}{.75pt}
\begin{picture}(40,80)(-20,40)
\put(0,0){\vector(0,1){80}}
\put(5,75){\tiny{$_{m}$}}
\end{picture})\{ \qb{N-m}{n}\}\left\langle n \right\rangle. \vspace{2pc}
\]
\end{proposition}

\begin{proof}
By a direct computation of matrix factorizations, one can show that the proposition is true if $n=N-m$. This is done in \cite[Lemma 5.15]{Wu-color}. 

Now consider the general situation. By \cite[Lemma 5.15]{Wu-color}, we have \vspace{-4pc}
\[
C(\setlength{\unitlength}{.75pt}
\begin{picture}(60,80)(-30,40)
\put(0,0){\vector(0,1){30}}
\put(0,30){\vector(0,1){20}}
\put(0,50){\vector(0,1){30}}

\put(-1,40){\line(1,0){2}}

\qbezier(0,30)(25,20)(25,30)
\qbezier(0,50)(25,60)(25,50)
\put(25,50){\vector(0,-1){20}}

\put(5,75){\tiny{$_{m}$}}
\put(5,5){\tiny{$_{m}$}}
\put(-30,38){\tiny{$_{m+n}$}}
\put(14,60){\tiny{$_{n}$}}
\end{picture})\simeq C(\input{decomp-II-proof-1})\left\langle N-m-n\right\rangle. \vspace{2.5pc}
\] 
By Corollary \ref{contract-expand}, \vspace{-4pc}
\[
C(\input{decomp-II-proof-1})\simeq C(\input{decomp-II-proof-2}). \vspace{2.5pc}
\]
By Proposition \ref{decomp-II},
\[
C(\input{decomp-II-proof-2})\simeq C(\input{decomp-II-proof-3})\qb{N-m}{n}. \vspace{2.5pc}
\]  
Using \cite[Lemma 5.15]{Wu-color} again, we get \vspace{-2pc}
\[
C(\input{decomp-II-proof-3}\simeq C(\setlength{\unitlength}{.75pt}
\begin{picture}(40,80)(-20,40)
\put(0,0){\vector(0,1){80}}
\put(5,75){\tiny{$_{m}$}}
\end{picture})\left\langle N-m\right\rangle. \vspace{2pc}
\]
And the proposition follows.
\end{proof}

\section{Morphisms Induced by Local Changes of MOY Graphs}\label{sec-morph}

In this section, we introduce morphisms of matrix factorizations induced by certain local changes of MOY graphs. These morphisms are building blocks of the more complex morphisms in the homological versions of MOY Relations (5)-(7) and in the chain complexes of link diagrams. 

\subsection{Terminology} Most morphisms defined in the rest of this paper are defined only up to homotopy and scaling by a non-zero scalar. To simplify our exposition, we introduce the following notations.

\begin{definition}
If $M,M'$ are matrix factorizations of the same potential over a graded commutative unital $\C$-algebra and $f,g:M\rightarrow M'$ are morphisms of matrix factorizations, we write $f \approx g$ if $\exists ~c\in \C\setminus \{0\}$ such that $f \simeq c \cdot g$.
\end{definition}

\begin{definition}
Let $\Gamma_1,\Gamma_2$ be two colored MOY graphs with a one-to-one correspondence $F$ between their end points such that
\begin{itemize}
	\item every exit corresponds to an exit, and every entrance corresponds to an entrance,
	\item edges adjacent to corresponding end points have the same color.
\end{itemize}
Mark $\Gamma_1,\Gamma_2$ so that every pair of corresponding end points are assigned the same alphabet, and alphabets associated to internal marked points are pairwise disjoint. Let $\mathbb{X}_1,\mathbb{X}_2,\dots,\mathbb{X}_n$ be the alphabets assigned to the end points of $\Gamma_1,\Gamma_2$.

Define $\Hom_F(C(\Gamma_1),C(\Gamma_2)) := \Hom_{\Sym(\mathbb{X}_1|\mathbb{X}_2|\cdots|\mathbb{X}_n)}(C(\Gamma_1),C(\Gamma_2))$, which is a $\zed_2$-graded chain complex, where the $\zed_2$-grading is induced by the $\zed_2$-gradings of $C(\Gamma_1), ~C(\Gamma_2)$. The quantum gradings of $C(\Gamma_1), ~C(\Gamma_2)$ induce a quantum pregrading on $\Hom_F(C(\Gamma_1),C(\Gamma_2))$.

Denote by $\Hom_{\HMF,F}(C(\Gamma_1),C(\Gamma_2))$ the homology of $\Hom_F(C(\Gamma_1),C(\Gamma_2))$, i.e. the module of homotopy classes of morphisms from $C(\Gamma_1)$ to $C(\Gamma_2)$. It inherits the $\zed_2$-grading from $\Hom_F(C(\Gamma_1),C(\Gamma_2))$. And the quantum pregrading of $\Hom_F(C(\Gamma_1),C(\Gamma_2))$ induces a quantum grading on $\Hom_{\HMF,F}(C(\Gamma_1),C(\Gamma_2))$. (See Lemmas \ref{hom-all-gradings-homo-finite} and \ref{MOY-object-of-hmf}.)

We drop $F$ from the above notations if it is clear from the context.
\end{definition}

\subsection{Bouquet moves}\label{subsec-bouquet} We call the change 
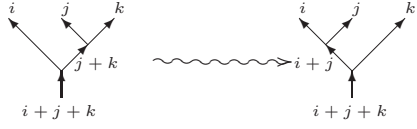
\vspace{-2pc}
\[
\xymatrix{
\setlength{\unitlength}{1pt}
\begin{picture}(50,50)(-80,20)

\put(-60,10){\vector(0,1){10}}

\put(-60,20){\vector(-1,1){20}}

\put(-60,20){\vector(1,1){10}}

\put(-50,30){\vector(-1,1){10}}

\put(-50,30){\vector(1,1){10}}

\put(-75,3){\tiny{$i+j+k$}}

\put(-55,21){\tiny{$j+k$}}

\put(-80,42){\tiny{$i$}}

\put(-60,42){\tiny{$j$}}

\put(-40,42){\tiny{$k$}}

\end{picture} \ar@{~>}[rr]<-5ex> && \setlength{\unitlength}{1pt}
\begin{picture}(50,50)(40,20)

\put(60,10){\vector(0,1){10}}

\put(60,20){\vector(1,1){20}}

\put(60,20){\vector(-1,1){10}}

\put(50,30){\vector(1,1){10}}

\put(50,30){\vector(-1,1){10}}

\put(45,3){\tiny{$i+j+k$}}

\put(38,21){\tiny{$i+j$}}

\put(80,42){\tiny{$k$}}

\put(60,42){\tiny{$j$}}

\put(40,42){\tiny{$i$}}

\end{picture}} \vspace{2pc}
\]
a bouquet move. From Corollary \ref{contract-expand}, we know that there is a homotopy equivalence \vspace{-2pc}
\[
\xymatrix{
C(\setlength{\unitlength}{1pt}
\begin{picture}(50,50)(-80,20)

\put(-60,10){\vector(0,1){10}}

\put(-60,20){\vector(-1,1){20}}

\put(-60,20){\vector(1,1){10}}

\put(-50,30){\vector(-1,1){10}}

\put(-50,30){\vector(1,1){10}}

\put(-75,3){\tiny{$i+j+k$}}

\put(-55,21){\tiny{$j+k$}}

\put(-80,42){\tiny{$i$}}

\put(-60,42){\tiny{$j$}}

\put(-40,42){\tiny{$k$}}

\end{picture}) \ar@<-5ex>[rr]^{\simeq} && C(\setlength{\unitlength}{1pt}
\begin{picture}(50,50)(40,20)

\put(60,10){\vector(0,1){10}}

\put(60,20){\vector(1,1){20}}

\put(60,20){\vector(-1,1){10}}

\put(50,30){\vector(1,1){10}}

\put(50,30){\vector(-1,1){10}}

\put(45,3){\tiny{$i+j+k$}}

\put(38,21){\tiny{$i+j$}}

\put(80,42){\tiny{$k$}}

\put(60,42){\tiny{$j$}}

\put(40,42){\tiny{$i$}}

\end{picture})} \vspace{2pc}
\]
that preserves both gradings. By \cite[Lemma 7.4]{Wu-color}, up to homotopy and scaling, such morphism is unique. Thus, up to homotopy and scaling, a bouquet move induces a unique homotopy equivalence. In this paper, we usually denote such a homotopy equivalence by $h$.

\subsection{Circle creations and annihilations}\label{subsec-circle-creation}

Let $\bigcirc_m$ be a circle colored by $m$. From \cite[Lemma 7.6]{Wu-color}, we know that, as $\zed_2\oplus\zed$-graded vector spaces over $\C$,
\[
\Hom_{\HMF}(C(\bigcirc_m),C(\emptyset)) \cong \Hom_{\HMF}(C(\emptyset),C(\bigcirc_m)) \cong C(\emptyset)\{\qb{N}{m}\}\left\langle m \right\rangle,
\]
where $C(\emptyset)$ is the matrix factorization $\C\rightarrow 0 \rightarrow \C$. In particular, the subspaces of $\Hom_{\HMF}(C(\emptyset),C(\bigcirc_m))$ and $\Hom_{\HMF}(C(\bigcirc_m),C(\emptyset))$ of elements of quantum degree $-m(N-m)$ are $1$-dimensional.

Thus, up to homotopy and scaling,
\begin{enumerate} [(i)]
	\item the circle creation $\emptyset\leadsto \bigcirc_m$ induces a unique homogeneous morphism
	\[
	\iota: C(\emptyset)(\cong \C) \rightarrow C(\bigcirc_m)
	\] 
	of quantum degree $-m(N-m)$ not homotopic to $0$,
  \item the circle annihilation $\bigcirc_m\leadsto \emptyset$ induces a unique homogeneous morphism 
  \[
  \epsilon:C(\bigcirc_m) \rightarrow C(\emptyset) (\cong \C)
  \]
  of quantum degree $-m(N-m)$ not homotopic to $0$.
\end{enumerate}
Note that both of $\iota$ and $\epsilon$ have $\zed_2$-degree $m$. 

\begin{lemma}\cite[Corollary 7.8]{Wu-color}\label{iota-epsilon-composition}
Denote by $\mathfrak{m}(S_{\lambda}(\mathbb{X}))$ the morphism $C(\bigcirc_m)\rightarrow C(\bigcirc_m)$ induced by the multiplication by $S_{\lambda}(\mathbb{X})$. Then, for any $\lambda, \mu \in \Lambda_{m,N-m}$,
\[
\epsilon \circ \mathfrak{m}(S_{\lambda}(\mathbb{X})) \circ \mathfrak{m}(S_{\mu}(\mathbb{X})) \circ \iota \approx
\begin{cases}
\id_{C(\emptyset)} & \text{if } \lambda_j + \mu_{m+1-j} =N-m ~\forall j=1,\dots,m, \\
0 & \text{otherwise.} 
\end{cases}
\]
\end{lemma}

\subsection{Edge splitting and merging}\label{subsec-edge-splitting} We call the following changes edge splitting and merging. \vspace{-2pc}
\[
\xymatrix{
\setlength{\unitlength}{.75pt}
\begin{picture}(55,80)(-20,40)
\put(0,0){\vector(0,1){80}}
\put(5,75){\tiny{$_{m+n}$}}
\end{picture} \ar@{~>}[rr]<-5ex>^{\text{edge splitting}} &&  \input{v-vector-m+n-bubble-slide} \ar@{~>}[ll]<7ex>^{\text{edge merging}}
}\vspace{2pc}
\]
By \cite[Lemma 7.9]{Wu-color}, we have \vspace{-2pc}
\[
\Hom_{\HMF}(C(\setlength{\unitlength}{.75pt}
\begin{picture}(55,80)(-20,40)
\put(0,0){\vector(0,1){80}}
\put(5,75){\tiny{$_{m+n}$}}
\end{picture}),C(\input{v-vector-m+n-bubble-slide})) \cong C(\emptyset) \{q^{(N-m-n)(m+n)}\qb{N}{m+n} \qb{m+n}{m}\},
\] 
\[
\Hom_{\HMF}(C(\input{v-vector-m+n-bubble-slide}),C(\setlength{\unitlength}{.75pt}
\begin{picture}(55,80)(-20,40)
\put(0,0){\vector(0,1){80}}
\put(5,75){\tiny{$_{m+n}$}}
\end{picture})) \cong C(\emptyset) \{q^{(N-m-n)(m+n)}\qb{N}{m+n} \qb{m+n}{m}\}. \vspace{2pc}
\]
In particular, the lowest quantum gradings of the above spaces are both $-mn$, and the subspaces of these spaces of homogeneous elements of quantum grading $-mn$ are both $1$-dimensional. Therefore, up to homotopy and scaling,
\begin{enumerate}[(i)]
	\item the edge splitting induces a unique homogeneous morphism \vspace{-2pc}
	\[
	\phi: C(\setlength{\unitlength}{.75pt}
\begin{picture}(55,80)(-20,40)
\put(0,0){\vector(0,1){80}}
\put(5,75){\tiny{$_{m+n}$}}
\end{picture}) \rightarrow C(\input{v-vector-m+n-bubble-slide}) \vspace{2pc}
	\] 
	of quantum degree $-mn$ not homotopic to $0$,
  \item the edge merging induces a unique homogeneous morphism \vspace{-2pc}
	\[
	\overline{\phi}: C(\input{v-vector-m+n-bubble-slide}) \rightarrow  C(\setlength{\unitlength}{.75pt}
\begin{picture}(55,80)(-20,40)
\put(0,0){\vector(0,1){80}}
\put(5,75){\tiny{$_{m+n}$}}
\end{picture}) \vspace{2pc}
	\] 
	of quantum degree $-mn$ not homotopic to $0$,
\end{enumerate}
Both of $\phi$ and $\overline{\phi}$ have $\zed_2$-grading $0$.

\begin{figure}[ht]

\setlength{\unitlength}{1pt}

\begin{picture}(360,75)(-180,-90)


\put(-67,-45){\tiny{$m+n$}}

\put(-70,-75){\vector(0,1){50}}

\put(-71,-50){\line(1,0){2}}

\put(-67,-30){\small{$\mathbb{X}$}}

\put(-95,-53){\small{$\mathbb{A}\cup\mathbb{B}$}}

\put(-67,-75){\small{$\mathbb{Y}$}}

\put(-75,-90){$\Gamma_0$}


\put(-25,-50){\vector(1,0){50}}

\put(25,-60){\vector(-1,0){50}}

\put(-5,-47){\small{$\phi$}}

\put(-5,-70){\small{$\overline{\phi}$}}


\put(70,-75){\vector(0,1){10}}

\put(70,-35){\vector(0,1){10}}

\qbezier(70,-65)(60,-65)(60,-55)

\qbezier(70,-35)(60,-35)(60,-45)

\put(60,-55){\vector(0,1){10}}

\put(59,-55){\line(1,0){2}}

\qbezier(70,-65)(80,-65)(80,-55)

\qbezier(70,-35)(80,-35)(80,-45)

\put(80,-55){\vector(0,1){10}}

\put(79,-55){\line(1,0){2}}

\put(73,-30){\tiny{$m+n$}}

\put(73,-70){\tiny{$m+n$}}

\put(83,-50){\tiny{$n$}}

\put(51,-50){\tiny{$m$}}

\put(60,-30){\small{$\mathbb{X}$}}

\put(60,-75){\small{$\mathbb{Y}$}}

\put(50,-58){\small{$\mathbb{A}$}}

\put(83,-58){\small{$\mathbb{B}$}}

\put(65,-90){$\Gamma_1$}

\end{picture}

\caption{}\label{edge-splitting}

\end{figure}

\begin{lemma}\cite[Lemma 7.11]{Wu-color}\label{phibar-compose-phi}
With the markings in Figure \ref{edge-splitting}, we have
\[
\overline{\phi} \circ \mathfrak{m}(S_{\lambda}(\mathbb{A})\cdot S_{\mu}(-\mathbb{B})) \circ \phi \approx \left\{%
\begin{array}{ll}
    \id_{\Gamma_0} & \text{if } \lambda_j + \mu_{m+1-j} = n ~\forall j=1,\dots,m, \\ 
    0 & \text{otherwise,}
\end{array}%
\right. 
\]
where $\lambda,\mu\in \Lambda_{m,n}$ and $\mathfrak{m}(S_{\lambda}(\mathbb{A})\cdot S_{\mu}(-\mathbb{B}))$ is the morphism induced by the multiplication of $S_{\lambda}(\mathbb{A})\cdot S_{\mu}(-\mathbb{B})$.
\end{lemma}

\subsection{$\chi$-morphisms}\label{subsec-chi-morphisms}

\begin{figure}[ht]

\setlength{\unitlength}{1pt}

\begin{picture}(360,75)(-180,-15)


\put(-120,0){\vector(1,1){20}}

\put(-100,20){\vector(1,-1){20}}

\put(-100,40){\vector(0,-1){20}}

\put(-100,40){\vector(1,1){20}}

\put(-120,60){\vector(1,-1){20}}

\put(-101,30){\line(1,0){2}}

\put(-132,45){\tiny{$_{m+n-l}$}}

\put(-115,15){\tiny{$_l$}}

\put(-90,45){\tiny{$_m$}}

\put(-90,15){\tiny{$_n$}}

\put(-95,28){\tiny{$_{n-l}$}}

\put(-130,55){\small{$\mathbb{A}$}}

\put(-75,0){\small{$\mathbb{Y}$}}

\put(-113,27){\small{$\mathbb{D}$}}

\put(-130,0){\small{$\mathbb{B}$}}

\put(-75,55){\small{$\mathbb{X}$}}

\put(-102,-15){$\Gamma_0$}


\put(-30,35){\vector(1,0){60}}

\put(30,25){\vector(-1,0){60}}

\put(-3,40){\small{$\chi^0$}}

\put(-3,15){\small{$\chi^1$}}


\put(60,10){\vector(1,1){20}}

\put(60,50){\vector(1,-1){20}}

\put(80,30){\vector(1,0){20}}

\put(100,30){\vector(1,1){20}}

\put(100,30){\vector(1,-1){20}}

\put(68,45){\tiny{$_{m+n-l}$}}

\put(70,15){\tiny{$_{l}$}}

\put(108,45){\tiny{$_m$}}

\put(106,15){\tiny{$_n$}}

\put(81,32){\tiny{$_{m+n}$}}

\put(50,45){\small{$\mathbb{A}$}}

\put(122,10){\small{$\mathbb{Y}$}}

\put(50,10){\small{$\mathbb{B}$}}

\put(122,45){\small{$\mathbb{X}$}}

\put(88,-15){$\Gamma_1$}

\end{picture}

\caption{}\label{general-general-chi-maps-figure}

\end{figure}

\begin{proposition}\cite[Propositions 7.20 and 7.29]{Wu-color}\label{general-general-chi-maps}
Let $\Gamma_0$ and $\Gamma_1$ be the MOY graphs in Figure \ref{general-general-chi-maps-figure}, where $1\leq l\leq n <m+n \leq N$. There exist homogeneous morphisms $\chi^0:C(\Gamma_0) \rightarrow C(\Gamma_1)$ and $\chi^1:C(\Gamma_1) \rightarrow C(\Gamma_0)$ satisfying
\begin{enumerate}[(i)]
	\item Both $\chi^0$ and $\chi^1$ have $\zed_2$-grading $0$ and quantum grading $ml$.
	\item \begin{eqnarray*}
	\chi^1 \circ \chi^0 & \simeq & (\sum_{\lambda\in\Lambda_{l,m}} (-1)^{|\lambda|} S_{\lambda'}(\mathbb{X}) S_{\lambda^c}(\mathbb{B})) \cdot \id_{C(\Gamma_0)}, \\
	\chi^0 \circ \chi^1 & \simeq & (\sum_{\lambda\in\Lambda_{l,m}} (-1)^{|\lambda|} S_{\lambda'}(\mathbb{X}) S_{\lambda^c}(\mathbb{B})) \cdot \id_{C(\Gamma_1)}, \\
	\end{eqnarray*}
	where $\Lambda_{l,m}  \{\mu=(\mu_1\geq\cdots\geq\mu_l) ~|~ \mu_1 \leq m\}$, $\lambda'\in \Lambda_{m,l}$ is the conjugate of $\lambda$, and $\lambda^c$ is the complement of $\lambda$ in $\Lambda_{l,m}$, i.e., if $\lambda=(\lambda_1\geq\cdots\geq\lambda_l)\in \Lambda_{l,m}$, then $\lambda^c = (m-\lambda_l\geq\cdots\geq m-\lambda_1)$.
	\item $\chi^0$ and $\chi^1$ are homotopically non-trivial.
	\item Up to homotopy and scaling, $\chi^0$ (resp. $\chi^1$) is the unique homotopically non-trivial homogeneous morphism of quantum degree $ml$ from $C(\Gamma_0)$ to $C(\Gamma_1)$ (resp. from $C(\Gamma_1)$ to $C(\Gamma_0)$.)
\end{enumerate}
\end{proposition}

The construction of $\chi^0$, $\chi^1$ and the proof of Proposition \ref{general-general-chi-maps} are rather technical and are omitted here. We refer the reader to \cite[Subsections 7.5--7.5]{Wu-color} for the details.

\subsection{Saddle moves}\label{subsec-saddle-move} Next we define the morphism $\eta$ induced by the saddle move in Figure \ref{saddle-move-figure}.

\begin{figure}[ht]

\setlength{\unitlength}{1pt}

\begin{picture}(360,80)(-180,-15)


\put(-5,35){$\eta$}

\put(-25,30){\vector(1,0){50}}


\qbezier(-140,10)(-120,30)(-140,50)

\put(-140,50){\vector(-1,1){0}}

\qbezier(-100,10)(-120,30)(-100,50)

\put(-100,10){\vector(1,-1){0}}

\multiput(-130,30)(4.5,0){5}{\line(1,0){2}}

\put(-150,50){\small{$\mathbb{X}$}}

\put(-140,30){\tiny{$m$}}

\put(-150,10){\small{$\mathbb{A}$}}

\put(-105,30){\tiny{$m$}}

\put(-95,50){\small{$\mathbb{Y}$}}

\put(-95,10){\small{$\mathbb{B}$}}

\put(-125,-15){$\Gamma_0$}


\qbezier(100,10)(120,30)(140,10)

\put(140,10){\vector(1,-1){0}}

\qbezier(100,50)(120,30)(140,50)

\put(100,50){\vector(-1,1){0}}

\put(90,50){\small{$\mathbb{X}$}}

\put(120,45){\tiny{$m$}}

\put(90,10){\small{$\mathbb{A}$}}

\put(145,50){\small{$\mathbb{Y}$}}

\put(120,13){\tiny{$m$}}

\put(145,10){\small{$\mathbb{B}$}}

\put(115,-15){$\Gamma_1$}

\end{picture}

\caption{}\label{saddle-move-figure}

\end{figure}

By \cite[Lemma 7.33]{Wu-color},
\[
\Hom_{HMF}(C(\Gamma_0),C(\Gamma_1)) \cong C(\emptyset) \{\qb{N}{m} q^{2m(N-m)}\} \left\langle m \right\rangle.
\]
In particular, the subspace of $\Hom_{HMF}(C(\Gamma_0),C(\Gamma_1))$ of homogeneous elements of quantum degree $m(N-m)$ is $1$-dimensional. Therefore, up to homotopy and scaling, the saddle move in Figure \ref{saddle-move-figure} induces a unique homogeneous morphism
\[
\eta : C(\Gamma_0) \rightarrow C(\Gamma_1)
\]
of quantum degree $m(N-m)$ not homotopic to $0$. The $\zed_2$-degree of $\eta$ is $m$. 

The following propositions show that, as expected, a pair of canceling $0$- and $1$-handles (or $1$- and $2$-handles) induce the identity morphism. The proofs of these propositions are very technical and omitted here. Please see \cite[Subsections 7.9--7.10]{Wu-color} for their proofs.

\begin{figure}[ht]

\setlength{\unitlength}{1pt}

\begin{picture}(360,70)(-180,-10)
\put(-110,30){\tiny{$m$}}

\put(-100,0){\vector(0,1){60}}

\put(-65,35){$\iota$}

\put(-80,30){\vector(1,0){40}}

\put(-103,-10){$\Gamma$}


\put(-20,30){\tiny{$m$}}

\put(-10,0){\vector(0,1){60}}

\put(12,48){\tiny{$m$}}

\put(15,30){\oval(20,30)}

\put(25,35){\vector(0,1){0}}

\multiput(-10,30)(5,0){3}{\line(1,0){3}}

\put(0,-10){$\Gamma_1$}

\put(45,35){$\eta$}

\put(30,30){\vector(1,0){40}}


\put(90,55){\vector(0,1){5}}

\qbezier(90,55)(90,45)(100,45)

\qbezier(100,45)(110,45)(110,30)

\qbezier(110,30)(110,15)(100,15)

\qbezier(100,15)(90,15)(90,5)

\put(90,0){\line(0,1){5}}

\put(115,30){\tiny{$m$}}

\put(90,-10){$\Gamma$}

\end{picture}

\caption{}\label{creation+saddle+figure}

\end{figure}

\begin{proposition}\cite[Proposition 7.36]{Wu-color}\label{creation+compose+saddle}
Let $\Gamma$ and $\Gamma_1$ be the colored MOY graphs in Figure \ref{creation+saddle+figure}, $\iota:C(\Gamma)\rightarrow C(\Gamma_1)$ the morphism associated to the circle creation and $\eta:C(\Gamma_1)\rightarrow C(\Gamma)$ the morphism associated to the saddle move. Then $\eta\circ \iota \approx \id_{C(\Gamma)}$.
\end{proposition}

\begin{figure}[ht]

\setlength{\unitlength}{1pt}

\begin{picture}(360,70)(-180,-10)


\put(-110,55){\vector(0,1){5}}

\qbezier(-110,55)(-110,45)(-100,45)

\qbezier(-100,45)(-90,45)(-90,30)

\qbezier(-90,30)(-90,15)(-100,15)

\qbezier(-100,15)(-110,15)(-110,5)

\put(-110,0){\line(0,1){5}}

\multiput(-100,15)(0,5){6}{\line(0,1){3}}

\put(-85,30){\tiny{$m$}}

\put(-55,35){$\eta$}

\put(-70,30){\vector(1,0){40}}

\put(-103,-10){$\Gamma$}


\put(-20,30){\tiny{$m$}}

\put(-10,0){\vector(0,1){60}}

\put(2,48){\tiny{$m$}}

\put(5,30){\oval(20,30)}

\put(45,35){$\varepsilon$}

\put(30,30){\vector(1,0){40}}

\put(-5,-10){$\Gamma_1$}


\put(90,30){\tiny{$m$}}

\put(100,0){\vector(0,1){60}}

\put(97,-10){$\Gamma$}

\end{picture}

\caption{}\label{saddle+annihilation+figure}

\end{figure}

\begin{proposition}\cite[Proposition 7.41]{Wu-color}\label{saddle+compose+annihilation}
Let $\Gamma$ and $\Gamma_1$ be the colored MOY graphs in Figure \ref{saddle+annihilation+figure}, $\eta:C(\Gamma)\rightarrow C(\Gamma_1)$ the morphism associated to the saddle move and $\epsilon:C(\Gamma_1)\rightarrow C(\Gamma)$ the morphism associated to circle annihilation. Then $\epsilon \circ \eta \approx \id_{C(\Gamma)}$.
\end{proposition}

\section{Homological MOY Calculus, Part Two}\label{sec-MOY-decomps-part2}

In this section, we sketch the proofs of homological versions of MOY relations (5)-(7). Again, the homological versions of the MOY relations remain true if we reverse the orientation of the MOY graph or the orientation of $\mathbb{R}^2$.

\subsection{MOY Relation (5)} Theorem \ref{decomp-III} below is the homological version of MOY Relation (5). It turns out that Theorem \ref{decomp-III} is not explicitly used in the construction of the colored $\mathfrak{sl}(N)$ link homology or the proof of its invariance.\footnote{The $m=1$ case of Theorem \ref{decomp-III} is established in \cite{KR1} and used there to prove the invariance of the uncolored $\mathfrak{sl}(N)$ link homology under the Reidemeister move II$_b$.} Also the proof of this theorem is somewhat similar to that of Theorem \ref{decomp-IV} in the next subsection. So we omit the proof of Theorem \ref{decomp-III} here and refer the reader to \cite[Section 8]{Wu-color}.

\begin{theorem}\cite[Theorem 8.1]{Wu-color}\label{decomp-III}
Assume $m\leq N-1$. Then \vspace{-1.5pc}
\[
C(\input{decomp-III-1-slide}) \simeq C(\setlength{\unitlength}{.75pt}
\begin{picture}(60,60)(-30,30)

\put(-20,0){\vector(0,1){60}}

\put(20,60){\vector(0,-1){60}}

\put(-25,30){\tiny{$_1$}}

\put(22,30){\tiny{$_m$}}
\end{picture}) \oplus C(\setlength{\unitlength}{.75pt}
\begin{picture}(60,60)(100,30)

\put(110,0){\vector(1,1){20}}

\put(130,20){\vector(1,-1){20}}

\put(130,40){\vector(0,-1){20}}

\put(130,40){\vector(-1,1){20}}

\put(150,60){\vector(-1,-1){20}}

\put(105,0){\tiny{$_1$}}

\put(105,55){\tiny{$_1$}}

\put(152,0){\tiny{$_m$}}

\put(152,55){\tiny{$_m$}}

\put(132,30){\tiny{$_{m-1}$}}

\end{picture})\{[N-m-1]\} \left\langle 1 \right\rangle. \vspace{2pc}
\]
\end{theorem}

\subsection{MOY Relation (6)} In this subsection, we sketch the proof of Theorem \ref{decomp-IV}, the homological version of MOY Relation (6). See \cite[Section 9]{Wu-color} for full details.

\begin{theorem}\cite[Theorem 9.1]{Wu-color}\label{decomp-IV}
Assume $l,m,n$ are integers satisfying $0\leq n \leq m \leq N$ and $0\leq l, m+l-1 \leq N$. Then \vspace{-2pc}
\[
C(\input{decomp-IV-1-slide}) \simeq C(\setlength{\unitlength}{.75pt}
\begin{picture}(85,90)(-30,45)

\put(-20,0){\vector(0,1){45}}

\put(-20,45){\vector(0,1){45}}

\put(20,0){\vector(0,1){45}}

\put(20,45){\vector(0,1){45}}

\put(20,45){\vector(-1,0){40}}

\put(-27,20){\tiny{$_1$}}

\put(23,20){\tiny{$_{m+l-1}$}}

\put(-27,65){\tiny{$_l$}}

\put(23,65){\tiny{$_m$}}

\put(-5,38){\tiny{$_{l-1}$}}

\end{picture})\{\qb{m-1}{n}\} \oplus C(\setlength{\unitlength}{.75pt}
\begin{picture}(60,90)(110,45)

\put(110,0){\vector(2,3){20}}

\put(150,0){\vector(-2,3){20}}

\put(130,30){\vector(0,1){30}}

\put(130,60){\vector(-2,3){20}}

\put(130,60){\vector(2,3){20}}

\put(117,20){\tiny{$_1$}}

\put(140,20){\tiny{$_{m+l-1}$}}

\put(117,65){\tiny{$_l$}}

\put(140,65){\tiny{$_m$}}

\put(133,42){\tiny{$_{m+l}$}}
\end{picture})\{\qb{m-1}{n-1}\}. \vspace{2pc}
\]
\end{theorem}

\begin{proof}[Sketch of Proof]
First, define morphisms \vspace{-2pc}
\[
f: C(\setlength{\unitlength}{.75pt}
\begin{picture}(85,90)(-30,45)

\put(-20,0){\vector(0,1){45}}

\put(-20,45){\vector(0,1){45}}

\put(20,0){\vector(0,1){45}}

\put(20,45){\vector(0,1){45}}

\put(20,45){\vector(-1,0){40}}

\put(-27,20){\tiny{$_1$}}

\put(23,20){\tiny{$_{m+l-1}$}}

\put(-27,65){\tiny{$_l$}}

\put(23,65){\tiny{$_m$}}

\put(-5,38){\tiny{$_{l-1}$}}

\end{picture}) \rightarrow C(\input{decomp-IV-1-slide}) \vspace{2pc}
\]
and\vspace{-2pc}
\[
g: C(\input{decomp-IV-1-slide}) \rightarrow C(\setlength{\unitlength}{.75pt}
\begin{picture}(85,90)(-30,45)

\put(-20,0){\vector(0,1){45}}

\put(-20,45){\vector(0,1){45}}

\put(20,0){\vector(0,1){45}}

\put(20,45){\vector(0,1){45}}

\put(20,45){\vector(-1,0){40}}

\put(-27,20){\tiny{$_1$}}

\put(23,20){\tiny{$_{m+l-1}$}}

\put(-27,65){\tiny{$_l$}}

\put(23,65){\tiny{$_m$}}

\put(-5,38){\tiny{$_{l-1}$}}

\end{picture}) \vspace{2.5pc}
\]
by the compositions of the morphisms in the diagram in Figure \ref{decomposition-IV-f-g-def}, where $\phi$ and $\overline{\phi}$ are the morphisms associated to the apparent edge splitting and merging, $h_0$ and $h_1$ are the homotopy equivalence induced by the apparent bouquet moves, $\chi^0$ and $\chi^1$ are the $\chi$-morphisms from the change of the left half of MOY graph. (See Section \ref{sec-morph} for the definitions of these morphisms.) It is easy to check that $f$ and $g$ are both homogeneous morphisms of quantum degree $-n(m-n-1)$ and $\zed_2$-degree $0$.

\begin{figure}[ht]

\setlength{\unitlength}{1pt}

\begin{picture}(360,255)(-180,-15)


\put(-110,150){\vector(0,1){25}}

\put(-110,215){\vector(0,1){25}}

\put(-110,175){\vector(0,1){40}}

\put(-110,215){\vector(1,0){40}}

\put(-70,175){\vector(0,1){40}}

\put(-70,175){\vector(-1,0){40}}

\put(-70,157){\vector(0,1){25}}

\put(-70,215){\vector(0,1){25}}

\put(-90,214){\line(0,1){2}}

\put(-90,174){\line(0,1){2}}

\put(-111,190){\line(1,0){2}}

\put(-71,190){\line(1,0){2}}

\put(-117,163){\tiny{$1$}}

\put(-117,222){\tiny{$l$}}

\put(-127,198){\tiny{$l+n$}}

\put(-67,163){\tiny{$m+l-1$}}

\put(-67,222){\tiny{$m$}}

\put(-67,198){\tiny{$m-n$}}

\put(-103,178){\tiny{$l+n-1$}}

\put(-92,208){\tiny{$n$}}

\put(-67,232){\small{$\mathbb{X}$}}

\put(-67,187){\small{$\mathbb{Y}$}}

\put(-67,150){\small{$\mathbb{W}$}}

\put(-120,232){\small{$\mathbb{T}$}}

\put(-120,187){\small{$\mathbb{S}$}}

\put(-125,150){\small{$\{r\}$}}

\put(-93,217){\small{$\mathbb{A}$}}

\put(-93,165){\small{$\mathbb{B}$}}

\put(-100,140){\vector(-1,-1){30}}

\put(-125,125){\small{$\chi^1$}}

\put(-120,110){\vector(1,1){30}}

\put(-100,125){\small{$\chi^0$}}

\put(-30,190){\vector(1,0){60}}

\put(30,200){\vector(-1,0){60}}

\put(-3,204){\small{$f$}}

\put(-3,180){\small{$g$}}


\put(70,150){\vector(0,1){45}}

\put(70,195){\vector(0,1){45}}

\put(110,150){\vector(0,1){45}}

\put(110,195){\vector(0,1){45}}

\put(110,195){\vector(-1,0){40}}

\put(63,170){\tiny{$1$}}

\put(113,170){\tiny{$m+l-1$}}

\put(63,215){\tiny{$l$}}

\put(113,215){\tiny{$m$}}

\put(85,188){\tiny{$l-1$}}

\put(113,232){\small{$\mathbb{X}$}}

\put(113,150){\small{$\mathbb{W}$}}

\put(60,232){\small{$\mathbb{T}$}}

\put(55,150){\small{$\{r\}$}}

\put(90,140){\vector(1,-1){30}}

\put(125,125){\small{$\overline{\phi}$}}

\put(130,110){\vector(-1,1){30}}

\put(90,125){\small{$\phi$}}


\put(-160,0){\vector(0,1){45}}

\put(-160,45){\vector(0,1){45}}

\put(-130,45){\vector(-1,0){30}}

\put(-130,45){\vector(2,1){30}}

\put(-100,30){\vector(-2,1){30}}

\put(-100,0){\vector(0,1){30}}

\put(-100,30){\vector(0,1){30}}

\put(-100,60){\vector(0,1){30}}

\put(-101,45){\line(1,0){2}}

\put(-115,51.5){\line(0,1){2}}

\put(-115,36.5){\line(0,1){2}}

\put(-167,13){\tiny{$1$}}

\put(-167,72){\tiny{$l$}}

\put(-97,13){\tiny{$m+l-1$}}

\put(-97,72){\tiny{$m$}}

\put(-97,48){\tiny{$m-n$}}

\put(-137,30){\tiny{$l+n-1$}}

\put(-127,51){\tiny{$n$}}

\put(-150,38){\tiny{$l-1$}}

\put(-97,82){\small{$\mathbb{X}$}}

\put(-97,37){\small{$\mathbb{Y}$}}

\put(-97,0){\small{$\mathbb{W}$}}

\put(-170,82){\small{$\mathbb{T}$}}

\put(-175,0){\small{$\{r\}$}}

\put(-115,57){\small{$\mathbb{A}$}}

\put(-115,40){\small{$\mathbb{B}$}}

\put(-40,40){\vector(1,0){80}}

\put(40,50){\vector(-1,0){80}}

\put(-3,54){\small{$h_0$}}

\put(-3,30){\small{$h_1$}}


\put(80,0){\vector(0,1){30}}

\put(80,30){\vector(0,1){60}}

\put(140,30){\vector(-1,0){60}}

\put(140,75){\vector(0,1){15}}

\put(140,30){\vector(0,1){15}}

\put(130,50){\vector(0,1){20}}

\put(150,50){\vector(0,1){20}}

\put(140,0){\vector(0,1){30}}

\qbezier(140,75)(130,75)(130,70)

\qbezier(140,75)(150,75)(150,70)

\qbezier(140,45)(130,45)(130,50)

\qbezier(140,45)(150,45)(150,50)

\put(129,58){\line(1,0){2}}

\put(149,58){\line(1,0){2}}

\put(73,13){\tiny{$1$}}

\put(73,72){\tiny{$l$}}

\put(153,62){\tiny{$m-n$}}

\put(120,62){\tiny{$n$}}

\put(143,13){\tiny{$m+l-1$}}

\put(130,80){\tiny{$m$}}

\put(143,35){\tiny{$m$}}

\put(105,23){\tiny{$l-1$}}

\put(143,82){\small{$\mathbb{X}$}}

\put(143,0){\small{$\mathbb{W}$}}

\put(153,52){\small{$\mathbb{Y}$}}

\put(70,82){\small{$\mathbb{T}$}}

\put(65,0){\small{$\{r\}$}}

\put(120,52){\small{$\mathbb{A}$}}

\end{picture}

\caption{}\label{decomposition-IV-f-g-def}

\end{figure}

Let $\Lambda=\Lambda_{n,m-n-1} =\{\lambda=(\lambda_1\geq\cdots\geq\lambda_n)~|~\lambda_1\leq m-n-1\}$. For $\lambda=(\lambda_1\geq\cdots\geq\lambda_n) \in \Lambda$, define $\lambda^c=(\lambda^c_1\geq\cdots\geq\lambda^c_n) \in \Lambda$ by $\lambda^c_j =m-n-1-\lambda_{n+1-j}$ for $j=1,\dots,n$. For $\lambda \in \Lambda$, define $f_{\lambda} = \mathfrak{m}(S_{\lambda}(\mathbb{A})) \circ f$ and $g_{\lambda} = g \circ \mathfrak{m}(S_{\lambda^c}(-\mathbb{Y}))$, where $\mathfrak{m}(\ast)$ is the endomorphism given by the multiplication by $\ast$.

Note that \vspace{-2pc}
\[
C(\setlength{\unitlength}{.75pt}
\begin{picture}(85,90)(-30,45)

\put(-20,0){\vector(0,1){45}}

\put(-20,45){\vector(0,1){45}}

\put(20,0){\vector(0,1){45}}

\put(20,45){\vector(0,1){45}}

\put(20,45){\vector(-1,0){40}}

\put(-27,20){\tiny{$_1$}}

\put(23,20){\tiny{$_{m+l-1}$}}

\put(-27,65){\tiny{$_l$}}

\put(23,65){\tiny{$_m$}}

\put(-5,38){\tiny{$_{l-1}$}}

\end{picture})\{\qb{m-1}{n}\} = \bigoplus_{\lambda \in \Lambda} C(\Gamma_0)\{q^{2|\lambda|-n(m-n-1)}\}. \vspace{2pc}
\]
Moreover, \vspace{-2pc}
\[
f_\lambda: C(\setlength{\unitlength}{.75pt}
\begin{picture}(85,90)(-30,45)

\put(-20,0){\vector(0,1){45}}

\put(-20,45){\vector(0,1){45}}

\put(20,0){\vector(0,1){45}}

\put(20,45){\vector(0,1){45}}

\put(20,45){\vector(-1,0){40}}

\put(-27,20){\tiny{$_1$}}

\put(23,20){\tiny{$_{m+l-1}$}}

\put(-27,65){\tiny{$_l$}}

\put(23,65){\tiny{$_m$}}

\put(-5,38){\tiny{$_{l-1}$}}

\end{picture})\{q^{2|\lambda|-n(m-n-1)}\} \rightarrow C(\input{decomp-IV-1-slide}) \vspace{2pc}
\]
and\vspace{-2pc}
\[
g_\lambda: C(\input{decomp-IV-1-slide}) \rightarrow C(\setlength{\unitlength}{.75pt}
\begin{picture}(85,90)(-30,45)

\put(-20,0){\vector(0,1){45}}

\put(-20,45){\vector(0,1){45}}

\put(20,0){\vector(0,1){45}}

\put(20,45){\vector(0,1){45}}

\put(20,45){\vector(-1,0){40}}

\put(-27,20){\tiny{$_1$}}

\put(23,20){\tiny{$_{m+l-1}$}}

\put(-27,65){\tiny{$_l$}}

\put(23,65){\tiny{$_m$}}

\put(-5,38){\tiny{$_{l-1}$}}

\end{picture})\{q^{2|\lambda|-n(m-n-1)}\} \vspace{2.5pc}
\]
are homogeneous morphisms preserving both gradings. 

Define \vspace{-2pc}
\[
F = \sum_{\lambda\in\Lambda} f_{\lambda}: C(\setlength{\unitlength}{.75pt}
\begin{picture}(85,90)(-30,45)

\put(-20,0){\vector(0,1){45}}

\put(-20,45){\vector(0,1){45}}

\put(20,0){\vector(0,1){45}}

\put(20,45){\vector(0,1){45}}

\put(20,45){\vector(-1,0){40}}

\put(-27,20){\tiny{$_1$}}

\put(23,20){\tiny{$_{m+l-1}$}}

\put(-27,65){\tiny{$_l$}}

\put(23,65){\tiny{$_m$}}

\put(-5,38){\tiny{$_{l-1}$}}

\end{picture})\{\qb{m-1}{n}\} \rightarrow C(\input{decomp-IV-1-slide}) \vspace{2pc}
\]
and \vspace{-2pc}
\[
G = \sum_{\lambda\in\Lambda} g_{\lambda}:C(\input{decomp-IV-1-slide}) \rightarrow C(\setlength{\unitlength}{.75pt}
\begin{picture}(85,90)(-30,45)

\put(-20,0){\vector(0,1){45}}

\put(-20,45){\vector(0,1){45}}

\put(20,0){\vector(0,1){45}}

\put(20,45){\vector(0,1){45}}

\put(20,45){\vector(-1,0){40}}

\put(-27,20){\tiny{$_1$}}

\put(23,20){\tiny{$_{m+l-1}$}}

\put(-27,65){\tiny{$_l$}}

\put(23,65){\tiny{$_m$}}

\put(-5,38){\tiny{$_{l-1}$}}

\end{picture})\{\qb{m-1}{n}\}. \vspace{2pc}
\] 
Then $F$ and $G$ are homogeneous morphisms preserving both gradings. By \cite[Lemma 9.8]{Wu-color}, \vspace{-2pc}
\[
G \circ F : C(\setlength{\unitlength}{.75pt}
\begin{picture}(85,90)(-30,45)

\put(-20,0){\vector(0,1){45}}

\put(-20,45){\vector(0,1){45}}

\put(20,0){\vector(0,1){45}}

\put(20,45){\vector(0,1){45}}

\put(20,45){\vector(-1,0){40}}

\put(-27,20){\tiny{$_1$}}

\put(23,20){\tiny{$_{m+l-1}$}}

\put(-27,65){\tiny{$_l$}}

\put(23,65){\tiny{$_m$}}

\put(-5,38){\tiny{$_{l-1}$}}

\end{picture})\{\qb{m-1}{n}\} \rightarrow C(\setlength{\unitlength}{.75pt}
\begin{picture}(85,90)(-30,45)

\put(-20,0){\vector(0,1){45}}

\put(-20,45){\vector(0,1){45}}

\put(20,0){\vector(0,1){45}}

\put(20,45){\vector(0,1){45}}

\put(20,45){\vector(-1,0){40}}

\put(-27,20){\tiny{$_1$}}

\put(23,20){\tiny{$_{m+l-1}$}}

\put(-27,65){\tiny{$_l$}}

\put(23,65){\tiny{$_m$}}

\put(-5,38){\tiny{$_{l-1}$}}

\end{picture})\{\qb{m-1}{n}\} \vspace{2pc}
\]
is an upper triangular matrix with identity along the main diagonal.

\begin{figure}[ht]

\setlength{\unitlength}{1pt}

\begin{picture}(360,255)(-180,-15)


\put(-110,150){\vector(0,1){25}}

\put(-110,215){\vector(0,1){25}}

\put(-110,175){\vector(0,1){40}}

\put(-110,215){\vector(1,0){40}}

\put(-70,175){\vector(0,1){40}}

\put(-70,175){\vector(-1,0){40}}

\put(-70,157){\vector(0,1){25}}

\put(-70,215){\vector(0,1){25}}

\put(-90,214){\line(0,1){2}}

\put(-90,174){\line(0,1){2}}

\put(-111,190){\line(1,0){2}}

\put(-71,190){\line(1,0){2}}

\put(-117,163){\tiny{$1$}}

\put(-117,222){\tiny{$l$}}

\put(-127,198){\tiny{$l+n$}}

\put(-67,163){\tiny{$m+l-1$}}

\put(-67,222){\tiny{$m$}}

\put(-67,198){\tiny{$m-n$}}

\put(-103,178){\tiny{$l+n-1$}}

\put(-92,208){\tiny{$n$}}

\put(-67,232){\small{$\mathbb{X}$}}

\put(-67,187){\small{$\mathbb{Y}$}}

\put(-67,150){\small{$\mathbb{W}$}}

\put(-120,232){\small{$\mathbb{T}$}}

\put(-120,187){\small{$\mathbb{S}$}}

\put(-125,150){\small{$\{r\}$}}

\put(-93,217){\small{$\mathbb{A}$}}

\put(-93,165){\small{$\mathbb{B}$}}

\put(-100,140){\vector(-1,-1){30}}

\put(-125,125){\small{$\chi^0$}}

\put(-120,110){\vector(1,1){30}}

\put(-100,125){\small{$\chi^1$}}

\put(-30,190){\vector(1,0){60}}

\put(30,200){\vector(-1,0){60}}

\put(-3,204){\small{$\alpha$}}

\put(-3,180){\small{$\beta$}}


\put(70,150){\vector(2,3){20}}

\put(110,150){\vector(-2,3){20}}

\put(90,210){\vector(-2,3){20}}

\put(90,210){\vector(2,3){20}}

\put(90,180){\vector(0,1){30}}

\put(73,170){\tiny{$1$}}

\put(103,170){\tiny{$m+l-1$}}

\put(73,215){\tiny{$l$}}

\put(103,215){\tiny{$m$}}

\put(95,188){\tiny{$m+l$}}

\put(113,232){\small{$\mathbb{X}$}}

\put(113,150){\small{$\mathbb{W}$}}

\put(60,232){\small{$\mathbb{T}$}}

\put(55,150){\small{$\{r\}$}}

\put(90,140){\vector(1,-1){30}}

\put(125,125){\small{$\overline{\phi}$}}

\put(130,110){\vector(-1,1){30}}

\put(90,125){\small{$\phi$}}


\put(-160,0){\vector(3,2){30}}

\put(-160,60){\vector(0,1){30}}

\put(-160,60){\vector(1,0){50}}

\put(-110,60){\line(1,0){10}}

\put(-130,20){\vector(0,1){20}}

\put(-130,40){\vector(3,2){30}}

\put(-130,40){\vector(-3,2){30}}

\put(-100,0){\vector(-3,2){30}}

\put(-100,60){\vector(0,1){30}}

\put(-130,59){\line(0,1){2}}

\put(-115,49.5){\line(0,1){2}}

\put(-147,13){\tiny{$1$}}

\put(-167,72){\tiny{$l$}}

\put(-107,13){\tiny{$m+l-1$}}

\put(-97,72){\tiny{$m$}}

\put(-107,48){\tiny{$m-n$}}

\put(-127,30){\tiny{$m+l$}}

\put(-131,53){\tiny{$n$}}

\put(-165,48){\tiny{$n+l$}}

\put(-97,82){\small{$\mathbb{X}$}}

\put(-115,40){\small{$\mathbb{Y}$}}

\put(-97,0){\small{$\mathbb{W}$}}

\put(-170,82){\small{$\mathbb{T}$}}

\put(-175,0){\small{$\{r\}$}}

\put(-132,63){\small{$\mathbb{A}$}}

\put(-40,40){\vector(1,0){80}}

\put(40,50){\vector(-1,0){80}}

\put(-3,54){\small{$h_1$}}

\put(-3,30){\small{$h_0$}}


\put(80,0){\vector(3,2){30}}

\put(80,55){\vector(0,1){35}}

\put(110,20){\vector(0,1){15}}

\put(110,35){\line(3,1){30}}

\put(140,45){\vector(0,1){10}}

\put(110,35){\line(-3,2){30}}

\put(140,75){\vector(0,1){15}}

\put(130,60){\vector(0,1){10}}

\put(150,60){\vector(0,1){10}}

\put(140,0){\vector(-3,2){30}}

\qbezier(140,75)(130,75)(130,70)

\qbezier(140,75)(150,75)(150,70)

\qbezier(140,55)(130,55)(130,60)

\qbezier(140,55)(150,55)(150,60)

\put(129,63){\line(1,0){2}}

\put(149,63){\line(1,0){2}}

\put(93,12){\tiny{$1$}}

\put(73,72){\tiny{$l$}}

\put(153,67){\tiny{$m-n$}}

\put(120,67){\tiny{$n$}}

\put(124,12){\tiny{$m+l-1$}}

\put(130,80){\tiny{$m$}}

\put(140,40){\tiny{$m$}}

\put(113,23){\tiny{$m+l$}}

\put(143,82){\small{$\mathbb{X}$}}

\put(143,0){\small{$\mathbb{W}$}}

\put(153,57){\small{$\mathbb{Y}$}}

\put(70,82){\small{$\mathbb{T}$}}

\put(65,0){\small{$\{r\}$}}

\put(120,57){\small{$\mathbb{A}$}}

\end{picture}

\caption{}\label{decomposition-IV-alpha-beta-def}

\end{figure}

Similarly, using the compositions in the diagram in Figure \ref{decomposition-IV-alpha-beta-def}, we construct homogeneous morphisms \vspace{-2pc}
\[
\vec{\alpha}:C(\setlength{\unitlength}{.75pt}
\begin{picture}(60,90)(110,45)

\put(110,0){\vector(2,3){20}}

\put(150,0){\vector(-2,3){20}}

\put(130,30){\vector(0,1){30}}

\put(130,60){\vector(-2,3){20}}

\put(130,60){\vector(2,3){20}}

\put(117,20){\tiny{$_1$}}

\put(140,20){\tiny{$_{m+l-1}$}}

\put(117,65){\tiny{$_l$}}

\put(140,65){\tiny{$_m$}}

\put(133,42){\tiny{$_{m+l}$}}
\end{picture})\{\qb{m-1}{n-1}\}\rightarrow C(\input{decomp-IV-1-slide}) \vspace{2pc}
\] 
and \vspace{-2pc}
\[
\vec{\beta}:C(\input{decomp-IV-1-slide})\rightarrow C(\setlength{\unitlength}{.75pt}
\begin{picture}(60,90)(110,45)

\put(110,0){\vector(2,3){20}}

\put(150,0){\vector(-2,3){20}}

\put(130,30){\vector(0,1){30}}

\put(130,60){\vector(-2,3){20}}

\put(130,60){\vector(2,3){20}}

\put(117,20){\tiny{$_1$}}

\put(140,20){\tiny{$_{m+l-1}$}}

\put(117,65){\tiny{$_l$}}

\put(140,65){\tiny{$_m$}}

\put(133,42){\tiny{$_{m+l}$}}
\end{picture})\{\qb{m-1}{n-1}\} \vspace{2pc}
\] 
preserving both gradings such that \vspace{-2pc}
\[
\vec{\beta} \circ \vec{\alpha}: C(\setlength{\unitlength}{.75pt}
\begin{picture}(60,90)(110,45)

\put(110,0){\vector(2,3){20}}

\put(150,0){\vector(-2,3){20}}

\put(130,30){\vector(0,1){30}}

\put(130,60){\vector(-2,3){20}}

\put(130,60){\vector(2,3){20}}

\put(117,20){\tiny{$_1$}}

\put(140,20){\tiny{$_{m+l-1}$}}

\put(117,65){\tiny{$_l$}}

\put(140,65){\tiny{$_m$}}

\put(133,42){\tiny{$_{m+l}$}}
\end{picture})\{\qb{m-1}{n-1}\}\rightarrow C(\setlength{\unitlength}{.75pt}
\begin{picture}(60,90)(110,45)

\put(110,0){\vector(2,3){20}}

\put(150,0){\vector(-2,3){20}}

\put(130,30){\vector(0,1){30}}

\put(130,60){\vector(-2,3){20}}

\put(130,60){\vector(2,3){20}}

\put(117,20){\tiny{$_1$}}

\put(140,20){\tiny{$_{m+l-1}$}}

\put(117,65){\tiny{$_l$}}

\put(140,65){\tiny{$_m$}}

\put(133,42){\tiny{$_{m+l}$}}
\end{picture})\{\qb{m-1}{n-1}\} \vspace{2pc}
\]
is an upper triangular matrix with identity along the main diagonal. (See \cite[Lemma 9.12]{Wu-color}.) 

From \cite[Lemma 9.15]{Wu-color} we know that $\vec{\beta} \circ F \circ G \circ \vec{\alpha}$ and $G \circ \vec{\alpha} \circ \vec{\beta} \circ F$ are both homotopically nilpotent. 

Using all the above, it is not hard to modify $F$, $G$, $\vec{\alpha}$ and $\vec{\beta}$ to construct homogeneous morphisms \vspace{-2pc}
\[
\Phi: C(\setlength{\unitlength}{.75pt}
\begin{picture}(85,90)(-30,45)

\put(-20,0){\vector(0,1){45}}

\put(-20,45){\vector(0,1){45}}

\put(20,0){\vector(0,1){45}}

\put(20,45){\vector(0,1){45}}

\put(20,45){\vector(-1,0){40}}

\put(-27,20){\tiny{$_1$}}

\put(23,20){\tiny{$_{m+l-1}$}}

\put(-27,65){\tiny{$_l$}}

\put(23,65){\tiny{$_m$}}

\put(-5,38){\tiny{$_{l-1}$}}

\end{picture})\{\qb{m-1}{n}\} \oplus C(\setlength{\unitlength}{.75pt}
\begin{picture}(60,90)(110,45)

\put(110,0){\vector(2,3){20}}

\put(150,0){\vector(-2,3){20}}

\put(130,30){\vector(0,1){30}}

\put(130,60){\vector(-2,3){20}}

\put(130,60){\vector(2,3){20}}

\put(117,20){\tiny{$_1$}}

\put(140,20){\tiny{$_{m+l-1}$}}

\put(117,65){\tiny{$_l$}}

\put(140,65){\tiny{$_m$}}

\put(133,42){\tiny{$_{m+l}$}}
\end{picture})\{\qb{m-1}{n-1}\} \rightarrow C(\input{decomp-IV-1-slide}) \vspace{2pc}
\]
and \vspace{-2pc}
\[
\Psi: C(\input{decomp-IV-1-slide}) \rightarrow C(\setlength{\unitlength}{.75pt}
\begin{picture}(85,90)(-30,45)

\put(-20,0){\vector(0,1){45}}

\put(-20,45){\vector(0,1){45}}

\put(20,0){\vector(0,1){45}}

\put(20,45){\vector(0,1){45}}

\put(20,45){\vector(-1,0){40}}

\put(-27,20){\tiny{$_1$}}

\put(23,20){\tiny{$_{m+l-1}$}}

\put(-27,65){\tiny{$_l$}}

\put(23,65){\tiny{$_m$}}

\put(-5,38){\tiny{$_{l-1}$}}

\end{picture})\{\qb{m-1}{n}\} \oplus C(\setlength{\unitlength}{.75pt}
\begin{picture}(60,90)(110,45)

\put(110,0){\vector(2,3){20}}

\put(150,0){\vector(-2,3){20}}

\put(130,30){\vector(0,1){30}}

\put(130,60){\vector(-2,3){20}}

\put(130,60){\vector(2,3){20}}

\put(117,20){\tiny{$_1$}}

\put(140,20){\tiny{$_{m+l-1}$}}

\put(117,65){\tiny{$_l$}}

\put(140,65){\tiny{$_m$}}

\put(133,42){\tiny{$_{m+l}$}}
\end{picture})\{\qb{m-1}{n-1}\} \vspace{2pc}
\]
preserving both gradings such that $\Psi \circ \Phi \simeq \id$. (See \cite[Lemma 9.17]{Wu-color}.)

By Lemma \ref{Krull-Schmidt-hmf}, the category $\hmf$ involved here is Krull-Schmidt. Thus, the above implies that \vspace{-2pc}
\[
C(\input{decomp-IV-1-slide}) \simeq C(\setlength{\unitlength}{.75pt}
\begin{picture}(85,90)(-30,45)

\put(-20,0){\vector(0,1){45}}

\put(-20,45){\vector(0,1){45}}

\put(20,0){\vector(0,1){45}}

\put(20,45){\vector(0,1){45}}

\put(20,45){\vector(-1,0){40}}

\put(-27,20){\tiny{$_1$}}

\put(23,20){\tiny{$_{m+l-1}$}}

\put(-27,65){\tiny{$_l$}}

\put(23,65){\tiny{$_m$}}

\put(-5,38){\tiny{$_{l-1}$}}

\end{picture})\{\qb{m-1}{n}\} \oplus C(\setlength{\unitlength}{.75pt}
\begin{picture}(60,90)(110,45)

\put(110,0){\vector(2,3){20}}

\put(150,0){\vector(-2,3){20}}

\put(130,30){\vector(0,1){30}}

\put(130,60){\vector(-2,3){20}}

\put(130,60){\vector(2,3){20}}

\put(117,20){\tiny{$_1$}}

\put(140,20){\tiny{$_{m+l-1}$}}

\put(117,65){\tiny{$_l$}}

\put(140,65){\tiny{$_m$}}

\put(133,42){\tiny{$_{m+l}$}}
\end{picture})\{\qb{m-1}{n-1}\} \oplus M, \vspace{2pc}
\]
where $M$ is a graded matrix factorization. A direct computation shows that the graded dimension of $M$ is $0$. (See \cite[Lemma 9.16]{Wu-color}.) So, by Proposition \ref{prop-homology-detects-homotopy}, $M\simeq 0$. This proves the Theorem \ref{decomp-IV}.
\end{proof}

\subsection{MOY Relation (7)} Theorem \ref{decomp-V} below is the homological version of MOY Relation (7) and a generalization of Theorem \ref{decomp-IV}. It is much harder to explicitly write down the homotopy equivalence in Theorem \ref{decomp-V} than to write down those in homological versions of MOY Relations (1)-(6). Fortunately, we can use the Krull-Schmidt property of matrix factorizations to establish the homotopy equivalence implicitly. 

\begin{theorem}\cite[Theorem 10.1]{Wu-color}\label{decomp-V}
Let $m,n,l$ be non-negative integers satisfying $N\geq n+l,m+l$. Then, for $\max\{m-n,0\}\leq k \leq m+l$, 
\[
\label{decomp-V-1} C(\input{decomp-V-1-slide}) \simeq \bigoplus_{j=\max\{m-n,0\}}^m C(\input{decomp-V-2-slide}) \{\qb{l}{k-j}\}, \vspace{2pc}
\]
where we use the convention $\qb{a}{b}=0$ if $b<0$ or $b>a$.
\end{theorem}

\begin{proof}[Sketch of Proof]
We prove Theorem \ref{decomp-V} by an induction on $k$. For $k=1$, consider the following MOY graph. 
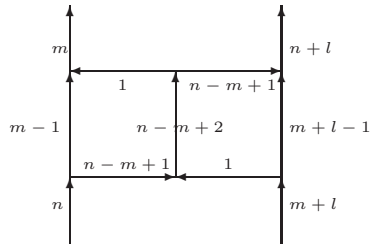
\[
\input{two-squares-m-n-l-1} \vspace{3pc}
\]
Applying Theorem \ref{decomp-IV} to the left and the right rectangles in this MOY graph, we get
\[
C(\input{two-squares-m-n-l-1}) \simeq C(\input{square-m-n-l-right-1-low-slide}) \oplus C(\setlength{\unitlength}{1pt}
\begin{picture}(60,30)(-30,30)

\put(-15,0){\vector(0,1){20}}
\put(-15,20){\vector(0,1){40}}

\put(15,0){\vector(0,1){20}}
\put(15,20){\vector(0,1){40}}

\put(-15,20){\vector(1,0){30}}

\put(-25,5){\tiny{$_{n}$}}
\put(-25,55){\tiny{$_{m}$}}

\put(-8,15){\tiny{$_{n-m}$}}

\put(18,5){\tiny{$_{m+l}$}}
\put(18,55){\tiny{$_{n+l}$}}
\end{picture})\{[m-1][m+l]\}, \vspace{4pc}
\]
\[
C(\input{two-squares-m-n-l-1}) \simeq C(\input{square-m-n-l-right-1-high-slide}) \oplus C(\setlength{\unitlength}{1pt}
\begin{picture}(60,30)(-30,30)

\put(-15,0){\vector(0,1){20}}
\put(-15,20){\vector(0,1){40}}

\put(15,0){\vector(0,1){20}}
\put(15,20){\vector(0,1){40}}

\put(-15,20){\vector(1,0){30}}

\put(-25,5){\tiny{$_{n}$}}
\put(-25,55){\tiny{$_{m}$}}

\put(-8,15){\tiny{$_{n-m}$}}

\put(18,5){\tiny{$_{m+l}$}}
\put(18,55){\tiny{$_{n+l}$}}
\end{picture})\{[m][m+l-1]\}. \vspace{3pc}
\]
Note that $[m]\cdot [m+l-1] - [m-1]\cdot [m+l] =[l]$. By the Krull-Schmidt property (Lemmas \ref{Krull-Schmidt-hmf} and \ref{KS-oplus-cancel},) it follows that  
\[
C(\input{square-m-n-l-right-1-low-slide}) \simeq C(\input{square-m-n-l-right-1-high-slide}) \oplus C(\setlength{\unitlength}{1pt}
\begin{picture}(60,30)(-30,30)

\put(-15,0){\vector(0,1){20}}
\put(-15,20){\vector(0,1){40}}

\put(15,0){\vector(0,1){20}}
\put(15,20){\vector(0,1){40}}

\put(-15,20){\vector(1,0){30}}

\put(-25,5){\tiny{$_{n}$}}
\put(-25,55){\tiny{$_{m}$}}

\put(-8,15){\tiny{$_{n-m}$}}

\put(18,5){\tiny{$_{m+l}$}}
\put(18,55){\tiny{$_{n+l}$}}
\end{picture})\{[l]\}. \vspace{3pc}
\]
So Theorem \ref{decomp-V} is true for $k=1$. (See \cite[Lemma 10.2]{Wu-color} for more details.)

Now assume Theorem \ref{decomp-V} is true for $1,2,\dots,k$. Consider the $k+1$ case of Theorem \ref{decomp-V}. If $k \leq l$, then apply the induction hypothesis to the upper rectangle in the MOY graph 
\[
\input{square-m-n-l-right-k+1-low-tilde-slide}.
\]
If $k>l$, then apply the induction hypothesis to the upper rectangle in the MOY graph 
\[
\input{square-m-n-l-right-k-low-hat-slide}.
\]
In either case, we can use the Krull-Schmidt property (Lemmas \ref{Krull-Schmidt-hmf}, \ref{KS-oplus-cancel} and \ref{yonezawa-lemma},) to deduce that Theorem \ref{decomp-V} is true for $k+1$. (See \cite[Proof of Theorem 10.1]{Wu-color} for more details.)
\end{proof}

\section{Chain Complexes Associated to Knotted MOY Graphs}\label{sec-chain-complex-def}

In this section, we define the chain complex associated to a colored link diagram. In fact, we define the chain complex associated to a knotted MOY graph, which is slightly more general and is useful in our proof of the invariance of the colored $\mathfrak{sl}(N)$ link homology.

\begin{figure}[h]

\setlength{\unitlength}{1pt}

\begin{picture}(360,50)(-180,-30)

\linethickness{.5pt}


\put(-100,-20){\vector(1,1){40}}

\put(-60,-20){\line(-1,1){15}}

\put(-85,5){\vector(-1,1){15}}

\put(-84,-30){$+$}


\put(100,-20){\vector(-1,1){40}}

\put(60,-20){\line(1,1){15}}

\put(85,5){\vector(1,1){15}}

\put(76,-30){$-$}

\end{picture}

\caption{}\label{crossing-sign-figure}

\end{figure}

\subsection{Knotted MOY graphs}

\begin{definition}\label{knotted-MOY-def}
A knotted MOY graph is an immersion of an abstract MOY graph into $\mathbb{R}^2$ such that:
\begin{itemize}
	\item The only singularities are finitely many transversal double points in the interior of edges (that is, away from the vertices.)
	\item At each of these transversal double points, one of the two edges is specified as ``upper", and the other as ``lower".
\end{itemize}
Each transversal double point in a knotted MOY graph is called a crossing. We follow the usual sign convention for crossings given in Figure \ref{crossing-sign-figure}.

If there are crossings in an edge, these crossing divide the edge into several parts. We call each part a segment of the edge.
\end{definition}

Note that colored oriented link/tangle diagrams and (embedded) MOY graphs are knotted MOY graphs.

\begin{definition}\label{knotted-MOY-marking-def}
A marking of a knotted MOY graph $D$ consists the following:
\begin{enumerate}
	\item A finite collection of marked points on $D$ such that
	\begin{itemize}
	\item every segment of every edge of $D$ contains at least one marked point;
	\item all the end points (vertices of valence $1$) are marked;
	\item none of the crossings and internal vertices (vertices of valence at least $2$) is marked.
  \end{itemize}
  \item An assignment of pairwise disjoint alphabets to the marked points such that the alphabet associated to a marked point on an edge of color $m$ has $m$ independent indeterminates. (Recall that an alphabet is a finite collection of homogeneous indeterminates of degree $2$.)
\end{enumerate}
\end{definition}

\subsection{The chain complex associated to a knotted MOY graph}

Given a knotted MOY graph $D$ with a marking, we cut $D$ at the marked points. This produces a collection $\{D_1,\dots,D_m\}$ of simple knotted MOY graphs marked only at their end points. We call each $D_i$ a piece of $D$. It is easy to see that each $D_i$ is one of the following:
\begin{enumerate}[(i)]
	\item an oriented arc from one marked point to another,
	\item a star-shaped neighborhood of a vertex in an (embedded) MOY graph,
	\item a crossing with colored branches.
\end{enumerate}

For a given $D_i$, let $\mathbb{X}_1,\dots,\mathbb{X}_{n_i}$ be the alphabets assigned to all end points of $D_i$, among which $\mathbb{X}_1,\dots,\mathbb{X}_{k_i}$ are assigned to exits and $\mathbb{X}_{k_i+1},\dots,\mathbb{X}_{n_i}$ are assigned to entrances. Let $R_i=\Sym(\mathbb{X}_1|\cdots|\mathbb{X}_{n_i})$ and $w_i= \sum_{j=1}^{k_i} p_{N+1}(\mathbb{X}_j) - \sum_{j=k_i+1}^{n_i} p_{N+1}(\mathbb{X}_j)$. Then the unnormalized and normalized chain complexes $\hat{C}(D_i)$ and $C(D_i)$ associated to $D_i$ are objects of the category $\hch(\hmf_{R_i,w_i})$ differing from each other by a grading shift.

\begin{definition}\label{def-complex-embedded}
If $D_i$ is of type (i) or (ii), then it is an (embedded) MOY graph, and its matrix factorization $C(D_i)$ is an object of $\hmf_{R_i,w_i}$. In this case, the unnormalized and normalized chain complexes $\hat{C}(D_i)$ and $C(D_i)$ associated to $D_i$ are both defined to be the chain complex
\[
0 \rightarrow C(D_i) \rightarrow 0,
\]
where $C(D_i)$ is the matrix factorization associated to $D_i$ and is assigned the homological grading $0$. (The abuse of notations here should not be confusing.)
\end{definition}

For pieces of type (iii) (that is, colored crossings,) the definitions of these chain complexes are much more complex. We defer these definitions to the next subsection.

\begin{definition}\label{complex-knotted-MOY-def}
Let $D$ and $\{D_1,\dots,D_m\}$ be as above. Then
\begin{eqnarray*}
\hat{C}(D) & := & \bigotimes_{i=1}^m \hat{C}(D_i), \\
C(D) & := & \bigotimes_{i=1}^m C(D_i),
\end{eqnarray*}
where the tensor product is done over the common end points. For example, for two pieces $D_{i_1}$ and $D_{i_2}$ of $D$, let $\mathbb{W}_1,\dots,\mathbb{W}_l$ be the alphabets associated to their common end points. Then, in the above tensor product, 
\[
C(D_{i_1}) \otimes C(D_{i_2}) = C(D_{i_1}) \otimes_{\Sym(\mathbb{W}_1|\cdots|\mathbb{W}_l)} C(D_{i_2}).
\]

If $D$ is closed, i.e. has no end points, then $C(D)$ is an object of $\hch(\hmf_{\C,0})$. 

Assume $D$ has end points. Let $\mathbb{E}_1,\dots,\mathbb{E}_n$ be the alphabets assigned to all end points of $D$, among which $\mathbb{E}_1,\dots,\mathbb{E}_k$ are assigned to exits and $\mathbb{E}_{k+1},\dots,\mathbb{E}_n$ are assigned to entrances. Let $R=\Sym(\mathbb{E}_1|\cdots|\mathbb{E}_n)$ and $w= \sum_{i=1}^k p_{N+1}(\mathbb{E}_i) - \sum_{j=k+1}^n p_{N+1}(\mathbb{E}_j)$. In this case, $C(D)$ is an object of $\hch(\hmf_{R,w})$.

Note that, as an object of $\hch(\hmf_{R,w})$, $C(D)$ has a $\zed_2$-grading, a quantum grading and a homological grading.
\end{definition}

\begin{figure}[h]

\setlength{\unitlength}{1pt}

\begin{picture}(360,50)(-180,-30)

\linethickness{.5pt}


\put(-100,-20){\vector(1,1){40}}

\put(-60,-20){\line(-1,1){15}}

\put(-85,5){\vector(-1,1){15}}

\put(-84,-30){$c_{m,n}^+$}

\put(-92,16){\tiny{$_m$}}

\put(-70,16){\tiny{$_n$}}

\put(-105,-20){\tiny{$\mathbb{A}$}}
\put(-105,15){\tiny{$\mathbb{X}$}}

\put(-55,-20){\tiny{$\mathbb{B}$}}
\put(-55,15){\tiny{$\mathbb{Y}$}}


\put(100,-20){\vector(-1,1){40}}

\put(60,-20){\line(1,1){15}}

\put(85,5){\vector(1,1){15}}

\put(76,-30){$c_{m,n}^-$}

\put(68,16){\tiny{$_m$}}

\put(90,16){\tiny{$_n$}}

\put(55,-20){\tiny{$\mathbb{A}$}}
\put(55,15){\tiny{$\mathbb{X}$}}

\put(105,-20){\tiny{$\mathbb{B}$}}
\put(105,15){\tiny{$\mathbb{Y}$}}

\end{picture}

\caption{}\label{colored-crossing-sign-figure}

\end{figure}

\subsection{The chain complex associated to a colored crossing} Let $c_{m,n}^+$ and $c_{m,n}^-$ be the colored crossings with marked end points in Figure \ref{colored-crossing-sign-figure}. In this subsection, we define the chain complexes associated to them, which completes the definition of chain complexes associated to knotted MOY graphs.

\begin{figure}[ht]
\[
\input{square-m-n-k-left}
\]
\caption{}\label{decomp-V-special-1-figure}

\end{figure}

Modeling on Definition \ref{MOY-poly-def}, we would like to resolve $c_{m,n}^\pm$ into MOY graphs of the form in Figure \ref{decomp-V-special-1-figure} and define $\hat{C}(c_{m,n}^\pm)$ and $C(c_{m,n}^\pm)$ to be complexes such that, at each homological grading, the term of these chain complexes is of the form $C(\Gamma_k^L)$. To do this, we need the following lemma.

\begin{lemma}\cite[Lemma 11.11]{Wu-color}\label{colored-crossing-res-HMF}
Let $m,n$ be integers such that $0\leq m,n \leq N$. For $\max\{m-n,0\} \leq j,k \leq m$, 
\begin{eqnarray*}
& & \Hom_\HMF (C(\Gamma_j^L), C(\Gamma_k^L)) \\
& \cong & C(\emptyset) \{\qb{_{n+j+k-m}}{_k}\qb{_{n+j+k-m}}{_j} \qb{_{N+m-n-j-k}}{_{m-k}} \qb{_{N+m-n-j-k}}{_{m-j}} \qb{_N}{_{n+j+k-m}} q^{(m+n)N -n^2-m^2}\}.
\end{eqnarray*}

In particular, the lowest non-vanishing quantum grading of $\Hom_\HMF (C(\Gamma_j^L), C(\Gamma_k^L))$ is $(k-j)^2$. And the subspace of homogeneous elements of quantum degree $(k-j)^2$ of $\Hom_\HMF (C(\Gamma_j^L), C(\Gamma_k^L))$ is $1$-dimensional and have $\zed_2$-grading $0$.
\end{lemma}

The proof of the above lemma is long and technical, we refer the reader to \cite[Section 11]{Wu-color} for the details.

\begin{corollary}\label{complex-colored-crossing-well-defined}
Let $m,n$ be integers such that $0\leq m,n \leq N$. 
\begin{enumerate}[(i)]
	\item For $\max\{m-n,0\}+1 \leq k \leq m$ and up to homotopy and scaling, there is a unique homogeneous morphism
\[
d_k^+ : C(\Gamma_k^L) \rightarrow C(\Gamma_{k-1}^L), 
\]
that has $\zed_2$-degree $0$ and quantum degree $1$, and is not homotopic to $0$. Moreover, $d_{k-1}^+ \circ d_k^+ \simeq 0$.
  \item For $\max\{m-n,0\} \leq k \leq m-1$ and up to homotopy and scaling, there is a unique homogeneous morphism
\[
d_k^- : C(\Gamma_k^L) \rightarrow C(\Gamma_{k+1}^L), 
\]
that has $\zed_2$-degree $0$ and quantum degree $1$, and is not homotopic to $0$. Moreover, $d_{k+1}^- \circ d_k^- \simeq 0$.
\end{enumerate}
\end{corollary}

\begin{proof}
The existence and uniqueness of $d_k^\pm$ follows easily from Lemma \ref{colored-crossing-res-HMF}. To prove that $d_{k-1}^+ \circ d_k^+ \simeq 0$, one just need to note that $d_{k-1}^+ \circ d_k^+$ has quantum degree $2$ and, by Lemma \ref{colored-crossing-res-HMF}, the least non-vanishing quantum degree of $\Hom_\HMF (C(\Gamma_k^L), C(\Gamma_{k-2}^L))$ is $4$. The same argument applies to show that $d_{k+1}^- \circ d_k^- \simeq 0$.
\end{proof}

\begin{definition}\label{complex-colored-crossing-def}
Let $c^\pm_{m,n}$ be the colored crossings in Figure \ref{colored-crossing-sign-figure}, $\hat{R}=\Sym(\mathbb{X}|\mathbb{Y}|\mathbb{A}|\mathbb{B})$ and $w= p_{N+1}(\mathbb{X}) +p_{N+1}(\mathbb{Y}) -p_{N+1}(\mathbb{A}) - p_{N+1}(\mathbb{B})$. 

We define the unnormalized chain complexes $\hat{C}(c^\pm_{m,n})$ first. 

If $m \leq n$, then $\hat{C}(c^+_{m,n})$ is defined to be the object
\[
0\rightarrow C(\Gamma^L_m) \xrightarrow{d^{+}_m} C(\Gamma^L_{m-1})\{q^{-1}\} \xrightarrow{d^{+}_{m-1}} \cdots \xrightarrow{d^{+}_1} C(\Gamma^L_0)\{q^{-m}\} \rightarrow 0
\]
of $\hch(\hmf_{\hat{R},w})$, where the homological grading on $\hat{C}(c^+_{m,n})$ is defined so that the term $C(\Gamma^L_{k})\{q^{-(m-k)}\}$ have homological grading $m-k$.

If $m>n$ , then $\hat{C}(c^+_{m,n})$ is defined to be the object
\[
0\rightarrow C(\Gamma^L_m) \xrightarrow{d^{+}_m} C(\Gamma^L_{m-1})\{q^{-1}\} \xrightarrow{d^{+}_{m-1}} \cdots \xrightarrow{d^{+}_{m-n+1}} C(\Gamma^L_{m-n})\{q^{-n}\} \rightarrow 0
\]
of $\hch(\hmf_{\hat{R},w})$, where the homological grading on $\hat{C}(c^+_{m,n})$ is defined so that the term $C(\Gamma^L_{k})\{q^{-(m-k)}\}$ has homological grading $m-k$.

If $m \leq n$, then $\hat{C}(c^-_{m,n})$ is defined to be the object
\[
0\rightarrow C(\Gamma^L_0)\{q^{m}\} \xrightarrow{d^{-}_0} \cdots \xrightarrow{d^{-}_{m-2}} C(\Gamma^L_{m-1}) \{ q \} \xrightarrow{d^{-}_{m-1}} C(\Gamma^L_m) \rightarrow 0
\]
of $\hch(\hmf_{\hat{R},w})$, where the homological grading on $\hat{C}(c^-_{m,n})$ is defined so that the term $C(\Gamma^L_{k})\{q^{m-k}\}$ has homological grading $k-m$.

If $m>n$ , then $\hat{C}(c^-_{m,n})$ is defined to be the object
\[
0\rightarrow C(\Gamma_{m-n})\{q^{n}\} \xrightarrow{d^{-}_{m-n}} \cdots \xrightarrow{d^{-}_{m-2}} C(\Gamma^L_{m-1})\{q\} \xrightarrow{d^{-}_{m-1}} C(\Gamma^L_{m}) \rightarrow 0
\]
of $\hch(\hmf_{\hat{R},w})$, where the homological grading on $\hat{C}(c^-_{m,n})$ is defined so that the term $C(\Gamma^L_{k})\{q^{m-k}\}$ has homological grading $k-m$.

The normalized chain complex $C(c^\pm_{m,n})$ is defined to be
\begin{eqnarray*}
C(c^+_{m,n}) & = & \begin{cases}
\hat{C}(c^+_{m,m})\left\langle m \right\rangle\|-m\| \{q^{m(N+1-m)}\} & \text{if } m=n, \\
\hat{C}(c^+_{m,n}) & \text{if } m\neq n,
\end{cases} \\
C(c^-_{m,n}) & = & \begin{cases}
\hat{C}(c^-_{m,m})\left\langle m \right\rangle\|m\| \{q^{-m(N+1-m)}\} & \text{if } m=n, \\
\hat{C}(c^-_{m,n}) & \text{if } m\neq n.
\end{cases}
\end{eqnarray*}
(Recall that $\|m\|$ means shifting the homological grading up by $m$. See Definition \ref{categories-of-complexes}.)
\end{definition}

Using Lemma \ref{marking-independence} and the uniqueness of $d_k^\pm$, it is easy to see that $\hat{C}(D)$ and $C(D)$ are independent of the choice of the marking. Please see \cite[Corollary 11.18]{Wu-color}for the proof of the following corollary.

\begin{corollary}\cite[Corollary 11.18]{Wu-color}\label{complex-knotted-MOY-marking-independence}
The isomorphism type of the chain complexes $\hat{C}(D)$ and $C(D)$ associated to a knotted MOY graph $D$ (see Definitions \ref{def-complex-embedded}, \ref{complex-knotted-MOY-def} and \ref{complex-colored-crossing-def}) is independent of the choice of the marking of $D$.
\end{corollary}

By the uniqueness of $d_k^\pm$, it is also easy to see that, for uncolored link diagrams, $C(D)$ is isomorphic to the Khovanov-Rozansky chain complex given in \cite{KR1}. See \cite[Corollary 11.28]{Wu-color} for a proof of the following corollary.

\begin{corollary}\cite[Corollary 11.28]{Wu-color}\label{explicit-differential-1-n-crossings--res}
If $D$ is a link diagram colored completely by $1$, then $C(D)$ is the Khovanov-Rozansky chain complex of $D$ given in \cite{KR1}.
\end{corollary}

\subsection{The Euler Characteristic and the $\zed_2$-grading}\label{subsec-euler-char} In this subsection, we sketch a proof of Theorem \ref{euler-char-main}.

\begin{figure}[ht]
$
\xymatrix{
\input{tri-vertex-s-split} & \text{or} & \input{tri-vertex-m-split}
} 
$
\caption{}\label{tri-vertex-split} 

\end{figure}

\begin{definition}
Let $\Gamma$ be a closed trivalent MOY graph. Replace each edge of $\Gamma$ of color $m$ by $m$ parallel edges colored by $1$ and replace each vertex of $\Gamma$, as depicted in Figure \ref{tri-vertex}, to the corresponding configuration in Figure \ref{tri-vertex-split}, where each strand is an edge colored by $1$. This changes $\Gamma$ into a collection of disjoint embedded circles in the plane. 

The colored rotation number $\mathrm{cr}(\Gamma)$ of $\Gamma$ is defined to be the sum of the rotation numbers of these circles.
\end{definition}

Recall that the homology $H(\Gamma)$ of a MOY graph $\Gamma$ is defined in Definition \ref{homology-MOY-def}, and the graded dimension $\gdim(C(\Gamma))$ is defined to be
\[
\gdim(C(\Gamma)) = \sum_{\ve, i} \tau^\ve q^i H^{\ve,i}(\Gamma) \in \C[\tau,q]/(\tau^2-1),
\]
where $H^{\ve,i}(\Gamma)$ is the subspace of $H(\Gamma)$ of homogeneous elements of $\zed_2$-degree $\ve$ and quantum degree $i$. 

The following lemma implies Theorem \ref{euler-char-main}.

\begin{lemma}\cite[Theorem 14.7]{Wu-color}\label{MOY-gdim-rt}
Let $\Gamma$ be a closed trivalent MOY graph. Then
\begin{enumerate}
	\item $\gdim(C(\Gamma))|_{\tau=1} = \left\langle \Gamma \right\rangle_N$,
	\item $H^{\ve,i}(\Gamma)=0$ if $\ve-\mathrm{cr}(\Gamma)=1$.
\end{enumerate}
\end{lemma}

\begin{proof}[Sketch of Proof]
By Part (9) of Theorem \ref{MOY-poly-skein}, we know that MOY Relations (1)-(7) uniquely determine the $\mathfrak{sl}(N)$ MOY graph polynomial. From Sections \ref{sec-MOY-decomps-part1} and \ref{sec-MOY-decomps-part2}, we know that $C(\Gamma)$ satisfies homological versions of these relation, which implies that $\gdim(C(\Gamma))|_{\tau=1}$ satisfies MOY Relations (1)-(7). So $\gdim(C(\Gamma))|_{\tau=1} = \left\langle \Gamma \right\rangle_N$.

The proof of Part (9) of Theorem \ref{MOY-poly-skein} in \cite[Proof of Theorem 14.7]{Wu-color} is an induction on the highest color of an edge of $\Gamma$. The same inductive argument can be used to show that $H^{\ve,i}(\Gamma)=0$ if $\ve-\mathrm{cr}(\Gamma)=1$. (See \cite[Theorem 14.7]{Wu-color} for more details.)
\end{proof}

\begin{proof}[Proof of Theorem \ref{euler-char-main}]
Theorem \ref{euler-char-main} follows easily from Lemma \ref{MOY-gdim-rt} and Definitions \ref{MOY-poly-def} and \ref{complex-colored-crossing-def}.
\end{proof}

\section{Invariance under Reidemeister Moves}\label{sec-inv-reidemeister}

In this section, we sketch a proof of the invariance of the homotopy type of $C(D)$ under Reidemeister moves, which completes the proof of Theorem \ref{main}. The proof is an induction on the highest color of the link. When the highest color of a link is $1$, then the link is uncolored and the invariance is proved by Khovanov and Rozansky in \cite{KR1}. To complete this inductive argument, we need to establish the invariance of the unnormalized chain complex $\hat{C}(D)$ under the so called fork sliding.

\subsection{Invariance under fork sliding}\label{subsec-inv-fork}

\begin{figure}[ht]
$
\xymatrix{
\input{fork-sliding-general-10} & \input{fork-sliding-general-11} && \input{fork-sliding-general-12} & \input{fork-sliding-general-13} \\
\input{fork-sliding-general-20} & \input{fork-sliding-general-21} && \input{fork-sliding-general-22} & \input{fork-sliding-general-23} \\
\input{fork-sliding-general-30} & \input{fork-sliding-general-31} && \input{fork-sliding-general-32} & \input{fork-sliding-general-33} \\
\input{fork-sliding-general-40} & \input{fork-sliding-general-41} && \input{fork-sliding-general-42} & \input{fork-sliding-general-43}
}
$
\caption{}\label{fork-sliding-invariance-general-fig}

\end{figure}

\begin{proposition}\cite[Theorem 12.1]{Wu-color}\label{fork-sliding-invariance-general}
Let $D_{i,j}^\pm$ be the knotted MOY graphs in Figure \ref{fork-sliding-invariance-general-fig}. Then $\hat{C}(D_{i,0}^+) \simeq \hat{C}(D_{i,1}^+)$ and $\hat{C}(D_{i,0}^-) \simeq \hat{C}(D_{i,1}^-)$. That is, $\hat{C}(D_{i,0}^+)$ (resp. $\hat{C}(D_{i,0}^-)$) is isomorphic in $\hch(\hmf)$ to $\hat{C}(D_{i,1}^+)$ (resp. $\hat{C}(D_{i,1}^-)$).
\end{proposition}

We prove Proposition \ref{fork-sliding-invariance-general} by induction. The hardest part of the proof is to show that Proposition \ref{fork-sliding-invariance-general} is true for some special cases in which either $m=1$ or $l=1$. Once we prove these special cases, the rest of the induction is quite easy. Next, we state these special cases of Proposition \ref{fork-sliding-invariance-general} separately as Lemma \ref{fork-sliding-invariance-special} and then use this lemma to prove Proposition \ref{fork-sliding-invariance-general}.

\begin{lemma}\cite[Proposition 12.2]{Wu-color}\label{fork-sliding-invariance-special}
Let $D_{i,j}^\pm$ be the knotted MOY graphs in Figure \ref{fork-sliding-invariance-general-fig}. 
\begin{enumerate}[(i)]
	\item If $l=1$, then $\hat{C}(D_{i,0}^+) \simeq \hat{C}(D_{i,1}^+)$ and $\hat{C}(D_{i,0}^-) \simeq \hat{C}(D_{i,1}^-)$ for $i=1,4$.
	\item If $m=1$, then $\hat{C}(D_{i,0}^+) \simeq \hat{C}(D_{i,1}^+)$ and $\hat{C}(D_{i,0}^-) \simeq \hat{C}(D_{i,1}^-)$ for $i=2,3$.
\end{enumerate}
\end{lemma}

\begin{proof}[Sketch of Proof of Lemma \ref{fork-sliding-invariance-special}]
Note that the the differential map $d_k^\pm$ is defined implicitly in Definition \ref{complex-colored-crossing-def}. To prove Lemma \ref{fork-sliding-invariance-special}, one need explicit descriptions of these differential maps, which are given in \cite[Theorem 11.26]{Wu-color}. With these explicit descriptions, one can use the Gaussian Elimination Lemma \cite[Lemma 4.2]{Bar-fast} repeatedly to simplify the chain complex $\hat{C}(D_{i,1}^\pm)$ and deduce that it is homotopic to $\hat{C}(D_{i,0}^\pm)$. Please see \cite[Subsections 11.4-5 and Section 12]{Wu-color} for the details.
\end{proof}

\begin{proof}[Proof of Proposition \ref{fork-sliding-invariance-general} (assuming Lemma \ref{fork-sliding-invariance-special} is true)]
Each homotopy equivalence in Proposition \ref{fork-sliding-invariance-general} can be proved by an induction on $m$ or $l$. We only give the details for the proof of 
\begin{equation}\label{fork-sliding-invariance-general-induction-1-+}
\hat{C}(D_{1,0}^+) \simeq \hat{C}(D_{1,1}^+).
\end{equation} 
The proof of the rest of Proposition \ref{fork-sliding-invariance-general} is very similar and left to the reader.

We prove \eqref{fork-sliding-invariance-general-induction-1-+} by an induction on $l$. The $l=1$ case is covered by Part (i) of Lemma \ref{fork-sliding-invariance-special}. Assume that \eqref{fork-sliding-invariance-general-induction-1-+} is true for some $l=k\geq 1$. Consider $l=k+1$. 

\begin{figure}[ht]
$
\xymatrix{
\input{fork-sliding-induction-10} \ar[rr]^{h}_{\cong} && \input{fork-sliding-induction-morph1} \ar[rr]^{\alpha}_{\simeq} && \input{fork-sliding-induction-morph2} \ar[lllld]_{\beta}^{\simeq} \\
\input{fork-sliding-induction-morph3} \ar[rr]^{\xi}_{\simeq} && \input{fork-sliding-induction-morph4} \ar[rr]^{\overline{h}}_{\cong} && \input{fork-sliding-induction-11}
}
$
\caption{}\label{fork-sliding-invariance-induction-fig2}

\end{figure}

Let $\widetilde{D}_{10}^+$ and $\widetilde{D}_{11}^+$ be the first and last knotted MOY graphs in Figure \ref{fork-sliding-invariance-induction-fig2}. By MOY Relation (3) (Proposition \ref{decomp-II}), we have  $\hat{C}(\widetilde{D}_{10}^+) \cong \hat{C}(D_{10}^+)\{[k+1]\}$ and $\hat{C}(\widetilde{D}_{11}^+) \cong \hat{C}(D_{11}^+)\{[k+1]\}$ in $\ch(\hmf)$. Consider the diagram in Figure \ref{fork-sliding-invariance-induction-fig2}. Here, $h$ and $\overline{h}$ are the isomorphisms in $\ch(\hmf)$ induced by the apparent bouquet moves. $\alpha$ is the isomorphism in $\hch(\hmf)$ given by induction hypothesis. $\beta$ is the isomorphism in $\hch(\hmf)$ given by Part (i) of Lemma \ref{fork-sliding-invariance-special}. $\xi$ is also the isomorphism in $\hch(\hmf)$ given by Part (i) of Lemma \ref{fork-sliding-invariance-special}. Altogether, we have
\[
\hat{C}(D_{10}^+)\{[k+1]\} \cong \hat{C}(\widetilde{D}_{10}^+) \simeq \hat{C}(\widetilde{D}_{11}^+) \cong \hat{C}(D_{11}^+)\{[k+1]\}.
\]
So, by Lemmas \ref{Krull-Schmidt-hmf} and \ref{yonezawa-lemma}, $\hat{C}(D_{10}^+) \simeq \hat{C}(D_{11}^+)$ when $l=k+1$.
\end{proof}

\subsection{Invariance under Reidemeister moves} In this subsection, we prove the invariance of the homotopy type of $C(D)$ under Reidemeister moves. In fact, we prove the following slightly more general theorem.

\begin{theorem}\cite[Theorem 13.1]{Wu-color}\label{invariance-reidemeister-all}
Let $D_0$ and $D_1$ be two knotted MOY graphs. Assume that there is a finite sequence of Reidemeister moves that changes $D_0$ into $D_1$. Then $C(D_0) \simeq C(D_1)$, that is, they are isomorphic as objects of $\hch(\hmf)$.
\end{theorem}

The proof of Theorem \ref{invariance-reidemeister-all} is an induction on the highest color of the edges involved in the Reidemeister move. The starting point of our induction is the following theorem by Khovanov and Rozansky \cite{KR1}.

\begin{theorem}\cite[Theorem 2]{KR1}\label{invariance-reidemeister-all-color=1}
Let $D_0$ and $D_1$ be two knotted MOY graphs. Assume that there is a Reidemeister move changing $D_0$ into $D_1$ that involves only edges colored by $1$. Then $C(D_0) \simeq C(D_1)$, that is, they are isomorphic as objects of $\hch(\hmf)$. 
\end{theorem}

\begin{figure}[ht]
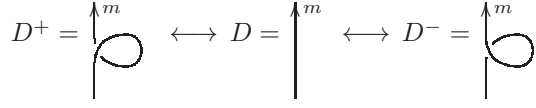

$
D^+= \xygraph{
!{0;/r1pc/:}
[u]
!{\xcapv@(0)=>>{m}}
!{\hover}
!{\hcap}
[ld]!{\xcapv@(0)}
}\, \longleftrightarrow \,
D=\xygraph{
!{0;/r1pc/:}
[u]
!{\xcapv@(0)=>>{m}}
!{\xcapv@(0)}
!{\xcapv@(0)}
}\, \longleftrightarrow \,
D^-=\xygraph{
!{0;/r1pc/:}
[u]
!{\xcapv@(0)=>>{m}}
!{\hunder}
!{\hcap}
[ld]!{\xcapv@(0)}
}
$
\caption{Reidemeister Move I}\label{reidemeisterI}

\end{figure}

\begin{figure}[ht]
$
\xygraph{
!{0;/r2pc/:}
[u]
!{\vtwist=<>{m}|{n}}
!{\vtwistneg}
} \, \longleftrightarrow \,
\xygraph{
!{0;/r2pc/:}
[u]
!{\xcapv[-1]@(0)=<>{m}}
!{\xcapv@(0)}
[ruu]!{\xcapv@(0)=><{n}}
!{\xcapv@(0)}
} \, \longleftrightarrow \,
\xygraph{
!{0;/r2pc/:}
[u]
!{\vtwistneg=<>{n}|{m}}
!{\vtwist}
}
$
\caption{Reidemeister Move II$_a$}\label{reidemeisterII-a}

\end{figure}

\begin{figure}[ht]
$
\xygraph{
!{0;/r2pc/:}
[u]
!{\vcross=>>{n}}
!{\vcrossneg[-1]|{m}}
} \, \longleftrightarrow \,
\xygraph{
!{0;/r2pc/:}
[u]
!{\xcapv[-1]@(0)=>}
!{\xcapv[-1]@(0)>{m}}
[ruu]!{\xcapv@(0)>{n}}
!{\xcapv@(0)=>}
} \, \longleftrightarrow \,
\xygraph{
!{0;/r2pc/:}
[u]
!{\vcrossneg=>|{n}}
!{\vcross<{m}}}
$
\caption{Reidemeister Move II$_b$}\label{reidemeisterII-b}

\end{figure}

\begin{figure}[ht]
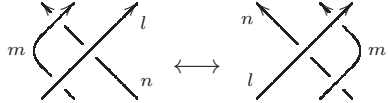

$
\xygraph{
!{0;/r1pc/:}
[uu]
!{\xoverv=<}
!{\xcapv[-1]|{m}}
!{\xoverv}
[uur]!{\xoverv}
[r]!{\zbendh@(0)>{n}}
[uul]!{\sbendh@(0)=>>{l}}
}\, \longleftrightarrow \,
\xygraph{
!{0;/r-1pc/:}
[u]
!{\xoverv}
!{\xcapv[-1]|{m}}
!{\xoverv=>}
[uur]!{\xoverv}
[r]!{\zbendh@(0)=>>{n}}
[uul]!{\sbendh@(0)>{l}}
}
$
\caption{Reidemeister Move III}\label{reidemeisterIII}

\end{figure}

\begin{proof}[Sketch of Proof of Theorem \ref{invariance-reidemeister-all}]
We need to show that the homotopy type of $C(D)$ is invariant under the Reidemeister moves in Figures \ref{reidemeisterI}, \ref{reidemeisterII-a}, \ref{reidemeisterII-b} and \ref{reidemeisterIII}.

The proofs of invariance under the regular Reidemeister moves (II$_a$,II$_b$ and III) are similar. We only give details for the Reidemeister move II$_a$ here.

\begin{figure}[ht]
$
D_0=\xygraph{
!{0;/r2pc/:}
[u]
!{\xcapv[-1]@(0)=<>{m}}
!{\xcapv@(0)}
[ruu]!{\xcapv@(0)=><{n}}
!{\xcapv@(0)}
}  \hspace{2cm}
D_1=\xygraph{
!{0;/r2pc/:}
[u]
!{\vtwist=<>{m}|{n}}
!{\vtwistneg}
}
$
\caption{}\label{reidemeisterII-a-proof}

\end{figure}

Let $D_0$ and $D_1$ be the knotted MOY graphs in Figure \ref{reidemeisterII-a-proof}. We prove by induction on $k$ that $C(D_0) \simeq C(D_1)$ if $1 \leq m,n \leq k$. If $k=1$, then this statement is true because it is a special case of Theorem \ref{invariance-reidemeister-all-color=1}. Assume that this statement is true for some $k\geq 1$. 

\begin{figure}[ht]
$\Gamma_0 =\xygraph{
!{0;/r3pc/:}
[u(1.25)]
!{\xcapv[-0.75]@(0)=<<{m}}
[u(0.25)]!{\xcapv=>|{1}}
[u]!{\xcapv[-1]=<|{m-1}}
!{\xcapv[-0.75]@(0)=<>{m}}
[r][u(2.75)]!{\xcapv[-0.75]@(0)=<<{n}}
[u(0.25)]!{\xcapv=>|{n-1}}
[u]!{\xcapv[-1]=<|{1}}
!{\xcapv[-0.75]@(0)=<>{n}}
}
\Gamma_2 = \xygraph{
!{0;/r1pc/:}
[uuuu]
!{\xcapv[-1]@(0)=<<{m}}
[ur]!{\xcapv@(0)=>>{n}}
[dll]!{\zbendv<{m-1}}
!{\vtwist}
[ru]!{\sbendv[-1]<{n-1}}
[lll]!{\vtwist}
[urr]!{\vtwist}
[l]!{\vtwist=<}
[lu]!{\xcapv@(0)=>}
!{\xcapv[-1]@(0)<{1}}
[rrruu]!{\xcapv@(0)=>}
!{\xcapv@(0)>{1}}
[llu]!{\vtwistneg}
[l]!{\vtwistneg}
[urr]!{\vtwistneg}
[l]!{\vtwistneg}
[lu]!{\sbendv}
[r]!{\zbendv}
[dl]!{\xcapv@(0)=><{n}}
[lu]!{\xcapv[-1]@(0)=<>{m}}
}
\Gamma_1=
\xygraph{
!{0;/r2pc/:}
[uu]
!{\xcapv[-0.5]@(0)=<<{m}}
[u(0.5)]!{\xcapv=>>{1}}
[u]!{\xcapv[-1]=<|{m-1}}
[r][u(1.5)]!{\xcapv[-0.5]@(0)=<<{n}}
[u(0.5)]!{\xcapv=>|{n-1}}
[u]!{\xcapv[-1]=<>{1}}
[l]!{\vtwist}
!{\vtwistneg}
!{\xcapv[-0.5]@(0)=<>{m}}
[ru]!{\xcapv[0.5]@(0)=><{n}}
}
$
\caption{}\label{reidemeisterII-a-bubbles}

\end{figure}

Now consider $k+1$. Assume that $1\leq m,n \leq k+1$ in $D_0$ and $D_1$. Let $\Gamma_0$, $\Gamma_1$ and $\Gamma_2$ be in the knotted MOY graphs in Figure \ref{reidemeisterII-a-bubbles}. Here, in case $m$ or $n=1$, we use the convention that an edge colored by $0$ is an edge that does not exist. By MOY Relation (3) (Proposition \ref{decomp-II}), we know that $\hat{C}(\Gamma_0) \simeq \hat{C}(D_0)\{[m][n]\}$ and $\hat{C}(\Gamma_1) \simeq \hat{C}(D_1)\{[m][n]\}$. Note that $m-1,n-1\leq k$. By the induction hypothesis and the normalization in Definition \ref{complex-colored-crossing-def}, we know that $\hat{C}(\Gamma_0) \simeq \hat{C}(\Gamma_2)$. By the invariance under the fork sliding (Proposition \ref{fork-sliding-invariance-general}), we know that $\hat{C}(\Gamma_1) \simeq \hat{C}(\Gamma_2)$. Thus, $\hat{C}(\Gamma_0) \simeq \hat{C}(\Gamma_1)$. By Lemmas \ref{Krull-Schmidt-hmf} and \ref{yonezawa-lemma}, it follows that $\hat{C}(D_0) \simeq \hat{C}(D_1)$, which, by the normalization in Definition \ref{complex-colored-crossing-def}, is equivalent to $C(D_0) \simeq C(D_1)$. This completes the induction and proves the invariance under Reidemeister move II$_a$. The invariance under Reidemeister moves II$_b$ and III follows similarly.

The invariance under Reidemeister move I is slightly more complex. We need the following lemma.

\begin{figure}[ht]
$
\xymatrix{
\input{twisted-fork-1-n} & \input{twisted-fork-1-n-neg} & \input{fork-1-n} \\
\input{twisted-fork-m-1} & \input{twisted-fork-m-1-neg} & \input{fork-m-1}
} 
$
\caption{}\label{twisted-forks-fig} 

\end{figure}

\begin{lemma}\cite[Proposition 6.1]{Yonezawa3}\label{twisted-forks}
Let $\Gamma_{1,n}$, $\Gamma_{1,n}^\pm$, $\Gamma_{m,1}$ and $\Gamma_{m,1}^\pm$ be the knotted MOY graphs in Figure \ref{twisted-forks-fig}. Then
\begin{eqnarray*}
\hat{C}(\Gamma_{1,n}^+) \simeq \hat{C}(\Gamma_{1,n})\{q^{n}\}, && \hat{C}(\Gamma_{1,n}^-) \simeq \hat{C}(\Gamma_{1,n})\{q^{-n}\}, \\
\hat{C}(\Gamma_{m,1}^+) \simeq \hat{C}(\Gamma_{m,1})\{q^{m}\}, && \hat{C}(\Gamma_{m,1}^-) \simeq \hat{C}(\Gamma_{m,1})\{q^{-m}\},
\end{eqnarray*}
where ``$\simeq$" is the isomorphism in $\hch(\hmf)$.
\end{lemma}

Let $D^+$, $D^-$ and $D$ be the knotted MOY graphs in Figure \ref{reidemeisterI}. We claim:
\begin{eqnarray}
\label{invariance-reidemeister-I-unnormal+}\hat{C}(D^+) & \simeq & \hat{C}(D)\left\langle m \right\rangle\|m\| \{q^{-m(N+1-m)}\}, \\
\label{invariance-reidemeister-I-unnormal-}\hat{C}(D^-) & \simeq & \hat{C}(D)\left\langle m \right\rangle\|-m\| \{q^{m(N+1-m)}\},
\end{eqnarray}
where $\|\ast\|$ means shifting the homological grading up by $\ast$. 

We prove \eqref{invariance-reidemeister-I-unnormal+} by an induction on $m$. The proof of \eqref{invariance-reidemeister-I-unnormal-} is similar and left to the reader. 

\begin{figure}[ht]
\[
\Gamma_1=\xygraph{
!{0;/r2pc/:}
[uu]
!{\xcapv@(0)=>>{m+1}}
!{\hover}
!{\hcap}
[ld]!{\xcapv=><{1}}
[u]!{\xcapv[-1]=<>{m}}
!{\xcapv@(0)=><{m+1}}
}\hspace{.5cm}
\Gamma_2=\xygraph{
!{0;/r1pc/:}
[uuuu]
!{\xcapv@(0)=>>{m+1}}
!{\zbendh}
!{\hcross}
[d]!{\hcross}
!{\hcap=<}
[lld]!{\hcross}
[llu]!{\hover}
[uul]!{\xcapv[2]@(0)}
[dd]!{\xcapv[-2]@(0)|{m}}
[d]!{\sbendh|{1}}
[dl]!{\xcapv@(0)=><{m+1}}
[uuuuurr]!{\hloop[3]=<}
}\hspace{.5cm}
\Gamma_3=\xygraph{
!{0;/r1pc/:}
[uuuu]
!{\xcapv@(0)=>>{m+1}}
!{\zbendh}
!{\hcross}
[d]!{\hcap=>}
[ld]!{\hcross}
[llu]!{\hover}
[uul]!{\xcapv[2]@(0)}
[dd]!{\xcapv[-2]@(0)<{m}}
[d]!{\sbendh|{1}}
[dl]!{\xcapv@(0)=><{m+1}}
[uuuuurr]!{\hcap[3]=<}
}
\]
\[
\Gamma_4=\xygraph{
!{0;/r1.5pc/:}
[uu]
!{\xcapv[0.5]@(0)=>>{m+1}}
[u(0.5)]!{\xcapv@(0)}
[u]!{\sbendv}
[l]!{\vtwist}
!{\xcapv[-1]@(0)=<<{1}}
[ur]!{\hover}
!{\hcap}
[lld]!{\vtwist}
!{\xcapv[-1]@(0)=<<{m}}
!{\zbendv}
[dl]!{\xcapv[0.5]@(0)=>>{m+1}}
}\hspace{.5cm}
\Gamma_5=\xygraph{
!{0;/r2pc/:}
[uu]
!{\xcapv[0.5]@(0)=>>{m+1}}
[u(0.5)]!{\xcapv@(0)}
[u]!{\sbendv}
[l]!{\vtwist}
!{\vtwist}
!{\xcapv@(0)=>>{m}}
!{\zbendv[-1]=<<{1}}
[dl]!{\xcapv[0.5]@(0)=>>{m+1}}
}\hspace{.5cm}
\Gamma_6=\xygraph{
!{0;/r2pc/:}
[uu]
!{\xcapv[0.5]@(0)=>>{m+1}}
[u(0.5)]!{\xcapv@(0)}
[u]!{\sbendv}
[l]!{\vtwist}
!{\xcapv@(0)=>>{m}}
!{\zbendv[-1]=<<{1}}
[dl]!{\xcapv[0.5]@(0)=>>{m+1}}
}\hspace{.5cm}
\Gamma_7=\xygraph{
!{0;/r2pc/:}
[uu]
!{\xcapv[0.5]@(0)=>>{m+1}}
[u(0.5)]!{\xcapv@(0)}
[u]!{\sbendv}
[l]!{\xcapv@(0)=>>{m}}
!{\zbendv[-1]=<<{1}}
[dl]!{\xcapv[0.5]@(0)=>>{m+1}}
}~
\]
\caption{}\label{reidemeisterI-bubble}

\end{figure}

If $m=1$, then \eqref{invariance-reidemeister-I-unnormal+} follows from \cite[Theorem 2]{KR1}. (See Theorem \ref{invariance-reidemeister-all-color=1} above.) Assume that \eqref{invariance-reidemeister-I-unnormal+} is true for some $m\geq 1$. Let us prove \eqref{invariance-reidemeister-I-unnormal+} for $m+1$. 

Consider the knotted MOY graphs $\Gamma_1,\dots,\Gamma_7$ in Figure \ref{reidemeisterI-bubble}. By MOY Relation (3) (Proposition \ref{decomp-II}), we have 
\begin{eqnarray*}
\hat{C}(\Gamma_1) \simeq \hat{C}(D^+)\{[m+1]\} & \text{ and } & \hat{C}(\Gamma_7) \simeq \hat{C}(D)\{[m+1]\}.
\end{eqnarray*}
By Proposition \ref{fork-sliding-invariance-general}, we have $\hat{C}(\Gamma_1) \simeq \hat{C}(\Gamma_2)$. Since \eqref{invariance-reidemeister-I-unnormal+} is true for $1$, we know that $\hat{C}(\Gamma_2) \simeq \hat{C}(\Gamma_3)\left\langle 1 \right\rangle\|1\| \{q^{-N}\}$. From the invariance under Reidemeister moves II and III, one can see that $\hat{C}(\Gamma_3) \simeq \hat{C}(\Gamma_4)$. Since \eqref{invariance-reidemeister-I-unnormal+} is true for $m$, we know that $\hat{C}(\Gamma_4) \simeq \hat{C}(\Gamma_5)\left\langle m \right\rangle\|m\| \{q^{-m(N+1-m)}\}$. By Lemma \ref{twisted-forks}, we know that $\hat{C}(\Gamma_5) \simeq \hat{C}(\Gamma_6)\{q^{m}\}$ and $\hat{C}(\Gamma_6) \simeq \hat{C}(\Gamma_7)\{q^{m}\}$. Putting these together, we get that
\[
\hat{C}(\Gamma_1) \simeq \hat{C}(\Gamma_7)\left\langle m+1 \right\rangle\|m+1\| \{q^{-(m+1)(N-m)}\}.
\]
By Lemmas \ref{Krull-Schmidt-hmf} and \ref{yonezawa-lemma}, it follows that \eqref{invariance-reidemeister-I-unnormal+} is true for $m+1$. This completes the induction and proves \eqref{invariance-reidemeister-I-unnormal+}.

Comparing \eqref{invariance-reidemeister-I-unnormal+} and \eqref{invariance-reidemeister-I-unnormal-} to the normalization in Definition \ref{complex-colored-crossing-def}, we get that 
\begin{eqnarray*}
C(D^+) & \simeq & C(D), \\
C(D^-) & \simeq & C(D).
\end{eqnarray*}
\end{proof}

\end{document}